\documentclass[12pt]{amsart}
\usepackage[english]{babel}
\usepackage{graphicx}
\usepackage{color}
\usepackage{amsmath,amssymb,hyperref,srcltx}
\usepackage[shortlabels]{enumitem}
\usepackage{xifthen}
\DeclareMathOperator{\ord}{ord}

\newcommand{\cn}[1]{\ensuremath{({\mathbb C}^{#1},0)}}

\def\nice{separated}

\newboolean{pedro}
\setboolean{pedro}{true}
\newcommand{\pedro}[1]{\ifthenelse{\boolean{pedro}}{$^{\dagger}$\marginpar{$\dagger$ #1}\color{magenta}
 \setboolean{pedro}{false}}{\color{black}\setboolean{pedro}{true}}}
\newcounter{margin}

\newboolean{javier}
\setboolean{javier}{true}
\newcommand{\javier}{\ifthenelse{\boolean{javier}}{\color{red}
 \setboolean{javier}{false}}{\color{black}\setboolean{javier}{true}}} 

\newtheorem{pro}{Proposition}[section]
\newtheorem{teo}{Theorem}[section]

\newtheorem{cor}{Corollary}[section]

\newtheorem{lem}{Lemma}[section]

\theoremstyle{remark}
\newtheorem{rem}{Remark}[section]

\theoremstyle{definition}
\newtheorem{defi}{Definition}[section]
\theoremstyle{example}
\newtheorem{example}{Example}[section]


\title[Constructive solution of the Moduli Prolem]{Constructive solution of Zariski's Moduli Problem for Plane Branches}

\author{Pedro Fortuny Ayuso and Javier Rib\'{o}n}
\email{fortunypedro@uniovi.es}
\email{jribon@id.uff.br}
\address{Universidad de Oviedo, Spain --- Universidade Federal Fluminense, Brazil}

\date{\today}

\newcommand{\fg}[1]{{\mathsf{Br}({\mathcal E}_{#1})}}
\newcommand{\fga}[1]{{\mathsf{Br}^{\ast}({\mathcal E}_{#1})}}
\newcommand{\fgb}[2]{{{\mathsf{Br}({\mathcal E}_{#1,\geq #2})}}} 
\newcommand{\fgba}[2]{{{\mathsf{Br}^{\ast}({\mathcal E}_{#1,\geq #2})}}}

\newcommand{\touch}{accompany}
\newcommand{\touches}{accompanies}
\newcommand{\Beg}{\boldsymbol{\alpha}}
\newcommand{\DD}[3]{{\mathsf{Ld}_{\beta_{#3}}(#1,#2)}}
\newcommand{\Dd}[2]{{\mathsf{Ld}_{\beta_{#2}}(#1)}}
\newcommand{\dd}[1]{{\mathsf{Ld}(#1)}}
\newcommand{\oDD}[3]{{\overline{\mathsf{Ld}}_{\beta_{#3}}}(#1,#2)}
\newcommand{\oDd}[2]{{\overline{\mathsf{Ld}}_{\beta_{#2}}(#1)}}
\newcommand{\odd}[1]{{\overline{\mathsf{Ld}}(#1)}}
\newcommand{\tDd}[2]{{\mathsf{Ld}^{\Beg}_{\beta_{#2}}}(#1)}
\newcommand{\hatDD}[3]{{{\mathsf{Ld}}_{R= * a^{*}_{\beta_{#3}}}(#1,#2)}}
\newcommand{\low}{{\mathsf{Low}}}
\newcommand{\upper}{{\mathsf{Upp}}}

\usepackage{algorithm}
\usepackage{algpseudocode}
\begin{document}
\hfuzz=2pt
\maketitle

\bibliographystyle{plain}

\begin{abstract}
  In this paper we present an explicit solution of Zariski's moduli problem
  for plane branches. We compute (in an algorithmic way) the set of  K\"{a}hler
  differentials of an irreducible germ of holomorphic plane curve. We show that there is a
  basis of this set whose main elements correspond to dicritical foliations.  Indeed, we
  discuss several concepts of generation for the semimodule of values of K\"{a}hler
  differentials of the curve and provide basis of K\"{a}hler differentials, for every of
  these concepts, whose geometric properties are described.  Moreover, we give an
  algorithmic construction of the bases.
\end{abstract}

\section{Introduction} 

Consider a branch of plane curve, i.e. 
an irreducible germ of holomorphic curve $\Gamma$ in $\cn{2}$. 
We are interested in understanding the set  
\begin{equation*} 
\Lambda_{\Gamma} := \{ \nu_{\Gamma} (\omega) : \omega \in \Omega^{1} \cn{2} \}  
\end{equation*}
of contacts of germs of holomorphic $1$-forms with $\Gamma$ (cf. Definitions \ref{def:omega} and \ref{def:order}). 
This is the main ingredient of
the classification of branches of planar curves by Hefez and Hernandes 
\cite{Hefez-Hernandes-classification}. We can restrict our study to curves that are singular at the origin since
$\Lambda_{\Gamma} = {\mathbb Z}_{\geq 1}$ otherwise, and indeed 
all smooth curves are analytically conjugated. The structure of $\Lambda_{\Gamma}$ 
was studied by Delorme \cite{Delorme:modules} for the particular case where the curve $\Gamma$ has a unique 
Puiseux characteristic exponent. We focus in this paper in the general case.

Up to a holomorphic change of coordinates, 
we can assume that $\Gamma$ has a Puiseux 
parametrization of the form 
\begin{equation}
	\label{equ:param}
	\Gamma (t) = \left( t^{n}, \sum_{\beta \geq \beta_{1}} a_{\beta, \Gamma} t^{\beta} \right)  
\end{equation}
where $n \geq 2$ is the multiplicity of $\Gamma$ at the origin and $\beta_{1}$ is its first Puiseux characteristic 
exponent.

We provide a geometrical interpretation of the ``gap" between the so called semigroup of values of 
$\Gamma$ and $\Lambda_{\Gamma}$. We briefly introduce the main ingredients of our approach in the next subsections.

\subsection{The semigroup} 
There are several ways to describe the topological class of a planar branch $\Gamma$, i.e.
the set of planar branches that are topologically conjugated to $\Gamma$ by a local
homeomorphism defined in a neighborhood of the origin of $\mathbb{C}^{2}$. 
One of the most 
interesting, since it algebrizes the problem, 
is in terms of its semigroup
\begin{equation*} {\mathcal S}_{\Gamma} := \{ \nu_{\Gamma} (f) : f \in {\mathbb C}\{x,y\} \ \mathrm{and} \ f(0,0)=0 \} ,  
\end{equation*}
cf. Definition \ref{def:order}. 
Indeed, it is well known since at least Zariski \cite{Zariski-1932} that two branches are   topologically conjugated  
if and only if their semigroups coincide and also if and only if they have the same
graph of resolutions of singularities
(cf. \cite[Theorem 8.4.21]{brieskorn2012plane}). 
Hence, the topological class of $\Gamma$ is also its equisingularity class.  
Once the topological classification problem is settled, it is natural to consider the
analytic one. 
The (analytic) moduli of $\Gamma$ is the quotient space of its topological class
obtained by identifying two branches if they are analytically 
conjugated. It has been studied by Zariski \cite{Zariski:moduli}, 
Ebey \cite{Ebey:moduli}, Delorme  \cite{Delorme:modules} and many others
(cf. \cite{ayuso2024construction} for a more detailed discussion).

As a semigroup, ${\mathcal S}_{\Gamma}$ has a minimal set of generators
$\overline{\beta}_{0},  \overline{\beta}_{1}, \ldots, \overline{\beta}_{g}$, 
where $\overline{\beta}_{0} = \beta_{0} =n$,
$g$ is the number of Puiseux characteristic exponents of $\Gamma$, the so called genus of $\Gamma$, and 
$\overline{\beta}_{1}, \ldots, \overline{\beta}_{g}$ are directly related to the Puiseux characteristic 
exponents ${\beta}_{1}, \ldots, {\beta}_{g}$ (cf. equation \eqref{equ:rec_beta}).

\subsection{${\mathcal C}$-collections}
Since $\nu_{\Gamma} (f) + \nu_{\Gamma} (\omega) = \nu_{\Gamma} (f \omega)$, we get
\begin{equation*} \Lambda_{\Gamma} + {\mathcal S}_{\Gamma} \subset \Lambda_{\Gamma}. \end{equation*}
  
In other words, 
$\Lambda_{\Gamma}$ is a semimodule over ${\mathcal S}_{\Gamma}$.
Our first result shows that $\Lambda_{\Gamma}$ has a stronger structure 
than a semimodule if $g \geq 2$.
\begin{defi}
  We say that a subset $S$ of ${\mathbb Z}_{>0}$ is a
  ${\mathcal C}_{\Gamma}$-{\it collection} (or just
  ${\mathcal C}$-collection if $\Gamma$ is implicit) if
  $S + {\mathcal S}_{\Gamma} \subset S$ and moreover, for all
  $m \in S$ and $1 \leq j \leq g$ such that
  $m \leq \overline{\beta}_{j}$, we have
  \begin{equation*}
    m + \overline{\beta}_{k} - \overline{\beta}_{j}
    \in S
  \end{equation*}
  for any $j \leq k \leq g$.
\end{defi}
\begin{rem}
  Zariski \cite[pg. 25]{Zariski:moduli} already realized that $\beta_1-\beta_{0}$ is an important invariant 
  in the moduli problem. 
  Theorem \ref{teo:collection} exhibits how the analogous differences $\overline{\beta}_{k} - \overline{\beta}_{j}$ 
  are fundamental in the classification problem. 
 \end{rem}
\begin{teo}
        \label{teo:collection}
	Let $\Gamma$ be a singular branch of planar curve.  Then
        $\Lambda_{\Gamma}$ is a ${\mathcal C}$-collection.
\end{teo}
Theorem \ref{teo:collection} is proved (as Proposition \ref{pro:L_is_c_coll}, page \pageref{pro:L_is_c_coll}) 
providing an explicit construction of
$\Lambda_{\Gamma}$.
\begin{defi}
  \label{def:cbasis}
  We say that ${\mathcal V}=(\lambda_1, \ldots, \lambda_m)$ is a
  ${\mathcal C}$-{\it basis of values} (of K\"{a}hler differentials) for
  $\Gamma$ if ${\mathcal V}$ is a minimal set of generators of the
  ${\mathcal C}$-collection $\Lambda_{\Gamma}$ (cf. Definition
  \ref{def:gen_collection}).  We say that
  ${\mathcal B} = (\omega_1, \ldots, \omega_m)$ is a
  ${\mathcal C}$-{\it basis} (of $1$-forms) for $\Gamma$ if
  $(\nu_{\Gamma} (\omega_1), \ldots, \nu_{\Gamma} (\omega_m))$ is a
  ${\mathcal C}$-basis of values. 
\end{defi}
\begin{rem}
  There is a unique ${\mathcal C}$-basis of values of the
  ${\mathcal C}$-collection $\Lambda_{\Gamma}$ (Proposition
  \ref{pro:min_gen}).
\end{rem}
\begin{rem}
\label{rem:lambda_on_s_b}
$\Lambda_{\Gamma}$ depends only on the ${\mathcal C}$-basis of values
and the semigroup ${\mathcal S}_{\Gamma}$.
\end{rem}
So, in order to understand the ``gap" between ${\mathcal S}_{\Gamma}$
and $\Lambda_{\Gamma}$ from a geometrical viewpoint, it suffices to
give a geometrical interpretation to the $1$-forms of a
${\mathcal C}$-basis. This is the goal of next subsection.

\subsection{Dicritical foliations}
Assume that $\Gamma$ is fixed.  The $1$-forms in a
${\mathcal C}$-basis for $\Gamma$, except $dx$ and $dy$, are
dicritical, i.e. have infinitely many integral curves 
(in a rough sense, one could conceive $dx$ and $dy$ as being ``dicritical'', with respect to the ``divisors'' $y=0$ and $x=0$, respectively).
 We describe them in this subsection.
\begin{defi}
\label{def:e_n}
We define $e_0 = n$, $e_j = \gcd (e_{j-1}, \beta_j)$, $n_0
=1$
and $n_j = e_{j-1}/e_j$ for any $1 \leq j \leq g$. We define the
partial multiplicities $\nu_j = n_0 \ldots n_j$ for any
$0 \leq j \leq g$.
\end{defi}
\begin{defi}
\label{def:j_family}
Let $\omega \in \Omega^{1} \cn{2} \setminus \{0 \}$.
Consider $1 \leq j \leq g$.  We say that a family
${\mathfrak C} = \{ \gamma_{j,s} : s \in {\mathbb C}^{*} \}$ of
invariant curves of $\omega=0$ 
\emph{fans $\Gamma$ to genus $j$}
if there exists a holomorphic map
\begin{equation*} 
  \begin{array}{ccccc}
    {\mathcal P} & : & U & \to & {\mathbb C}^{2} \\
                 & & (t, s) & \mapsto & \left( t^{n},
                                        \sum_{\beta < \beta_{j}}
                                        a_{\beta, \Gamma} t^{\beta}
                                        + s t^{\beta_{j}} 
                                        + \sum_{k=1}^{\infty}
                                        a_{\beta_{j} + k e_{j}} (s)
                                        t^{\beta_{j} + k e_{j}} \right)
  \end{array}
\end{equation*}
where $U$ is a neighborhood of $\{0\} \times {\mathbb C}^{*}$ in
${\mathbb C}^{2}$ and ${\mathcal P} (\cdot, s_0)$ is a parame\-tri\-zation
(which could be non-injective) of $\gamma_{j,s_0}$ for any
$s_0 \in {\mathbb C}^{*}$.  For short, whenever we speak of a
family $\mathfrak{C}$ of invariant curves for $\omega$, we shall omit
the fact that they are invariant, and just speak of a family
$\mathfrak{C}$ of $\omega$.
\end{defi}
\begin{rem}
  Fixed $\Gamma$ and $\omega$, there is at most one
  family $\mathfrak{C}$ of $\omega$
  fanning $\Gamma$ to genus $j$.  Hence, in such
  a case, each $\gamma_{j,s}$ is a well-defined branch.  Indeed,
  ${\mathfrak C}$ is a family of curves that intersects a
  non-invariant divisor of the desingularization of $\Gamma$ (Remark
  \ref{rem:bij_trace_exp}).
\end{rem}
\begin{rem} \label{rem:select}
  All curves belonging to a family of
  $\omega$ fanning $\Gamma$ to genus $j$
  are in the same equisingularity
  class as the curves $\Gamma_{j,s}$ of parametrization
  \begin{equation}
  \label{equ:G_j_s}
    \Gamma_{j,s}(t) = \left( t^{n/e_j},
      \sum_{\beta < \beta_{j}}
      a_{\beta, \Gamma} t^{\beta/e_j}  + s t^{\beta_{j}/e_j}
    \right)
  \end{equation}
  for $s \in {\mathbb C}^{*}$. Such curves have multiplicity $n/e_{j}$
  and Puiseux characteristic exponents
  $\beta_{1}/e_{j}, \ldots, \beta_{j}/e_{j}$. In particular, for
  $j=g$, where $e_{g} =1$, all curves in a family fanning
  $\Gamma$ to genus $g$ are in the equisingularity class of
  $\Gamma$.

  Moreover, a family (of $\omega$) fanning $\Gamma$ to genus
  $j$
  ``selects" a special invariant curve $\gamma_{j, s}$ for
  $s \in {\mathbb C}^{*}$.  It is the unique invariant curve $\gamma$
  in the equisingularity class of $\Gamma_{j,1}$ such that the
  intersection number $(\gamma, \Gamma_{j,s})$ at the origin is
  maximal (Remark \ref{rem:max_inter}).
  In other words, the family has a single member
  having maximal contact with $\Gamma$.
\end{rem}
 \begin{rem}
   For the case $g=1$, a family of $\omega$ fanning
   $\Gamma$ to genus $1$
   is called an {\it analytic semiroot} of $\Gamma$
   associated to $\omega$ in \cite{Cano-Corral-Senovilla:semiroots}.
 \end{rem}
\begin{defi}
\label{def:j_form}
We say that $\omega$ fans $\Gamma$ to genus $j$
if there exists a family of $\omega$ fanning $\Gamma$ to genus
$j$.
We say that $\omega$ fans $\Gamma$ if there is some
$1\leq j \leq g$ such that $\omega$ fans $\Gamma$ to order $j$. Notice
that $\omega$ is necessarily dicritical if this is the case.
\end{defi}
\subsection{Bases of fanning dicritical forms}
Now, we can describe a ${\mathcal C}$-basis of $\Gamma$ from a
geometrical viewpoint.  It consists of the trivial $1$-forms $dx$ and
$dy$ and dicritical forms fanning $\Gamma$. The following result follows as a corollary 
of Proposition \ref{pro:6-5-theorem} (page \pageref{pro:6-5-theorem}).
\begin{teo}
\label{teo:lambda_on_s_b}
Let $\Gamma$ be a singular branch of planar curve. Then there is a
${\mathcal C}$-basis
\begin{equation*}
  {\mathcal B} = (\Omega_1 = dx, \Omega_2 = dy, \Omega_3, \ldots,
  \Omega_m)
\end{equation*}
such that $\Omega_j$ fans $\Gamma$ to genus $k_j$
where
\begin{equation*}
  k_j = \max \{ 1 \leq k \leq g : \overline{\beta}_k < \nu_{\Gamma} (\Omega_j) \},
\end{equation*}
for any $2 < j \leq m$.
\end{teo}
In particular, note that given the ${\mathcal C}$-basis of values
$(n, \beta_1, \lambda_2, \ldots, \lambda_m)$, $k_j$ just depends on
$\lambda_j$ and ${\mathcal S}_{\Gamma}$ for any $1 < j \leq m$.  Thus,
the knowledge of $\lambda_j$ determines what kind of
fanning dicritical form $\omega_j$ is for $j>2$.
\begin{rem}
  The interest of Theorem \ref{teo:lambda_on_s_b} is that is valid for
  any singular branch $\Gamma$ of planar curve. The result is already
  known for the case $g=1$.  Then a ${\mathcal C}$-collection is just
  a set $S \subset {\mathbb Z}_{>0}$ such that
  $S + {\mathcal S}_{\Gamma} \subset S$ and Theorem
  \ref{teo:lambda_on_s_b} is a consequence of the main theorem of
  \cite[Cano, Corral,
  Senovilla-Sanz]{Cano-Corral-Senovilla:semiroots}.
 \end{rem}
 \begin{rem}
   A theorem of Rouill\'{e} asserts that given a germ of $1$-form
   $\Omega$ in $\cn{2}$ such that the foliation $\Omega=0$ is a
   generalized curve then
   $\nu_{\Gamma} (\Omega) \in {\mathcal S}_{\Gamma}$
   \cite{Rouille:Merle}.  Let us recall that a germ of foliation in
   $\cn{2}$ is a generalized curve if it is non-dicritical and
   no saddle-node
   singularity appears in its reduction of singularities.
   So the gap between ${\mathcal S}_{\Gamma}$ and
   $\Lambda_{\Gamma}$ has to be associated to saddle-node
   singularities or dicritical behavior.  We see in this paper that only
   the latter phenomenon is responsible for the gap and moreover that
   the dicritical nature of the $1$-forms in the ${\mathcal C}$-basis
   is linked to the characteristic divisors of the desingularization
   of $\Gamma$ (cf. Remark \ref{rem:bij_trace_exp} and Proposition
   \ref{pro:jG}).
 \end{rem}
 \subsection{Step by step construction of a basis}
 \label{subsec:step_by_step}
 First, we are going to introduce the exponents associated to
 $\Gamma$, their coefficients are parameters in the equisingularity
 class of $\Gamma$.  Our method to describe $\Lambda_{\Gamma}$
 analyzes the contributions of the coefficients of the exponents of
 $\Gamma$ one by one.
 \begin{defi}
 \label{def:set_exp}
 We define the {\it set of exponents}
 \begin{equation*}
   {\mathcal E}_{\Gamma} = 
   \cup_{1 \leq j \leq g}
   \{ \beta_j + k e_j :  k \in {\mathbb Z}_{\geq 0} \}
 \end{equation*}
 of $\Gamma$. Given $\beta \in {\mathcal E}_{\Gamma}$, we denote by
 ${\mathfrak n} (\beta)$ (resp. ${\mathfrak p} (\beta)$) the next
 element (resp. the previous) element of ${\mathcal E}_{\Gamma}$.  By
 convention, we define ${\mathfrak n}(n) = \beta_{1}$.
 \end{defi}
 Next, we introduce families of curves that appear naturally along the
 process.  Indeed, we work with the family $\fgb{\Gamma}{\beta}$ when
 we ponder the contribution of $a_{\beta}$ to the construction of
 bases.
 \begin{defi}
   We define $\fg{\Gamma}$ (the branches generated by
   $\mathcal{E}_{\Gamma}$) as the set of formal curves
 \begin{equation}
 \label{equ:exp}
 \overline{\Gamma} (t) =
 \left( t^{n},
   \sum_{\beta \in {\mathcal E}_{\Gamma} }
   a_{\beta, \overline{\Gamma}} t^{\beta}
 \right)
 \in {\mathbb C}\{t\} \times {\mathbb C}[[t]]  ,
\end{equation}
where $a_{\beta, \overline{\Gamma}} \in {\mathbb C}$ for any
$\beta \in {\mathcal E}_{\Gamma}$.  We denote by $\fga{\Gamma}$
(the branches strictly generated by $\mathcal{E}_{\Gamma}$) the subset
of $\fg{\Gamma}$ of formal curves such that
$a_{\beta_{1}, \overline{\Gamma}} \ldots a_{\beta_{g},
  \overline{\Gamma}} \neq 0$.  Finally,
\begin{equation*}
  \fgb{\Gamma}{\beta} = \left\{ 
    \overline{\Gamma} \in \fg{\Gamma}  :
    a_{\beta^{\prime}, \overline{\Gamma}}  = a_{\beta^{\prime}, \Gamma}
    \ \forall \beta^{\prime} < \beta  \right\} 
\end{equation*}
and
\begin{equation*}
     \fgba{\Gamma}{\beta}  = \fgb{\Gamma}{\beta}  \cap \fga{\Gamma}
\end{equation*}
for $\beta \in {\mathcal E}_{\Gamma}$.  By convention
$\fgb{\Gamma}{n}$ consists only of the curve $x=0$ and
$\fgb{\Gamma}{\infty} = \fgba{\Gamma}{\infty} =  \{ \Gamma \}$.
 \end{defi}
 \begin{rem}
   The parametrization \eqref{equ:exp} can be written in the form
   $\gamma (t^{k})$, where
   $\gamma \in {\mathbb C}\{t\} \times {\mathbb C}[[t]]$, if
   $k:= \gcd (\{ \beta \in {\mathcal E}_{\Gamma} : a_{\beta,
     \overline{\Gamma}} \neq 0 \} \cup \{ n\}) >1$.
 \end{rem}
 \begin{rem}
   A curve in $\fg{\Gamma}$ belongs to the equisingularity class of
   $\Gamma$ if and only if it belongs to
   $\fga{\Gamma}$. Moreover, any curve equisingular to $\Gamma$
   is analytically conjugated to a curve in $\fga{\Gamma}$. So,
   in order to understand $\Lambda_{\overline{\Gamma}}$ for
   $\overline{\Gamma}$ in the equisingularity class of $\Gamma$, it
   suffices to consider $\overline{\Gamma} \in \fga{\Gamma}$.
 \end{rem}
 \begin{rem}
   Besides $\Gamma$, we work with subfamilies of $\fg{\Gamma}$. Note
   that the coefficients $a_{\beta}$, with
   $\beta \in {\mathcal E}_{\Gamma}$, are parameters in $\fg{\Gamma}$.
\end{rem}
In order to obtain a ${\mathcal C}$-basis of values and forms, we work
with weaker (and simpler) structures.  Anyway, these new concepts are
building blocks to obtain the ${\mathcal C}$-basis.
\begin{defi}
  We say that a subset $S$ of ${\mathbb Z}_{>0}$ is a
  ${\mathbb C}[[x]]$-collection 
  if $s + kn \in S$ for all $s \in S$
  and $k \in {\mathbb Z}_{\geq 0}$.
\end{defi}
\begin{defi}
\label{def:nbasis}
As in Definition \ref{def:cbasis}, we can define a
${\mathbb C}[[x]]$-basis of values 
(also called an Apery basis of values \cite{Apery:branches})
and a ${\mathbb C}[[x]]$-basis of
$1$-forms for $\Gamma$.
\end{defi}
\begin{rem}
  A ${\mathbb C}[[x]]$-basis of values contains $n$ elements, one for
  each class of congruences modulo $n$.  A ${\mathbb C}[[x]]$-basis of
  $1$-forms for $\Gamma$ has independent interest since it provides a
  free basis of the ${\mathbb C}[[x]]$-module
  $\Gamma^{*} \hat{\Omega}^{1} \cn{2}$ (cf. Definition
  \ref{def:omega}) of formal differentials of $\Gamma$ (see
  Proposition \ref{pro:unique_expression_cbasis}).
\end{rem}
Let us give the idea of our method for constructing
a ${\mathbb C}[[x]]$-basis of
$1$-forms for $\Gamma$.  We divide the construction of
the forms in $g+1$ levels, from the
$0$-level to the $g$-level. The $1$-forms we obtain along
the way, are of the following special type:
\begin{defi}
  \label{def:strong}
  Consider $\omega \in \Omega^{1} \cn{2}$.  We say that $\omega$
  \emph{\touches{}} $\Gamma$
  if there exist an exponent $\beta \in {\mathcal E}_{\Gamma} \cup \{n \}$
  and a branch
  $\Gamma_{\beta} \in \fgb{\Gamma}{\beta}$ equisingular to the
  curve of parametrization
  $\left( t^{n}, \sum_{\beta^{\prime} < \beta} a_{\beta^{\prime},
      {\Gamma}} t^{\beta^{\prime}} \right)$ and is also invariant by
  $\omega =0$.
  We say that $\omega$ \emph{\touches{} $\Gamma$
    to order $\beta^{\prime}$}
  if $\beta^{\prime}$ is the supremum in
  ${\mathcal E}_{\Gamma} \cup \{n, \infty\}$ of the values satisfying
  that property.  Then we say that $\Gamma_{\beta^{\prime}}$ is a
  $\Gamma$-{\it companion curve} of $\omega$ (or just a companion
  curve if $\Gamma$ is implicit).
\end{defi}
\begin{rem}
  Obviously, a $1$-form $\omega$ which fans $\Gamma$
  necessarily \touches{} $\Gamma$ to some order. The converse does not
  hold in general, as a non-dicritical form can \touch{} $\Gamma$ but not
  fan it.
\end{rem}
 \begin{rem}
   Let $\omega$ be a $1$-form \touch{}ing $\Gamma$ to order $\beta$.
   The idea is that the higher the
   $\beta$, the closer $\Gamma$ is to being invariant by $\omega=0$.
 \end{rem}
 \begin{rem}
   We can think of a $\Gamma$-companion curve of $\omega$ as an
   invariant curve of $\omega=0$ that has maximal contact
   with $\Gamma$.
 \end{rem}
 We start with the family ${\mathcal T}_{\Gamma, 0} := \{ dx, dy \}$,
 both of whose elements \touch{} $\Gamma$.
 Indeed $dx$ \touches{} $\Gamma$ to order $n$, whereas $dy$
 \touches{} $\Gamma$ to order $\beta_1$. The companion curves of $dx$ and $dy$
 are $x=0$ and $y=0$ respectively.
 The former is a smooth curve whose tangent cone is different from that of $\Gamma$,
 while $y=0$ is a smooth curve with maximal contact with $\Gamma$. This family
 $\mathcal{T}_{\Gamma,0}$ will be called the \emph{terminal family of level} $0$,
 and is the starting point of our process. By hypothesis,
 $[n]_n\neq [\beta_1]_n$; this property of having different contacts modulo $n$ will
 be the basis of our construction: we are going to find, inductively, a set of
 $1$-forms $(\Omega_1,\ldots,\Omega_n)$ such that
 $[\nu_{\Gamma}(\Omega_i)]_n\neq [\nu_{\Gamma}(\Omega_{j})]_n$ for $i\neq j$. In
 order to find these forms, we shall need to impose specific properties on them, and
 use the approximate roots of $\Gamma$ (cf. Definition \ref{def:aprox_root}) in
 their construction.

 Let us briefly explain how to obtain iteratively the terminal families of levels
 $1,\ldots, g$.  As already indicated, the approximate roots
 $f_{ \Gamma, 1}, \ldots, f_{\Gamma, g}$ of $\Gamma$ are a key ingredient.
 We denote by $f_{\Gamma, g+1}$ an irreducible equation of
 $\Gamma$ and $\beta_{g+1} = \infty$. Consider the free
 ${\mathbb C}[[x]]$-module
\begin{equation}
\label{equ:module}
{\mathcal M}_{2 \nu_{j}} =  \left(
  \bigoplus_{j=0}^{\nu_{j}-1}
  \mathbb{C}[[x]] y^{j} dx\right) \,
\bigoplus\,
\left(\bigoplus_{j=0}^{\nu_{j}-1}  \mathbb{C}[[x]] y^{j} dy\right)
\end{equation}
for any $0 \leq j \leq g$ (cf. Definition \ref{def:e_n}).  Assume, inductively, that
the terminal family ${\mathcal T}_{\Gamma, l}$ of $l$-level for $0 \leq l < g$ is
known,which by definition (Definition \ref{def:level}, page \pageref{def:level})
is a free $\mathbb{C}[[x]]$-basis of ${\mathcal M}_{2 \nu_{l}}$ such that
any $1$-form in ${\mathcal T}_{\Gamma, l}$ \touches{} $\Gamma$ to order $\beta$, for
some $\beta$ with $n \leq \beta \leq \beta_{l+1}$. The initial family of level $l+1$ is defined as
\begin{equation}
\label{equ:ini_fam}
{\mathcal G}_{\Gamma, l+1}^{0} =
\{ f_{\Gamma, l+1}^{m} \omega : 0 \leq m < n_{l+1} 
\ \mathrm{and} \ \omega \in {\mathcal T}_{\Gamma, l} \}  
\end{equation}
Since $f_{\Gamma, l+1}$ is a monic polynomial in
$y$ with coefficients in ${\mathbb C}\{x\}$ of degree $\nu_{l}$
(cf. Remark \ref{rem:app_roots}), it follows that
${\mathcal G}_{\Gamma, l+1}^{0}$ is a free basis, consisting of
$2 \nu_{l+1}$ $1$-forms \touch{}ing $\Gamma$, of the the free
${\mathbb C}[[x]]$-module ${\mathcal M}_{2 \nu_{l+1}}$.  Indeed,
$f_{\Gamma, l+1}^{m} \omega$ \touches{} $\Gamma$ to order $\beta_{l+1}$
with companion curve $f_{\Gamma, l+1}=0$ for $m>0$ and
$\omega \in {\mathcal T}_{\Gamma, l}$.  Thus, all $1$-forms in
${\mathcal G}_{\Gamma, l+1}^{0}$ \touch{} $\Gamma$ to order $\beta$ for
some $n \leq \beta \leq \beta_{l+1}$.
Since $\deg_{y} (f_{\Gamma, l+2})= \nu_{l+1}$, we have
\begin{equation*}
  {\mathcal M}_{2 \nu_{l+1}} + (f_{\Gamma, l+2}) dx +  (f_{\Gamma, l+2}) dy
  = \hat{\Omega}^{1} \cn{2} ,
\end{equation*}  
where $(f_{\Gamma, l+2})$ denotes an ideal in ${\mathbb C}[[x,y]]$
(cf. Definition \ref{def:omega}).  So, somehow, we are working modulo
$(f_{\Gamma, l+2})$ in the $l+1$-level.

Unfortunately, the $1$-forms in ${\mathcal G}_{\Gamma, l+1}^{0}$ do
not have pairwise incongruent contacts with $\Gamma$ modulo $n$: for
instance, for the case $l=0$, we have
$\nu_{\Gamma} (y dx) = n + \nu_{\Gamma} (dy)$.  We introduce an
algorithm that modifies the basis for ${\mathcal M}_{2 \nu_{l+1}}$,
providing a new one with better properties.  At the final stage of the
algorithm for the $l+1$-level, we obtain a basis
${\mathcal T}_{\Gamma, l+1}$ of $1$-forms of
${\mathcal M}_{2 \nu_{l+1}}$ such that
\begin{itemize}
\item all contacts with $\Gamma$ of its forms are pairwise
  incongruent modulo $n$ if $l+1 < g$ or
\item it contains a ${\mathbb C}[[x]]$-basis of $1$-forms of $\Lambda_{\Gamma}$ if
  $l+1=g$, ending the construction.
\end{itemize}
The algorithm intends to construct bases
${\mathcal G}_{\Gamma, l+1}^{1}, \ldots, {\mathcal G}_{\Gamma,
  l+1}^{s} = {\mathcal T}_{\Gamma, l+1}$ of
${\mathcal M}_{2 \nu_{l+1}}$ with $1$-forms \touch{}ing $\Gamma$ whose
companion curves have increasing contact orders with $\Gamma$.  More
precisely, given $s \geq 1$, half the $1$-forms in
${\mathcal G}_{\Gamma, l+1}^{s}$ \touch{} $\Gamma$ to order $\beta$ for
some $\beta \geq {\mathfrak n}^{s} (\beta_{l+1})$.

For $l+1 < g$, we know how many times $s_l$ we have to run the algorithm
before finding the terminal family
${\mathcal G}_{\Gamma, l+1}^{s_l}={\mathcal T}_{\Gamma, l+1}$. Indeed, (see Proposition \ref{pro:from-l-to-next}, page \pageref{pro:from-l-to-next})
we have ${\mathfrak n}^{s_l} (\beta_{l+1}) = \beta_{l+2}$.
\begin{rem}
  The $1$-forms in ${\mathcal G}_{\Gamma, l+1}^{s}$, and hence their companion
  curves, depend on $a_{\beta, \Gamma}$, with
  $\beta < {\mathfrak n}^{s}(\beta_{l+1})$, for $s \geq 0$.  In other words, we have
  ${\mathcal G}_{\Gamma, l+1}^{s} = {\mathcal G}_{\Gamma^{\prime}, l+1}^{s}$ for any
  $\Gamma^{\prime} \in \fgba{\Gamma}{{\mathfrak n}^{s}(\beta_{l+1})}$.  
  Thus, our method contemplates the contribution of each
  coefficient $a_{\beta}$, for $\beta \in {\mathcal E}_{\Gamma}$, in the
  construction of a ${\mathbb C}[[x]]$-basis for $\Gamma$. Indeed, the contribution
  of $a_{{\mathfrak n}^{s} (\beta_{l+1}), \Gamma}$, with
  ${\mathfrak n}^{s} (\beta_{l+1}) < \beta_{l+2}$, is factored in when we compute
  ${\mathcal G}_{\Gamma, l+1}^{s+1}$ from ${\mathcal G}_{\Gamma, l+1}^{s}$.  The
  contribution of coefficients in the study of the analytic moduli space has been
  considered by Casas-Alvero in the case of one characteristic exponent
  \cite{Casas:class_1_exponent}.

  The $1$-forms in ${\mathcal G}_{\Gamma, l+1}^{s}$ that \touch{} $\Gamma$ to order
  $ \beta$ for $\beta \geq {\mathfrak n}^{s}(\beta_{l+1})$ are particularly
  interesting since their companion curves depend just on $a_{\beta, \Gamma}$, with
  $\beta < {\mathfrak n}^{s}(\beta_{l+1})$ but they deviate from $\Gamma$ at
  coefficients $a_{\beta}$ with $\beta \geq {\mathfrak n}^{s}(\beta_{l+1})$,
  i.e. that have not been treated in the algorithm yet.  Such companion curves are
  indicating possible gaps or jumps in the semimodule of curves in
  $\fga{\Gamma}$ associated to future coefficients. Our construction
  provides a geometrical insight into the semimodule of $\Gamma$, as we can describe
  the companion curves of the key dicritical $1$-forms obtained along the
  way: actually, any such $1$-form fans $\Gamma$ to genus $j$ for some
  $1 \leq j \leq g$ (see, for instance, Proposition \ref{pro:type_forms},
  page \ref{pro:type_forms}).
\end{rem}
 \begin{rem}
   Given $\beta \in {\mathcal E}_{\Gamma}$, we will describe the coefficients
   $a_{\beta}$ in which the behavior of $\overline{\Gamma}$-values of K\"{a}hler
   differentials for
   $\overline{\Gamma} \in \fgba{\Gamma}{\beta}$ change.
   These are the so called singular directions associated to $\overline{\Gamma}$ and
   the exponent $\beta$.  Any curve $\overline{\Gamma} \in \fga{\Gamma}$ whose
   coefficients $a_{\beta, \overline{\Gamma}}$ avoid the singular directions
   associated to $\overline{\Gamma}$ and $\beta$ for any
   $\beta \in {\mathcal E}_{\Gamma}$, has the same semimodule as a generic element
   of $\fga{\Gamma}$ (cf. Theorem \ref{teo:ana_inv_sing_dir}, page
   \pageref{teo:ana_inv_sing_dir}).
 \end{rem}

 Our methods provide remarkable connections between $\mathcal{S}_{\Gamma}$ and
 $\Lambda_{\Gamma}$. The following one is proved in page
 \pageref{pro:gen_sg_in_sbasis}, as Proposition \ref{pro:gen_sg_in_sbasis}:
 \begin{teo}
 	\label{teo:gen_sg_in_sbasis}
 	The minimal generator set of $\Lambda_{\Gamma}$, as a semimodule over
 	the semigroup $\mathcal{S}_{\Gamma}$, contains the minimal generator set of
 	$\mathcal{S}_{\Gamma}$.
 \end{teo}
 In the language of standard bases for $\Gamma$ 
 \cite{Delorme:modules} \cite{Cano-Corral-Senovilla:semiroots}, the theorem 
 states that there exists a standard basis that contains
 the $1$-forms $d f_{\Gamma,1}, \hdots, d f_{\Gamma, g}$.

 \strut
 
 Let us explain the structure of the paper.  We will describe basic
 concepts and properties that will be used throughout the paper in
 section \ref{sec:basic}.  The concept of leading variable for a
 $1$-form, along with some of its properties is introduced in section
 \ref{sec:leading}.  It is key in the constructions in the paper.
 This is already hinted at in subsection \ref{subsec:foliation}, where
 this framework is used to present foliations with interesting
 geometrical properties.  The construction of a geometrically
 significant ${\mathbb C}[[x]]$-basis for $\Gamma$ is described in
 section \ref{sec:construction_cx_bases}.  First, we describe the
 distinguished families of $1$-forms that will be used in the
 construction (subsection \ref{subsec:families}) and whose properties
 allow to organize the step-by-step construction that we present in
 subsections \ref{subsec:from_0_to_1} and \ref{subsec:from_s_to_s+1}.
 The analysis of the output of the construction is carried out in
 section \ref{sec:analysis_construction}.  In subsection
 \ref{subsec:step_by_step}, we explain how we are considering the
 contributions of the coefficients $a_{\beta}$, with
 $\beta \in {\mathcal E}_{\Gamma}$, one by one.  The geometrical
 properties of the $1$-forms of the ${\mathbb C}[[x]]$-basis of
 $1$-forms for $\Gamma$ given by the construction are described in
 subsection \ref{subsec:properties_forms_construction}.  This is a key
 ingredient in the proof of Theorem \ref{teo:lambda_on_s_b}.
 Subsection \ref{subsec:properties_semimodule} contains some initial
 consequences of our construction of a ${\mathbb C}[[x]]$-basis and in
 particular Theorem \ref{teo:collection}.  In section
 \ref{sec:generation_semimodule} we discuss several concepts of
 generator sets for $\Lambda_{\Gamma}$; we prove Theorem
 \ref{teo:lambda_on_s_b} and similar results for the other concepts
 of generation.  We also prove Theorem \ref{teo:gen_sg_in_sbasis}. 
  Finally, we profit from the step-by-step nature of
 our method to define singular directions for $\Gamma$.  They are
 infinitely near points of $\Gamma$ (cf. Definition
 \ref{def:seq_inf_pt}) that correspond to discontinuities in the
 values of K\"{a}hler differentials in the equisingularity class of
 $\Gamma$. We show that the singular directions are analytic
 invariants of $\Gamma$ (Theorem \ref{teo:ana_inv_sing_dir}).
 
 \section{Basic concepts and notation}
 \label{sec:basic}
 We introduce now the main definitions and notation required for our
 study.
 \begin{defi}
 \label{def:omega}
 We denote by $\Omega^{1} \cn{2}$ (resp. $\hat{\Omega}^{1} \cn{2}$)
 the set of germs of holomorphic (resp. formal) $1$-forms in a
 neighborhood of the origin in ${\mathbb C}^{2}$.  Their elements are
 of the form $a(x,y)dx + b(x,y)dy$ where $a, b \in {\mathbb C}\{x,y\}$
 (resp. $a, b \in {\mathbb C}[[x,y]]$).
  \end{defi}
 \begin{rem}
   Germs of holomorphic $1$-forms in $\cn{2}$ can be identified with
   derivations of the ${\mathbb C}$-algebra ${\mathbb C}\{x,y\}$.
   Formal $1$-forms can be identified with ${\mathbb C}$-derivations
   of ${\mathbb C}[[x,y]]$.
 \end{rem}  
 Let $\Gamma$ be an irreducible germ of holomorphic curve in $\cn{2}$. Given a power series $s(t)\in \mathbb{C}[[t]]$, we denote $\nu_0(s(t))=\ord_t(s(t))$.
 \begin{defi}
 \label{def:order}
 Let $\Gamma (t)$ be an irreducible Puiseux parametrization of $\Gamma$.  Let
 $f \in {\mathbb C}[[x,y]]$ and $\omega \in \hat{\Omega}^{1} \cn{2}$.  We define
 $\nu_{\Gamma} (f) = \nu_{0} (f \circ \Gamma(t))$, and
 $\nu_{\Gamma} (\omega) = \nu_{0} (h(t)) +1$ where $\Gamma^{*} \omega = h(t)
 dt$. This gives $\nu_{\Gamma}(\omega)=\ord_t(t\Gamma^{\ast}\omega)$. These
 numbers will be called the $\Gamma$-order of $f$ and of $\omega$,
 respectively.
\end{defi}
\begin{rem}
  The definition implies $\nu_{\Gamma} (f) = \nu_{\Gamma} (df)$. We prefer
  using this notation in order to avoid having to add $1$ everywhere when writing
  $\nu_{\Gamma}(\omega)$.
\end{rem}
\begin{rem}
\label{rem:order}
Let $\overline{\Gamma} (t) \in \fg{\Gamma}$ and suppose that
$\overline{\Gamma} (t) = \gamma(t^{e})$ where $\gamma (t)$ is an
irreducible parametrization and $e\in \mathbb{Z}_{>1}$.
Set $t \gamma^{*} \omega = h(t) dt$.
We have
\begin{equation*}
  t (\overline{\Gamma})^{*} (\omega) =
  t (t^{e})^{*} \gamma^{*} \omega = e h(t^{e}) dt .
\end{equation*}
So the order associated to the parametrization $\overline{\Gamma} (t)$
is $e$ times the order associated to $\gamma$.  We will consider this
last order, i.e.
$\nu_{\overline{\Gamma}} (\omega) := \nu_{\gamma} (\omega)$.
\end{rem}
 \begin{defi}
   For $n \in {\mathbb Z}_{>0}$, we denote
  ${\mathbb N}_{n} = \{ k \in {\mathbb Z} : 1 \leq k \leq n \}$.
\end{defi}
As stated in the in Introduction, we assume from now on that
$\Gamma$ is singular. What follows are the fundamental invariants of
$\Gamma$ (cf. \cite{Zariski:moduli}) we shall need.  First, we define the
approximate roots of $\Gamma$.
 \begin{defi}
 \label{def:truncation}
 Consider the Puiseux parametrization \eqref{equ:param} of $\Gamma$
 and $\beta \in {\mathcal E}_{\Gamma} \cup \{\infty\}$.  We define the
 curve $\Gamma_{< \beta}$ of parametrization
 \begin{equation*}
   \Gamma_{< \beta}(t) = \left( t^{n},
     \sum_{\beta^{\prime} < \beta} a_{\beta^{\prime}, \Gamma}
     t^{\beta^{\prime}} \right).
 \end{equation*}
 \end{defi} 
 \begin{rem}
 \label{rem:app_roots}
 Consider the Puiseux characteristic exponents $\beta_1, \hdots, \beta_g$ of
 $\Gamma$ and $\beta_{g+1}=\infty$.  An irreducible equation $f_{\Gamma, j}$ of
 $\Gamma_{< \beta_j}$ is of the form 
 \begin{equation}
 \label{equ:app_root_exp}
   f_{\Gamma, j} =
   \prod_{k=0}^{\nu_{j-1}-1}
   \left( y -  \sum_{\beta < \beta_{j}} e^{\frac{2 \pi i k \beta}{n}}
     a_{\beta, \Gamma} x^{\frac{\beta}{n}}
   \right) \in {\mathbb C}\{x\}[y] .
 \end{equation}
 Thus, the $y$-degree of $f_{\Gamma, j}$ is equal to $\nu_{j-1}$,
 which is also the multiplicity of $\Gamma_{< \beta_j}$.
 The curve $\Gamma_{< \beta_j}$ has Puiseux characteristic exponents
 $\beta_1/e_{j-1}, \ldots, \beta_{j-1}/e_{j-1}$ if $1 < j \leq g+1$.  Notice that
 $f_{\Gamma,1}=y$ (as we are assuming that $\Gamma$ has parametrization
 \eqref{equ:param}), that $\Gamma = \Gamma_{< \beta_{g+1}}$ and
 $f_{\Gamma, g+1}=0$ is an irreducible equation of $\Gamma$. Finally, for
 $j<g+1$, we actually have $f_{\Gamma,j}\in \mathbb{C}[x,y]$.
 \end{rem}
 \begin{defi}
 \label{def:aprox_root}
 The polynomials $f_{\Gamma, 1}, \ldots, f_{\Gamma, g}$ are called,
 since \cite{abhyankar-moh}, the \emph{approximate roots of $\Gamma$}.
 \end{defi}
 \begin{defi}
   We denote $\overline{\beta}_j = \nu_{\Gamma} (f_{\Gamma, j})$ for
   $1 \leq j \leq g$.
 \end{defi}
 \begin{rem}[cf. \cite{Zariski:moduli}]
 \label{rem:prop_beta}
 The set $\{ n, \overline{\beta}_1, \ldots, \overline{\beta}_g \}$ is
 the minimal set of generators of the semigroup
 ${\mathcal S}_{\Gamma}$ and hence
 $\overline{\beta}_1, \ldots, \overline{\beta}_g$ depend just on the
 equisingularity class of $\Gamma$. We have
 \begin{itemize}
 \item $\overline{\beta}_j$ is a multiple of $e_j$ and is the minimum
   element of ${\mathcal S}_{\Gamma}$ that is not a multiple of
   $e_{j-1}$ for any $j \in {\mathbb N}_{g}$;
 \item
   $n_j = \min \{ r \in {\mathbb Z}_{>0} : r \overline{\beta}_{j} \in
   (e_{j-1}) \}$ for any $j \in {\mathbb N}_{g}$ (cf. Definition
   \ref{def:e_n});
 \item  $\overline{\beta}_1 = \beta_1$ and
\begin{equation}
\label{equ:rec_beta}
\overline{\beta}_j =
n_{j-1} \overline{\beta}_{j-1} - \beta_{j-1} + \beta_{j} 
\end{equation}
for $2 \leq j \leq g$ (cf. \cite[Chapter II, Theorem
3.9]{Zariski:moduli});
 \item the $n$ expressions 
   \begin{equation*}
     k_1 \overline{\beta}_1 + \ldots + k_g \overline{\beta}_g,
   \end{equation*}
   with $0 \leq k_j < n_j$ for any $j \in {\mathbb N}_{g}$, define
   exactly $n$ classes modulo $n$. If we replace
   $\overline{\beta}_1, \ldots, \overline{\beta}_g$ with
   ${\beta}_1, \ldots, {\beta}_g$, the analogous property also holds
   true.
 \end{itemize}
\end{rem}
In our construction of a ${\mathbb C}[[x]]$-basis, we shall make frequent modular computations. 
The following notation has proved useful.
\begin{defi}
 \label{def:cong_n}
 Consider $n \in {\mathbb Z}_{\geq 1}$. Let $m \in \mathbb{Z}$.  We
 denote by $[m]_n$ the class of $m$ in ${\mathbb Z}/ n{\mathbb
   Z}$. Given $S \subset {\mathbb Z}$, we denote
 $[S]_{n} = \{ [m]_{n} : m \in S \}$.
 \end{defi}
 \begin{rem}
 \label{rem:no_fw_hole}
 It is clear from Definition \ref{def:set_exp} that given
 $\beta \in {\mathcal E}_{\Gamma}$, we have
 $\beta + k e_{\beta} \in {\mathcal E}_{\Gamma}$ for any
 $k \in {\mathbb Z}_{\geq 1}$, where
 \begin{equation*}
   e_{\beta} =
   \gcd (\{n\} \cup
   \{ \beta' \in  {\mathcal E}_{\Gamma} : \beta' \leq \beta \}).
 \end{equation*}
 Let $\beta', \beta \in {\mathcal E}_{\Gamma}$ with $\beta' < \beta$.
 Since ${\mathfrak n}(\beta') \leq \beta$, we get
 ${\mathfrak n}(\beta') - \beta' \in (e_{\beta})$ and thus
 $\beta + ({\mathfrak n}(\beta') - \beta') \in {\mathcal E}_{\Gamma}$.
 By repeating the argument with ${\mathfrak n}^{r}(\beta')$ and
 $\beta + ({\mathfrak n}^{r}(\beta') - \beta')$, we get that
 $\beta + ({\mathfrak n}^{r}(\beta') - \beta')\in {\mathcal
   E}_{\Gamma}$ for any $r \in {\mathbb Z}_{\geq 1}$ by induction on
 $r$.
 \end{rem}
 We provide now an alternative characterization of
 ${\mathbb C}[[x]]$-bases.
\begin{defi}
  The set of $1$-forms
  $\omega \in \hat{\Omega}^{1} \cn{2}$ such that
  $\Gamma^{*} \omega \equiv 0$, i.e. the $1$-forms that leave $\Gamma$
  invariant, will be denoted ${\mathcal I}_{\Gamma}$.
\end{defi}
\begin{pro}
\label{pro:unique_expression_cbasis}
Let $(\Omega_1, \ldots, \Omega_n)$ be a ${\mathbb C}[[x]]$-basis of
$\Gamma$. Then
\begin{equation*}
  \hat{\Omega}^{1} (\mathbb{C}^2,0)
  =\mathbb{C}[[x]] \Omega_{1} \oplus \hdots \oplus
  \mathbb{C}[[x]] \Omega_{n} \oplus {\mathcal I}_{\Gamma} .
\end{equation*}
\end{pro}
\begin{proof}
  Define the following operator ${\mathfrak f}$ in
  $\hat{\Omega}^{1} \cn{2}$: for $\omega \in \hat{\Omega}^{1} \cn{2}$, set
  ${\mathfrak f} (\omega)= \omega$ if $\nu_{\Gamma} (\omega) = \infty$; otherwise
  ($\nu_{\Gamma}(\omega)<\infty$), there exist $1 \leq k \leq n$ and
  $r \in {\mathbb Z}_{\geq 0}$ such that
  $\nu_{\Gamma} (\omega) - \nu_{\Gamma} (\Omega_{k}) = r n$. Thus, there exists a
  unique $c \in {\mathbb C}^{*}$ such that
  ${\mathfrak f} (\omega) := \omega - c x^{r} \Omega_{k}$ satisfies
  $\nu_{\Gamma} ({\mathfrak f} (\omega)) > \nu_{\Gamma} (\omega)$.
 
  Notice that
  $\omega - {\mathfrak f} (\omega) \in \sum_{j=1}^{n} {\mathbb C}[[x]]
  \Omega_j$ for any $\omega \in \hat{\Omega}^{1} \cn{2}$.  Let
  $\omega' = \sum_{j=0}^{\infty} ({\mathfrak f}^{j}(\omega) -
  {\mathfrak f}^{j+1}(\omega))$.  We have
  $\omega' \in \sum_{j=1}^{n} \mathbb{C}[[x]] \Omega_{j}$ by
  construction.  Since
  \begin{equation*}
    \omega = {\mathfrak f}^{k}(\omega) +
    \sum_{j=0}^{k-1}
    ({\mathfrak f}^{j}(\omega) - {\mathfrak f}^{j+1}(\omega))
  \end{equation*}
  for any $k \geq 1$ and
  $\lim_{k \to \infty} \nu_{\Gamma} ({\mathfrak f}^{k} (\omega)) =
  \infty$, we get $\omega - \omega' \in {\mathcal I}_{\Gamma}$.
 
  In order to show that the sum is direct, it suffices to prove that
  $\sum_{j=1}^{n} h_{j}(x) \Omega_j \in {\mathcal I}_{\Gamma}$ implies
  $h_1 \equiv \ldots \equiv h_n \equiv 0$.  Suppose it is not the
  case. Since
  $\nu_{\Gamma} (h_{j}(x) \Omega_j) \neq \nu_{\Gamma} (h_{k}(x)
  \Omega_k)$ for $j \neq k$, it follows that
  $\min_{1 \leq j \leq n} \nu_{\Gamma} (h_{j}(x) \Omega_j)$ is
  attained at exactly one $1 \leq l \leq n$ and thus
  \begin{equation*}
    \nu_{\Gamma} \left( \sum_{j=1}^{n} h_{j}(x) \Omega_j \right) =
    \nu_{\Gamma} (h_{l}(x) \Omega_l) \neq \infty.
  \end{equation*}
  This implies
  $ \sum_{j=1}^{n} h_{j}(x) \Omega_j \not \in {\mathcal I}_{\Gamma}$,
  providing a contradiction.
 \end{proof} 
\section{Leading variables} 
\label{sec:leading}
Let $\Gamma$ be a singular irreducible germ of holomorphic curve in $\cn{2}$.
The concept of $1$-forms \touch{}ing $\Gamma$ to order $\beta$ was presented
in the Introduction (cf. Definition \ref{def:strong}). They are a
fundamental ingredient in our project of describing $\Lambda_{\Gamma}$ from a
geometrical viewpoint. As explained above, if $\omega$ is such a $1$-form,
it automatically has an invariant curve, called \emph{companion curve},
whose Puiseux expansion deviates from the one of $\Gamma$ at the exponent
$\beta$. In this section we introduce a related concept for functions and
$1$-forms: having $\beta$ (or $a_{\beta}$) as leading variable.  On one hand, if a
$1$-form $\omega$ satisfies this property, then it \touches{} $\Gamma$. On the
other, keeping account of these leading variables allows us to develop an iterative
method for finding a ${\mathbb C}[[x]]$-basis, guaranteeing that all the $1$-forms
we obtain along the process \touch{} $\Gamma$.

The construction of a $\mathbb{C}[[x]]$-basis in the generic case is carried
out in \cite{ayuso2024construction}. There, the coefficients $a_{\beta}$ are
considered as parameters over $\mathbb{C}$. The present work deals with the dual
case: all the values $a_{\beta}$ are fixed in advance. This prevents transferring
the arguments in that reference to our present construction. However, we can proceed in
\emph{generic batches}, so to say, one for each Puiseux exponent. Instead of
considering all the coefficients $a_{\beta}$ as parameters, we shall, at each
\emph{level} $l+1$, work over the family $\fg{\geq \beta_{l+1}}$ (i.e. the coefficients
$a_{\beta}$ will be parameters for $\beta\geq \beta_{l+1}$). A careful study of our
procedure will show that the corresponding $1$-forms $\omega$ we construct will
satisfy that, for $\overline{\Gamma}\in \fg{\geq \beta_{l+1}}$, there is $k$ and $\beta$
such that:
\begin{equation*}
  (t\overline{\Gamma}^{\ast}\omega) = t^{k}\left[
    D_{\beta}t^{\beta} + D_{\mathfrak{n}(\beta)}t^{\mathfrak{n}(\beta)} + \cdots
  \right] dt
\end{equation*}
where each $D_{\eta}$ has a suitable expression of the form
\begin{equation*}
  D_{\eta} = a_{\beta}^ma_{\eta}\square_{<\beta} + \triangle_{<\eta}
\end{equation*}
for some polynomials $\square_{<\beta}$ and $\triangle_{<\eta}$ which depend on
$a_{\gamma}$, for $\gamma<\beta$ and $\gamma<\eta$, respectively. This property
(here stated in very rough form) is what we shall call having \emph{leading variable
  $a_{\beta}$} (Definition \ref{def:max}).

\subsection{Idea of the algorithm} 
Our method is purely algorithmic and deterministic. It is based on the fact that, in order to
find a $\mathcal{C}$-basis of $\Lambda_{\Gamma}$, it is enough to find a
$\mathbb{C}[[x]]$-basis, and this requires just finding, for each class $[a]_{n}$ with
$a\in\{0,\ldots, n-1\}$, a $1$-form $\Omega_a$ such that $\nu_{\Gamma}(\Omega_a)$ is minimal in
$[a]_n$. This is the basic idea.

There are two obvious initial candidates: $\Omega_0=dx$, and $\Omega_1=dy$. The
first one has $[\nu_{\Gamma}(\Omega_0)]_n=[0]_n$, and the other
$[\nu_{\Gamma}(\Omega_1)]_n=[\beta_1]_n$. Both classes are obviously different
because $\beta_1$ is not a multiple of $n$, and the values $\nu_{\Gamma}(\Omega_0)$
and $\nu_{\Gamma}(\Omega_1)$ are clearly minimal in their classes. These two
candidates provide what we call the \emph{terminal family of level} $0$. Notice how
we reach $\beta_1$ (the first Puiseux exponent) but cannot go further than $\beta_1$
with $1$-forms in $\mathbb{C}[[x]]\{dx, dy\}$.

Algorithm \ref{alg:construction} is a semi-formal description of our procedure. 
Starting with the family $\mathcal{T}_{\Gamma,0}=\{dx,dy\}$, 
we proceed, for each Puiseux exponent $\beta_{l+1}$, to compute another 
family $\mathcal{T}_{\Gamma,l+1}$, up to $\beta_g$. The computation of $\mathcal{T}_{\Gamma,l+1}$ 
starts with the family $\mathcal{G}^0_{\Gamma,l+1}$, which is built from $\mathcal{T}_{\Gamma,l}$ using powers of the $l+1$-th 
approximate root of $\Gamma$. Then, there is a \textsc{Main Transformation} which for each $s\geq 0$ produces a family 
$\mathcal{G}^{s+1}_{\Gamma,l+1}$ from $\mathcal{G}^s_{\Gamma,l+1}$. This transformation is repeated 
until $\mathcal{G}^{s+1}_{\Gamma,l+1}$ is \emph{terminal}. Finally, when $l+1=g$, the transformation $ \mathrm{SORT}_{n}$ 
returns the $n$ elements of $\mathcal{T}_{\Gamma,n}$ whose contacts with $\Gamma$ are minimal.
\begin{algorithm}
  \caption{Construction of a $\mathbb{C} [[x]]$-basis}
  \label{alg:construction}
  \begin{algorithmic}
    \State $l \gets 0$
    \State $\mathcal{T}_{\Gamma,l} \gets \{dx,dy\}$
    \While{$l<g$}
    \State $\mathcal{G}^0_{\Gamma,l+1}=\{\}$
    \State{\textsf{\hspace*{10pt}Next loop: build $\mathcal{G}^0_{\Gamma, l+1}$ with $\mathcal{T}_{\Gamma,l}$}}
    \State{\textsf{\hspace*{10pt}and powers of the $l+1$-th approximate root of $\Gamma$}}
    \For{$0\leq a < n_{l+1} \land 1\leq b \leq 2\nu_{l}$}
    \State $\mathcal{G}^{0}_{\Gamma,l+1} \gets
    \mathcal{G}^{0}_{\Gamma,l+1} \cup\{\Omega^0_{2a\nu_{l}+b}=f^a_{\Gamma,l+1}\Omega_{b}\}$
    \EndFor
    \State $s \gets 0$
    \While{$\mathcal{G}^s_{\Gamma,l+1}$ is not \emph{terminal}}
    \State{$\mathcal{G}^{s+1}_{\Gamma,l+1}\gets
      \textsc{Main Transformation}(\mathcal{G}^{s}_{\Gamma,l+1})$}
    \State{$s \gets s+1$}
    \EndWhile
    \State $\mathcal{T}_{\Gamma,l+1}=\mathcal{G}^s_{\Gamma,l+1}$
    \State $l\gets l+1$
    \EndWhile
    \State $\mathcal{B} \gets \mathrm{SORT}_{n} (\mathcal{T}_{\Gamma,g})$  
    \State $\mathcal{B}$ is a $\mathbb{C}[[x]]$-basis
  \end{algorithmic}
\end{algorithm}
The role of the approximate roots of $\Gamma$ is essential, as can be seen in the \textbf{for}
loop.

The \textsc{Main Transformation} on $\mathcal{G}^s_{\Gamma,l+1}$ consists of the following
operation: we will identify distinguished pairs $(a, b) \in \mathbb{N}_{2 \nu_{l+1}}^{2}$, with $a \neq b$, such that
$[\nu_{\Gamma}(\Omega^s_a)]_n=[\nu_{\Gamma}(\Omega^s_b)]_n$ 
and $\nu_{\Gamma}(\Omega^s_a) \leq \nu_{\Gamma}(\Omega^s_b)$
(cf. Definitions 
\ref{def:good1} and  \ref{def:next}).  
 The  ``transformation'' consists in keeping
$\Omega_a^s$ (and calling it $\Omega_a^{s+1}$) and setting
\begin{equation*}
  \Omega_b^{s+1} = \Omega_b^{s} - c x^{d}\Omega_a^{s},
\end{equation*}
where $dn + \nu_{\Gamma}(\Omega_a^s) = \nu_{\Gamma}(\Omega_b^s)$ and $c$ is the unique nonzero
complex number such that
$\nu_{\Gamma}(\Omega_b^{s} - c x^{d}\Omega_a^{s})>\nu_{\Gamma}(\Omega_b^s)$.
We define $\Omega_{r}^{s+1} = \Omega_{r}^{s} $ for any $r \in \mathbb{N}_{2 \nu_{l+1}}$ that is not in any of 
the distinguished pairs.

 Consider the case $l+1<g$ (the case $l+1=g$ requires a different study). The set $[\nu_{\Gamma} (\mathcal{G}^s_{l+1})]_{n}$ contains fewer than 
 $2 \nu_{l+1}$ classes modulo $n$ until some specific value of $s$ 
 (which depends on $\beta_{l+1}$ and $\beta_{l+2}$)
 for which all the classes $[\nu_{\Gamma}(\Omega^{s}_{j})]_n$ are different,
which is the definition of $\mathcal{G}^s_{l+1}$ being \emph{terminal}. When this happens, we
set $\mathcal{T}_{l+1}=\mathcal{G}^s_{l+1}$ and call $\mathcal{T}_{l+1}$ \emph{the terminal
  family of level $l+1$}.

The case $l+1 = g$ is a little different. The goal is achieved when there is a
subset $\mathcal{B}$ of $\mathcal{G}^s_{l+1}$ of $n$ elements such that
$\sharp [\nu_{\Gamma} (\mathcal{B})]_{n} = n$ (so there are no repetitions of
classes modulo $n$) and the $\Gamma$-orders of elements of $\mathcal{B}$ are minimal
in $\mathcal{G}^s_{l+1}$, or more precisely
$\max (\nu_{\Gamma} (\mathcal{B})) < \min (\nu_{\Gamma} (\mathcal{G}^s_{l+1}
\setminus \mathcal{B}))$ holds.  Then $\mathcal{T}_{g}=\mathcal{G}^s_{g}$ is the
\emph{the terminal family of level $g$} and $\mathcal{B}$ is a
$\mathbb{C}[[x]]-$\emph{basis} for $\Gamma$.

Verifying that the above elementary algorithm works requires proving the property
stated above, and a detailed analysis of the properties of
$\mathcal{G}^0_{\Gamma,l+1}$, and of $\mathcal{G}^{s+1}_{\Gamma,l+1}$ given
$\mathcal{G}^s_{\Gamma,l+1}$. This analysis, in essence, consists in a detailed
tracking of the values $\nu_{\Gamma}(\Omega^{s}_j)$ and
$\nu_{\Gamma}(\Omega^{s+1}_j)$ before and after the \textsc{Main Transformation},
and of the geometric properties of each $\Omega_j^s$ along the way. The necessary
control to carry out this program is provided by the concepts of {\it leading
  variables} and $c$-{\it sequences} that will be introduced in subsection
\ref{subsec:definition_leading}.

In subsection \ref{subsec:leading_products} we study the leading variables of the
different kinds of functions and $1$-forms we shall encounter (this is mainly for
the passage from $\mathcal{G}^0_{l+1}$ to $\mathcal{G}^1_{l+1}$, which needs a
specific study).  Leading variables are key to describe the two different types of
foliations associated with the forms $\Omega_j^s$ that will appear (subsection
\ref{subsec:foliation}).

In section \ref{sec:construction_cx_bases}, we carry out the comprehensive analysis
of the properties of the forms in $\mathcal{G}^s_{l+1}$, via their leading
variables, and how, after a finite number of iterations, we end up with a terminal
family for each $l+1\leq g$. The terminal family of level $g$ contains a
$\mathbb{C}[[x]]$-basis for $\Gamma$.

\subsection{Definition and basic properties of leading variables}
\label{subsec:definition_leading}
We now introduce the main ingredient which will allow us to keep track of
the arithmetic and geometric properties of the $1$-forms constructed along the way:
that of \emph{leading variable} of a power series (and, as a generalization, of a
function or a $1$-form), for a family $\fg{\geq \beta_j}$. We shall use the
following notation:
\begin{defi}
  Let $a_{\beta_1}, \ldots$ be parameters over $\mathbb{C}$, and
  consider a power series $g = \sum_{r=0}^{\infty} D_{r} t^{r}$ where
  $D_r \in {\mathbb C}[a_{\beta_1}, a_{{\mathfrak n}(\beta_1)}, a_{{\mathfrak
      n}^{2}(\beta_1)}, \ldots]$ for any $r \geq 0$.  We define
  \begin{equation*}
    D_{r} (\Gamma) =
    D_r (
    a_{\beta_1, \Gamma}, a_{{\mathfrak n}(\beta_1), \Gamma},
    a_{{\mathfrak n}^{2}(\beta_1), \Gamma}, \ldots)
  \end{equation*}
  and $g_{\Gamma} (t) = \sum_{r=0}^{\infty} D_{r}(\Gamma) t^{r}$.
\end{defi}
\begin{defi}
\label{def:max}
Let $1 \leq j \leq g+1$.  Using the same notation as above, we say
that $g \in {\mathbb C} [a_{\beta^{\prime}}]_{\beta^{\prime} \in {\mathcal E}_{\Gamma}, \
  \beta^{\prime} \geq \beta_j}[[t]]$ has {\it leading variable} $\beta$ (or
$a_{\beta}$) of degree $m$ for the family $\fgb{\Gamma}{\beta_j}$, where
$\beta \in {\mathcal E}_{\Gamma}$ and $m \in {\mathbb Z}_{\geq 0}$, if
\begin{enumerate}[a)]
\item\label{it:lead-C} The series $g(t)$ is of the form 
  \begin{equation*}
    g(t)
    = t^{k} \sum_{r=0}^{\infty}
    C_{{\mathfrak n}^{r} (\beta)} t^{{\mathfrak n}^{r} (\beta)}
  \end{equation*}
  where
  $C_{{\mathfrak n}^{r} (\beta)} \in {\mathbb C}[a_{\beta^{\prime}}]_{\beta^{\prime}
    \in {\mathcal E}_{\Gamma}, \ \beta_{j} \leq \beta^{\prime} \leq {\mathfrak
      n}^{r} (\beta)}$ for some $k \in {\mathbb Z}_{\geq 0}$ and any
  $r \in {\mathbb Z}_{\geq 0}$ and moreover $C_{\beta} \not \equiv 0$. This
  latter property implies that if $g$ has a leading variable, then it is
  unique.
\item\label{it:lead-R} There exist
  $R \in {\mathbb C}[a_{\beta^{\prime}}]_{\beta^{\prime} < \beta}$, a
  constant $\tilde{c} \in {\mathbb C}^{*}$ and a sequence
  $1=c_{\beta}, c_{{\mathfrak n}(\beta)}, c_{{\mathfrak
      n}^{2}(\beta)}, \ldots$ in ${\mathbb C}$ such that
\begin{equation}
	\label{equ:maximal}
        C_{{\mathfrak n}^{r} (\beta)}
        = \tilde{c} c_{{\mathfrak n}^{r}(\beta)} a_{\beta}^{m} 
        R (a_{\beta_{j}},
        \ldots, a_{{\mathfrak p}(\beta)})   a_{{\mathfrak n}^{r} (\beta)}
        + Q_{{\mathfrak n}^{r} (\beta)} ,   
\end{equation}
where
$R(a_{\beta_{j}, \Gamma}, \ldots, a_{{\mathfrak p}(\beta), \Gamma})
\neq 0$ and
$Q_{{\mathfrak n}^{r} (\beta)} \in {\mathbb C}[a_{\beta^{\prime}}]_
{\beta^{\prime}< {\mathfrak n}^{r}(\beta)}$ for any $r \geq 0$. Notice that $R$ is not necessarily a constant 
because $\beta$ may be greater than $\beta_j$.
\item\label{it:lead-m} $m=0$ if $\beta \not \in \{ \beta_{1}, \ldots \beta_j
  \}$. Note that the set contains $j$ elements, not $g$. Thus, if
  $\beta>\beta_j$,  
  then $m=0$.
\item\label{it:lead-k} $[k]_{e_{\ell}} = [m \beta]_{e_{\ell}}$, where $\ell = \max
  (\{0\} \cup \{ r \geq 1 : \beta_{r} < \beta
  \})$ (cf. Definition \ref{def:cong_n}).
\end{enumerate}
\end{defi}
\begin{rem}
  Properties \ref{it:lead-R} and \ref{it:lead-k} are
  essential; they reflect, among other things, the fact that we are
  considering the family $\fg{\geq \beta_j}$, not any family of curves
  (so that the set of available exponents in $g$ is not
  arbitrary). This will become apparent when we study the geometric
  properties of $1$-forms $\omega$ with some leading variable (for
  instance, Proposition \ref{pro:unique_companion}).
\end{rem}
\begin{defi}
We call $(c_{{\mathfrak n}^{r}(\beta)})_{r \geq 0}$ the \emph{$c$-sequence of $g$}.
Denote by $\Beg_{\Gamma}(g)$ (the \emph{beginning} of $g$ on $\Gamma$) the first
$\beta^{\prime}$ in ${\mathcal E}_{\Gamma}$ such that
$C_{\beta^{\prime}} (\Gamma) \neq 0$ and $\Beg_{\Gamma}(g) = \infty$ if
$g_{\Gamma} \equiv 0$.
\end{defi}
\begin{rem}
  In the definition above, it is assumed that
  ${\mathbb C}[\emptyset] = {\mathbb C}$.  For instance
  $ R \in \mathbb{C}{[a_{\beta'}]}_{\beta^{\prime} < \beta}$ should be interpreted
  as $R \in \mathbb{C}$ if $\beta < \beta_{j}$. The term $a_{\beta}^{m}$ in Equation
  \eqref{equ:maximal} should be interpreted as $a_{\beta, \Gamma}^{m}$ if
  $\beta < \beta_{j}$.
\end{rem}
Notice that not all power series have a leading variable. However, all the
$1$-forms we compute will have a very specific leading variable.
\begin{rem}
  \label{rem:l_b}
  Suppose that $g$ has leading variable $\beta$ for the family
  $\fgb{\Gamma}{\beta_j}$.  We have
  \begin{itemize}
  \item $\Beg_{\Gamma}(g) \geq \beta$ and
  \item $\Beg_{\Gamma}(g)=\beta$ if $\beta < \beta_{j}$.
  \end{itemize}
  Thus, we get $\Beg_{\Gamma}(g) > \beta_{j}$ if $\beta > \beta_{j}$ by the first
  condition. Moreover, we obtain $\Beg_{\Gamma}(g) < \beta_{j}$ if
  $\beta < \beta_{j}$ by the second condition. Indeed, $\beta<\beta_j$ implies that 
  $C_{\beta}$ is a constant and thus $\Beg_{\Gamma}(g)=\beta$.
  It follows that
  $\Beg_{\Gamma}(g) = \beta_{j}$ implies $\beta = \beta_{j}$.
\end{rem}
\begin{defi}
  The notion of leading variable for the family
  $\fgb{\Gamma}{\beta_j}$, can be extended to
  $f \in {\mathbb C}[[x,y]]$ (resp.
  $\omega \in \hat{\Omega}^{1} \cn{2}$) by associating
  $g(t) \in \fgb{\Gamma}{\beta_j}$ to $f$ (resp. $\omega$).  Indeed,
  since ${(a_{\beta})}_{\beta \geq \beta_{j}}$ parametrizes
  $\fgb{\Gamma}{\beta_j}$, it follows that the power series
  $g(t) = f \circ \Gamma^{\prime} (t)$ (resp. $g(t)$ such that
  $t (\Gamma^{\prime})^{*} \omega = g(t) dt$), where $\Gamma^{\prime}$
  varies in $\fgb{\Gamma}{\beta_j}$, can be interpreted as an element
  of $\mathbb{C} {[a_{\beta}]}_{\beta \geq \beta_{j}} [[t]]$.  We
  define $\Beg_{\Gamma}(f) = \Beg_{\Gamma}(g)$
  (resp. $\Beg_{\Gamma}(\omega)=\Beg_{\Gamma}(g)$).
\end{defi}
%
\begin{example}
  Consider the curve
  \begin{equation*}
    \Gamma = \left( t^{{15}},
      t^{18} + t^{24} + t^{25} + t^{26}
    \right),
  \end{equation*}    
  which has two Puiseux exponents, $\beta_1=18$ and $\beta_2=25$. The set
  $\mathcal{E}_{\Gamma}$ is:
  \begin{equation*}
    \mathcal{E}_{\Gamma} = \left\{ 15, 18, 21, 24, 25, 26,\ldots \right\}.
  \end{equation*}
  Set $f=y^5-x^6$, and $j=1$, so that $\beta_j=18$. The
  family $\fgb{\Gamma}{18}$ gives, for $f$:
    \begin{multline*}
      f(t) = \left(a_{18}^5-1\right) t^{90}+5 a_{18}^4 a_{21} t^{93}+
      \left(5 a_{18}^4 a_{24}+10 a_{18}^3 a_{21}^2\right) t^{96}+\\
      5 a_{18}^4 a_{25} t^{97}+5 a_{18}^4 a_{26} t^{98}+\cdots
  \end{multline*}
  which can be rewritten as:
  \begin{multline*}
    f(t) = t^{72}\left( \left(a_{18}^5-1\right) t^{18}+
      5 a_{18}^4 a_{21} t^{21}+\left(5
        a_{18}^4 a_{24}+10 a_{18}^3 a_{21}^2\right) t^{24}+\right.\\
      \left.5 a_{18}^4
        a_{25} t^{25}+ 5a_{18}^{4}a_{26}t^{26}+\cdots\right).
  \end{multline*}
  Setting $k=72$, $\beta=18$, $m=4$, $R=1$, $Q_{18}=-1$, $Q_{21}=0$,
  $Q_{24}=10a_{18}^3a_{21}^2$, $Q_{25}=0$, $Q_{26}=0$, $\tilde{c}=1$, and $c_{18}=1$,
  $c_{21}=5$, $c_{24}=5$, $c_{25}=5$, $c_{26}=5$, one can see that (at least for that
  truncation), $f$ has leading variable $\beta=18$. 
  Notice that $e_{\ell}=15$ and
  $[k]_{15} = [72]_{15}=[\beta m]_{15}$. In this case, as $f$ is
  an approximate root of $\Gamma$, the fact that $f$ has leading
  variable $18$ can be verified using the multinomial formula. When restricting $f(t)$ to
  $\Gamma$, we obtain:
  \begin{equation*}
    \Gamma^{\ast}f = 5t^{96} + 5t^{97} + \hdots = 
    t^{72}\left(5t^{24} + 5t^{25} + \hdots \right),
  \end{equation*}
  so that $\Beg_{\Gamma}(f)=24>\beta=18$.

  Let now $\omega=5xdy-6ydx$, and consider the same family $\fgb{\beta}{18}$.  We obtain, abusing
  notation:
  \begin{equation*}
    t\fgb{\beta}{18}^{\ast}\omega=15a_{21}t^{36}+30a_{24}t^{39}+35a_{25}t^{40}+40a_{26}t^{41} + \cdots
  \end{equation*}
  which is
  \begin{equation*}
    t\fgb{\beta}{18}^{\ast}\omega=t^{15}\left(15a_{21}t^{21}+30a_{24}t^{24}+
      35a_{25}t^{25}+40 a_{26} t^{26} + \cdots\right).
  \end{equation*}
  In this case, $k=15$, $\beta=21$, 
  $\tilde{c}=15$, $c_{21}=1$, $c_{24}=2$,
  $c_{25}=7/3$, $c_{26}=8/3$, etc;
  $Q_{\mathfrak{n}^{r}(\beta)}=0$ for all $r\geq 0$,
  $R=1$, $m=0$. When restricting to $\Gamma$, we obtain $\Beg_{\Gamma}(\omega)=24$,
  because $a_{21,\Gamma}=0$.
\end{example}
Our algorithm for constructing a ${\mathbb C}[[x]]$-basis of $\Gamma$ provides
$1$-forms \touch{}ing $\Gamma$ with $c$-sequences which are either increasing or
strictly increasing sequences of positive rational numbers. This property will be
essential in our arguments. 
\begin{defi}
  We denote by $\DD{\beta}{m}{j}$ the set of power series $g(t)$
  (resp. power series in ${\mathbb C}[[x,y]]$, formal $1$-forms in
  $\hat{\Omega}^{1} \cn{2}$) with leading variable $\beta$ and degree
  $m$ for the family $\fgb{\Gamma}{\beta_{j}}$. We consider two
  distinguished subsets, namely
  \begin{itemize}
  \item the set $\DD{\beta,\leq}{m}{j}$ of power series
    $g(t) \in \DD{\beta}{m}{j}$ whose $c$-sequence is an increasing
    sequence in ${\mathbb Q}_{>0}$;
  \item the set $\DD{\beta,<}{m}{j}$ of power series
    $g(t) \in \DD{\beta}{m}{j}$ whose $c$-sequence is a strictly
    increasing sequence in ${\mathbb Q}_{>0}$.
  \end{itemize} 
  We define also the ``strict'' families with leading variable $\beta$,
  i.e. those whose leading variable coincide with their start on $\Gamma$:
  \begin{equation*}
    \begin{split}
      &\oDD{\beta}{m}{j} = \DD{\beta}{m}{j} \cap
        \Beg_{\Gamma}^{-1} (\beta), \\
      & \oDD{\beta,\leq}{m}{j} =
        \oDD{\beta}{m}{j} \cap \DD{\beta,\leq}{m}{j},\\
      & \oDD{\beta,<}{m}{j}
        = \oDD{\beta}{m}{j} \cap \DD{\beta,<}{m}{j}.
    \end{split}
  \end{equation*}
  We also consider the union of the previous sets for all possible
  degreees. Hence, we denote
\begin{equation*}
  \Dd{\beta}{j} = \cup_{m \geq 0} \DD{\beta}{m}{j} \ \mathrm{and} \ 
  \oDd{\beta}{j} = \cup_{m \geq 0}\oDD{\beta}{m}{j} .
\end{equation*}
Families $\Dd{\beta,\leq}{j}$, $\Dd{\beta,<}{j}$,
$\oDd{\beta,\leq}{j}$ and $\oDd{\beta,<}{j}$ are defined analogously.

Finally, we denote by $\tDd{\beta}{j}$ the set of
$g \in \Dd{\beta}{j}$ such that $\Beg_{\Gamma}(g)$ is a Puiseux
exponent of $\Gamma$ and $Q_{\Beg_{\Gamma}(g)} (\Gamma) =0$
(cf. Definition \ref{def:max}).  We define
\begin{equation*}
  \begin{split}
    &\odd{\beta} = \cup_{j \geq 1} \oDd{\beta}{j} \\  
    &\odd{\beta,\leq} =
    \cup_{j \geq 1} \oDd{\beta,\leq}{j}\\
    &\odd{\beta,<} = \cup_{j \geq 1} \oDd{\beta,<}{j}.
  \end{split}
\end{equation*}
We define $\odd{n} = \odd{n,<} = \dd{n,<} =\{dx\}$
for completeness
in the case of $1$-forms.
\end{defi}
The next properties follow straightforwardly  from Definition
\ref{def:max}.
\begin{rem}
  Notice that $\omega \in \DD{\beta}{m}{j}$ implies
  $\Beg_{\Gamma}(\omega) \geq \beta$.
\end{rem}
\begin{rem}
  \label{rem:family_restriction}
  Let $\omega \in \oDd{\beta}{j}$. Then
  $\omega \in \oDd{\beta}{l}$ for $j \leq l$.
\end{rem}
\begin{rem}
	\label{rem:nb_puiseux}
	Consider a form $\omega \in \DD{\beta}{m}{j}$
        (resp. $\DD{\beta,\leq}{m}{j}$, $\DD{\beta,<}{m}{j}$) and $l$ such
        that $j \leq l$ and $\Beg_{\Gamma}(\omega) \leq
        \beta_{l}$. Then we have
        $\omega\in$ $\oDD{\Beg_{\Gamma}(\omega)}{m^{\prime}}{l}$
        (resp. $\oDD{\Beg_{\Gamma}(\omega),\leq}{m^{\prime}}{l}$,
        $\oDD{\Beg_{\Gamma}(\omega),<}{m^{\prime}}{l}$) where
        $m^{\prime} = m$ if $\Beg_{\Gamma}(\omega) = \beta$ and
        $m^{\prime} =0$ if $\Beg_{\Gamma}(\omega) > \beta$.
\end{rem}
\begin{rem}
  \label{rem:unique_b_j}
  Clearly, $\Dd{\beta}{j} \cap \Dd{\beta'}{j} = \emptyset$ if
  $\beta \neq \beta'$ so that the leading variable of
  $g \in \cup_{\beta \in {\mathcal E}_{\Gamma}} \Dd{\beta}{j}$ is unique.
\end{rem}  
\begin{rem}
  It is possible to have
  $\Dd{\beta}{j} \cap \Dd{\beta'}{l} \neq \emptyset$ with
  $\beta \neq \beta^{\prime}$ and $j \neq l$ (cf. Remark
  \ref{rem:nb_puiseux}).
\end{rem}  
\begin{rem}
  \label{rem:unique_b}
  We have
  $\odd{\beta} \cap \odd{\beta^{\prime}}=
  \emptyset$ if $\beta \neq \beta^{\prime}$ by Remarks
  \ref{rem:family_restriction} and \ref{rem:unique_b_j}.
\end{rem}

One of the main relevant properties of the $1$-forms in
$\Dd{\beta,\leq}{j}$, is that in the most important cases, they
\touch{} $\Gamma$ to some order, and have a unique companion
curve, as the following result shows. 
Notice, in the proof, the importance of property \ref{it:lead-k} in Definition \ref{def:max} (leading variable).
\begin{pro}
\label{pro:unique_companion}
Fix $j \in {\mathbb N}_{g}$ and
$\beta_{j} \leq \beta \leq \beta_{j+1}$.  Consider a $1$-form
$\omega \in \DD{\beta,\leq}{m}{j}$, and  suppose that either $m=0$ or
$\Beg_{\Gamma}(\omega) > \beta_{j}$. Then
$\Beg_{\Gamma}(\omega) \leq \beta_{j+1}$, and the form $\omega$
\touches{} $\Gamma$ to order $ \Beg_{\Gamma}(\omega)$ and has a unique
formal companion curve. In particular, if $\Beg_{\Gamma}(\omega)$ is a
Puiseux exponent, then $\omega\in \tDd{\beta}{j}$.
\end{pro}
\begin{rem}
  Geometrically, the statement says (but not only this) that 
  there is a separatrix $\gamma$ of $\omega$ such that if $D$ is the last exceptional divisor in the resolution of singularities of $\Gamma$ with $P_{D,\gamma}= P_{D,\Gamma}$ (the infinitely near points of $\gamma$ and $\Gamma$ in $D$ are equal), then there is no other separatrix of $\omega$ meeting $D$ at $P_{D,\Gamma}$.
\end{rem}
\begin{proof}
  Assume $\Beg_{\Gamma}(\omega) < \infty$ or $j <g$ since otherwise
  the result is obvious.  Denote
  $\overline{\beta} = \min (\beta_{j+1}, \Beg_{\Gamma}(\omega))$.

  We start by claiming that $a_{\beta, {\Gamma}^{\prime}}^{m} \neq 0$ for any
  ${\Gamma}^{\prime} \in \fgb{\Gamma}{\overline{\beta}}$. It is obvious if $m=0$ so we can
  assume $m >0$ (notice that, as $a_{\beta}^0=1$, we set $0^0=1$ too, when specifying
  to $\Gamma^{\prime}$).  This implies that $\beta = \beta_{j}$ by Definition
  \ref{def:max}.  Since $\Beg_{\Gamma}(\omega) > \beta_{j}$ by hypothesis, we get
  $\beta_{j} < \overline{\beta}$. As for any
  ${\Gamma}^{\prime} \in \fgb{\Gamma}{\overline{\beta}}$, we have
  $a_{\beta_{j}, \Gamma^{\prime}} = a_{\beta_{j}, \Gamma}\neq 0$, the claim holds.

  We now consider the equations
  $C_{{\mathfrak n}^{r} (\overline{\beta})} = 0$ ($r \in {\mathbb Z}_{\geq 0}$)
  for the family $\fgb{\Gamma}{\overline{\beta}}$ (cf. equation
  \eqref{equ:maximal}). The claim, and the fact that the
  $c$-sequences of elements in $\DD{\beta,\leq}{m}{j}$ contain no zero
  element, imply that the coefficient of
  $a_{{\mathfrak n}^{r} (\overline{\beta})}$ in
  $C_{{\mathfrak n}^{r} (\overline{\beta})}$ never vanishes in
  $\fgb{\Gamma}{\overline{\beta}}$ for any
  $r \in {\mathbb Z}_{\geq 0}$.  Thus, the method of undetermined
  coefficients lets us compute, term by term, starting with
  the equation $C_{\overline{\beta}}=0$, a unique element
  $\overline{\Gamma}$ of $\fgb{\Gamma}{\overline{\beta}}$ that is
  invariant by $\omega=0$.

  In order to complete the proof, it suffices to show that
  $\overline{\Gamma}$ is equisingular to $\Gamma_{< \overline{\beta}}$
  (cf. Definition \ref{def:truncation}). Consider the notations in
  Definition \ref{def:max} and denote
  \begin{equation*}
    \ell^{\prime} =
    \max (\{0\} \cup  \{ r \geq 1 : \beta_{r} < \overline{\beta} \}).
  \end{equation*}
  To prove the equisingularity, we only need to show that if
  $\beta^{\prime} \not \in (e_{\ell^{\prime}})$, then
  $a_{\beta^{\prime}, \overline{\Gamma}} = 0$, as we already know that $a_{\beta_l,\overline{\Gamma}}\neq 0$ 
  for $\beta_l<\overline{\beta}$.

  Before proceeding, let us prove that $k \in (e_{\ell^{\prime}})$. If
  $m=0$, then, as $\ell \leq \ell^{\prime}$, we have, from Definition
  \ref{def:cong_n} and because $e_{\ell^{\prime}}$ divides $e_{\ell}$:
  \begin{equation*}
    [k]_{e_{\ell}} =
    [0]_{e_{\ell}} \implies [k]_{e_{\ell^{\prime}}} =
    [0]_{e_{\ell^{\prime}}} \implies k \in (e_{\ell^{\prime}}),
  \end{equation*}
  Thus, we can assume $m \neq 0$ and hence
  $\beta = \beta_{j}$, so that
  $j \leq \ell^{\prime}$, as $\beta_{j} < \overline{\beta}$. Since
  $\beta_{j} \in (e_{j}) \subset (e_{\ell^{\prime}}) $, we obtain, again by Definition~\ref{def:cong_n}:
  \begin{equation*}
    [k]_{e_{j-1}} =
    [m \beta_{j}]_{e_{j-1}} \implies [k]_{e_{\ell^{\prime}}} =
    [m \beta_{j}]_{e_{\ell^{\prime}}}  \implies [k]_{e_{\ell^{\prime}}} =
    [0]_{e_{\ell^{\prime}}}.
  \end{equation*}
  Thus, we obtain $k\in(e_{\ell^{\prime}})$.
  
  We can prove now that if $\beta^{\prime} \not \in (e_{\ell^{\prime}})$,
  then $a_{\beta^{\prime}, \overline{\Gamma}} = 0$. Assume, aiming at a
  contradiction, that there exists $\beta^{\prime} \not \in (e_{\ell^{\prime}})$
  such that $a_{\beta^{\prime}, \overline{\Gamma}} \neq 0$. Take $\beta^{\prime}$
  minimal with such property, and notice that $\beta' \geq \overline{\beta}$.  Thus,
  $\overline{\Gamma}_{<\beta^{\prime}}$ and $\Gamma_{< \overline{\beta}}$ are
  equisingular by definition of $\beta'$.  Therefore
  $t (\overline{\Gamma}_{<\beta^{\prime}})^{*} \omega$ is of the form
  $h(t^{e_{\ell^{\prime}}}) dt$ for some $h \in {\mathbb C}[[t]]$. However, by
  construction, $\nu_{0} (h(t^{e_{\ell^{\prime}}})) = k + \beta^{\prime}$, and as
  $k\in(e_{\ell^{\prime}})$, we infer that
  $k+\beta^{\prime} \not \in (e_{\ell^{\prime}})$, which is the
  desired contradiction.

  Notice that if $\overline{\beta} \in \{\beta_{j}, \beta_{j+1}\}$ then,
  as $\overline{\beta} \not \in  (e_{\ell^{\prime}})$,
  we must have
  $a_{\overline{\beta}, \overline{\Gamma}} = 0$ and thus
  $Q_{\overline{\beta}} (\overline{\Gamma}) =0$ because
  $C_{\overline{\beta}} (\overline{\Gamma})  =0$.
  
  Let us now show that $\Beg_{\Gamma}(\omega)\leq \beta_{j+1}$. Assume the contrary,
  which implies that $\overline{\beta}=\beta_{j+1}$, so that
  $Q_{\beta_{j+1}}(\Gamma)=Q_{\beta_{j+1}}(\overline{\Gamma})=0$, from which we get
  $C_{\beta_{j+1}}(\Gamma) \neq 0$, i.e. $\Beg_{\Gamma}(\omega)\leq \beta_{j+1}$,
  a contradiction.  Thus, $\Beg_{\Gamma}(\omega) \leq \beta_{j+1}$, and by
  definition, $\overline{\beta} = \Beg_{\Gamma}(\omega)$.  We conclude that
  $Q_{\Beg_{\Gamma}(\omega)} (\Gamma) =0$ if $\Beg_{\Gamma}(\omega)$ is a Puiseux
  exponent. This is exactly what makes $\omega\in\tDd{\beta}{j}$, which completes
  the proof.
\end{proof}

\subsection{Technical results on leading variables of products}
\label{subsec:leading_products}
As described in Algorithm \ref{alg:construction}, in our project of finding a
${\mathbb C}[[x]]$-basis of $1$-forms for $\Gamma$, we have $g+1$ levels, from the $0$-level to
the $g$-level. Given $0 \leq l < g$, and a terminal family $\mathcal{T}_{\Gamma,l}$,
we build $ {\mathcal G}_{\Gamma, l+1}^{0}$ as in the \textbf{for} loop of that Algorithm, using
products of the forms in $\mathcal{T}_{\Gamma,l}$ with powers of $f_{l+1}$. If we know the
leading variables of the forms in $\mathcal{T}_{\Gamma,l}$, we shall be able to control the leading
variables of these products. This is the aim of this technical subsection. Moreover, as will become apparent later, 
all these results are needed just for the setup of the \textbf{while} loop in Algorithm \ref{alg:construction}, 
but their importance in paramount.


\begin{lem}
\label{lem:prod}
Fix $1 \leq j \leq l \leq g$ and $a \in {\mathbb Z}_{\geq 1}$. Consider 
\begin{equation*}
  g_{r} \in \oDD{\beta_{l}, <}{0}{j} \cap \tDd{\beta_{l}}{j} \ \  \mathrm{and} \  \
  g_{a+s} \in  \cup_{\beta < \beta_{l}} \oDd{\beta}{j}
\end{equation*}  
for $r \in {\mathbb N}_{a}$ and $s \in {\mathbb N}_{b}$. Then the function
$g:= \prod_{r=1}^{a+b} g_r$ belongs to $\oDD{\beta_{l}, <}{a-1}{j} \cap \tDd{\beta_{l}}{j}$.
Moreover, the $c$-sequence of $g$ is the sum of the $c$-sequences of $g_1, \ldots, g_a$ except
its first term, which is equal to $1$.
\end{lem}
\begin{proof}
  In \cite[Lemma 3.2]{ayuso2024construction}, the result is proved for the family
  $\fgb{\Gamma}{\beta_1}$, i.e. for $j=1$.  The other cases are derived from Remarks
  \ref{rem:family_restriction} and \ref{rem:nb_puiseux}.
\end{proof}  
We can even determine the leading variable of the approximate roots $f_{\Gamma, l}$ and their powers.
\begin{cor}
\label{cor:root}
Consider $l \in {\mathbb N}_{g}$ and $m \in {\mathbb Z}_{\geq 1}$. Then $f_{\Gamma, l}^{m}$
belongs to $\oDD{\beta_{l},\leq }{m-1}{l} \cap \tDd{\beta_{l}}{l}$ and $1, m, m, \ldots$ is its
$c$-sequence.
\end{cor}
\begin{proof}
  Consider $f_{l}$, which is the $l$-th approximate root of the curves in
  $\fgba{\Gamma}{\beta_{1}}$, so that
  $f_{l} \in {\mathbb C}[a_{\beta_1}, \hdots, a_{{\mathfrak p}(\beta_{l})}
  ]\{x,y\}$.  On one hand, we know \cite[Corollary 3.1]{ayuso2024construction} that
  $f_{l}^{m} \in \oDD{\beta_{l}, \leq}{m-1}{1} \cap \tDd{\beta_{l}}{1}$. On the
  other, $f_{l}$ restricted to $\fgba{\Gamma}{\beta_{l}}$ is
  equal to $f_{\Gamma, l}$ (cf. Equation \eqref{equ:app_root_exp}). Now Remarks
  \ref{rem:family_restriction} and \ref{rem:nb_puiseux} give
  $f_{\Gamma, l}^{m} \in \oDD{\beta_{l}, \leq}{m-1}{l} \cap
  \tDd{\beta_{l}}{l}$.
\end{proof}
Later we will see that the $1$-forms in the family $ {\mathcal G}_{\Gamma, l}^{0}$ have a
leading variable by applying the next corollaries.
\begin{cor}
\label{cor:f_dominates_omega}
Let $l \in {\mathbb N}_{g}$, $m \geq 1$ and $\omega \in \cup_{\beta < \beta_{l}} \oDd{\beta}{j}$. 
Then $f_{\Gamma, l}^{m} \omega$ belongs to $\oDD{\beta_{l},\leq }{m-1}{l} \cap \tDd{\beta_{l}}{l}$
and $1, m, \ldots, m, \ldots$ is its $c$-sequence.
\end{cor}
The next result is a corollary of Lemma \ref{lem:prod} and Proposition \ref{pro:unique_companion}.
\begin{cor}
\label{cor:f_omega}
Let $l \in {\mathbb N}_{g}$, $m \geq 0$ and $\omega \in \oDD{\beta_{l}, \leq}{0}{l}$. 
Then $f_{\Gamma, l}^{m} \omega$ belongs to $\oDD{\beta_{l},\leq }{m}{l} \cap \tDd{\beta_{l}}{l}$
and $1, c_{{\mathfrak n}(\beta_{l})} + m, c_{{\mathfrak n}^{2}(\beta_{l})} +m, \ldots$ is its $c$-sequence
where $(c_{{\mathfrak n}^{r}(\beta_{l})} )_{r \geq 0}$ is the $c$-sequence of $\omega$.
\end{cor}
\subsection{Main types of foliations we will encounter}
\label{subsec:foliation}
Along the way, while carrying out the operations in Algorithm
\ref{alg:construction}, we will construct different types of $1$-forms giving rise
to foliations of two main types. One type will be in
$\DD{\beta_{j}, \leq}{n_{j}-1}{ j}$ for $j\in \mathbb{N}_g$, while the other one
will be in the following singularly specific set. In this subsection, the remarkable
geometric relation with $\Gamma$ of these two types of foliation is described.
 \begin{defi}
  \label{def:Dj}
We define $\hatDD{{\mathfrak n}(\beta_{j}), <}{0}{j}$ 
as the set consisting of $\omega \in \DD{{\mathfrak n}(\beta_{j}), <}{0}{j}$
such that the polynomial $R$ associated to $\omega$ and the family $\fgb{\Gamma}{\beta_{j}}$
in Definition \ref{def:max} is of the form 
$c a_{\beta_{j}}^{s}$ where $c \in {\mathbb C}^{*}$ and $s \in {\mathbb Z}_{\geq 0}$.
\end{defi}

As we are studying geometric relations between $1$-forms and the singular curve
$\Gamma$, it is natural to introduce the blow-up of points and the sequence of infinitely near
points of a curve.
\begin{defi}
\label{def:seq_inf_pt}
Let $\pi_{1}: {\mathcal X}_{1} \to {\mathbb C}^{2}$ be the blow-up of the origin
$P_{0}:=(0,0)$.  Let $\Gamma_1 =\overline{\pi_{1}^{-1} (\Gamma \setminus \{P_{0}\})}$ be the
strict transform of $\Gamma$ by $\pi_1$.  We define the divisor
$E_{1}= D_{1}= \pi_{1}^{-1}(P_0)$.  Obviously, $D_{1} \cap \Gamma_{1}$ is a singleton
$\{ P_{\Gamma,1} \}$.
By iterating this method, we obtain a sequence 
\begin{equation}\label{eq:desing-Gamma}
  (\mathbb{C}^2,0) \stackrel{\pi_1}{\longleftarrow} \mathcal{X}_1
  \stackrel{\pi_2}{\longleftarrow} \cdots
  \mathcal{X}_{k-1}\stackrel{\pi_k}{\longleftarrow}\mathcal{X}_k
  \stackrel{\pi_{k+1}}{\longleftarrow}\cdots
\end{equation}
where $\pi_{k}$ is the blow-up of $P_{\Gamma, k-1}$,
$D_{k} = \pi_{k}^{-1} (P_{\Gamma, k-1})$,
$E_k = \pi_{k}^{-1} (E_{k-1})$, $\Gamma_k$ is the strict transform of
$\Gamma_{k-1}$ by $\pi_k$ and $P_{\Gamma, k}$ is the unique point in
$\Gamma_{k} \cap D_k$ for any $k \in {\mathbb Z}_{\geq 1}$. The
sequence ${(P_{\Gamma,k})}_{k \geq 1}$ is called the sequence of {\it
  infinitely near points} of $\Gamma$.
\end{defi}
\begin{defi}
  We say that $P \in E_{\iota}$ is a {\it trace point} if there is a unique irreducible
  component of $E_{\iota}$ through $P$.
\end{defi}
\begin{rem}
\label{rem:bij_trace_exp}
There exists a bijective correspondence between
\begin{equation*}
  \{\iota \in {\mathbb Z}_{\geq 1} :
  P_{\Gamma, \iota} \ {\rm is \ trace \ point \ of} \ E_{\iota} \}
\end{equation*}
and the set ${\mathcal E}_{\Gamma} \cup n {\mathbb Z}_{\geq 1}$: given
$\beta \in {\mathcal E}_{\Gamma} \cup n {\mathbb Z}_{\geq 1}$, consider the family of curves
\begin{equation*}
  \Gamma_{\zeta}(t) =
  \left( t^{n},  \zeta t^{\beta}  +
    \sum_{\beta^{\prime}\in \mathcal{E}_{\Gamma}\cup n\mathbb{Z}_{\geq 1}}
    a_{\beta^{\prime}, \Gamma} t^{\beta^{\prime}}
  \right) .
\end{equation*}
for $\zeta$ in a neighborhood of $0$. Clearly, the $\iota$ associated to $\beta$ is the first
index $m \in {\mathbb Z}_{\geq 1}$ such that $\zeta \mapsto P_{\Gamma_{\zeta},m}$ depends on
$\zeta$.  Define a function $\hat{\theta}$ as $\beta = \hat{\theta} (\iota)$ and
$\iota = \hat{\theta}^{-1} (\beta)$.  For simplicity, we denote
${\mathfrak D}_j = D_{\hat{\theta}^{-1} (\beta_j)}$ for $1 \leq j \leq g$.  The divisors
${\mathfrak D}_1, \hdots, {\mathfrak D}_g$ are called {\it characteristic divisors}
(cf. \cite{Fortuny-Ribon-Canadian}).
\end{rem}
\begin{rem}
\label{rem:coef_dir}
Fix $\beta \in {\mathcal E}_{\Gamma} \cup n {\mathbb Z}_{\geq 1}$, and let
$\iota = \hat{\theta}^{-1} (\beta)$.  Consider a curve $\Gamma_{\eta}$ of parametrization
\begin{equation*}
  \Gamma_{\eta} (t)=
  \left( t^{n},  \sum_{\beta' < \beta} a_{\beta', \Gamma} t^{\beta'}  + \eta t^{\beta}  
    + \sum_{\substack{\beta' \in {\mathcal E}_{\Gamma} \cup n {\mathbb Z}_{\geq 1}\\ \beta' > \beta}}
    c_{\beta'}  t^{\beta'}\right) .
\end{equation*}
Then $\tau: \eta \mapsto P_{\Gamma_{\eta}, \iota}$ is a map that does not depend on
$c_{\beta'}$ for $\beta' > \beta$.  Let us consider the case where $c_{\beta'}=0$
for any $\beta' > \beta$.  First, if $\beta\not\in\{\beta_1,\ldots,\beta_g\}$, then
$\tau$ is a surjective map between ${\mathbb C}$ and the trace points of $D_{\iota}$
in $E_{\iota}$. Second, if $\beta=\beta_j$ with $j \in \mathbb{N}_{g}$, then $\tau$
is a surjective map between $\mathbb{C}^{\ast}$ and the trace points of $D_{\iota}$.
Moreover, $\tau (\eta) = \tau (\eta')$ if and only if $\Gamma_{\eta}$ and
$\Gamma_{\eta'}$ are parametrizations of the same curve. This equality holds if and
only if there exists $\xi \in {\mathbb C}^{*}$ such that
$\Gamma_{\eta} (\xi t) \equiv \Gamma_{\eta'} (t)$.  It is easy to see that, in the
first case ($\beta$ is not a Puiseux exponent) this implies
$\eta =\eta'$, so that $\tau$ is injective.  In the second case
($\beta=\beta_j$ is a Puiseux exponent), such a property is
equivalent to $\xi^{e_{j-1}} =1$ and $\eta' = \xi^{\beta_{j}} \eta$, which gives,
for this second case:
\begin{equation*}
  \tau(\eta) = \tau (\eta')\,\Longleftrightarrow\, (\eta' \eta^{-1})^{n_j} =1
\end{equation*}
because $\{ \xi^{\beta_j} : \xi^{e_{j-1}} =1 \}$ is the set of $n_j$-roots of unity.
\end{rem}
\subsubsection{Geometry of the $1$-forms in $\hatDD{{\mathfrak n}(\beta_{j}), <}{0}{j}$}
The $1$-forms in $\hatDD{{\mathfrak n}(\beta_{j}), <}{0}{j}$ have the following key
geometric relation with $\Gamma$, part of which is that they are always dicritical at a very specific divisor:
\begin{pro}
\label{pro:jG}
Fix $j \in {\mathbb N}_{g}$, and $\omega$ in $\hatDD{{\mathfrak n}(\beta_{j}), <}{0}{j}$. Then
$\omega$ fans $\Gamma$ to genus $j$.  Moreover, the fanning family associated to
$\omega$ and the strict transforms of its members are transverse to the
characteristic divisor ${\mathfrak D}_{j}$ of $\Gamma$.
\end{pro}
\begin{proof}
  We claim that, under the hypothesis, given any $\eta \in {\mathbb C}^{*}$, there exists a
  unique family
  ${(a_{\beta}(\eta))}_{\substack{\beta \in {\mathcal E}_{\Gamma}\\ \beta >  \beta_j}}$ 
  of complex numbers such that the formal curve $\Gamma_{\eta}$ of
  parametrization
  \begin{equation*}
    \Gamma_{\eta} (t)=
    \left( t^{n},  \sum_{\beta < \beta_j} a_{\beta, \Gamma} t^{\beta}  + \eta t^{\beta_{j}} +
      \sum_{\beta \in {\mathcal E}_{\Gamma}, \ \beta > \beta_j} a_{\beta} (\eta) t^{\beta}\right)
  \end{equation*}
  is invariant by $\omega =0$. By Definition \ref{def:Dj} there exists
  $s \in {\mathbb Z}_{\geq 0}$ such that the coefficient of $a_{{\mathfrak n}^{r} (\beta_{j})}$
  in $C_{{\mathfrak n}^{r} (\beta_{j})}$ is equal to $d_{r} a_{\beta_{j}}^{s}$, where
  $d_{r} \in {\mathbb C}^{*}$, for any $r \in {\mathbb Z}_{>0}$.  Thus, fixed
  $\eta \in {\mathbb C}^{*}$, the equations $C_{{\mathfrak n}^{r} (\beta_{j})} =0$, for
  $r \geq 1$, determine the sequence
  $(a_{\beta})_{\substack{\beta \in {\mathcal E}_{\Gamma}\\ \beta > \beta_j}}$ so that 
  $\Gamma_{\eta}$ is an invariant curve for $\omega$.  Moreover, for any
  $\beta > \beta_{j}$, the function $a_{\beta}: {\mathbb C}^{*} \to {\mathbb C}$ is
  holomorphic.

  Consider $k$, $m$ and $\ell$ as in Definition \ref{def:max}. Since $m=0$ and $\ell = j$ by
  hypothesis, it follows that $k \in (e_{j})$. The same reasoning as in the proof
  of  Proposition \ref{pro:unique_companion}, shows that $a_{\beta} \equiv 0$ if
  $\beta \not \in (e_j)$.
 
  Denote $\iota = \hat{\theta}^{-1} (\beta_{j})$ and $P_{\eta} = P_{\Gamma_{\eta}, \iota}$.
  Since $a_{\beta}(\eta)=0$ whenever $\beta \not \in (e_j)$, we deduce that the strict
  transform of $\Gamma_{\eta}$ by $\pi_{1} \circ \ldots \circ \pi_{\iota}$ is smooth and
  transverse to ${\mathfrak D}_{j}$ at $P_{\eta}$ for any $\eta \in {\mathbb C}^{*}$,
  so that $\mathfrak{D}_j$ is not invariant for the strict transform of $\omega$.  The
  non-invariance of ${\mathfrak D}_{j}$ implies that there exists a finite subset $S$ of
  ${\mathbb C}^{*}$ such that the power series $\Gamma_{\eta}(t)$ converges for any
  $\eta \in {\mathbb C}^{*} \setminus S$.  Moreover, $(t, \eta) \mapsto \Gamma_{\eta}(t)$ is a
  holomorphic map defined in a neighborhood of $\{ 0 \} \times ({\mathbb C}^{*} \setminus S)$
  in ${\mathbb C} \times {\mathbb C}^{*}$. The modulus maximum theorem,
  applied to the coefficients $a_{\beta}$, implies that $(t, \eta) \mapsto \Gamma_{\eta}(t)$ is
  holomorphic in a neighborhood of $\{ 0 \} \times {\mathbb C}^{*}$. Therefore,
  ${(\Gamma_{\eta})}_{\eta \in {\mathbb C}^{*}}$ is a family of invariant curves for $\omega$
  fanning $\Gamma$ to genus $j$.  The family ${(P_{\eta})}_{\eta \in {\mathbb C}^{*}}$ consists
  of the trace points of ${\mathfrak D}_j$.
\end{proof}
\begin{rem} \label{rem:unique_fan}
Consider the family ${(\Gamma_{\eta})}_{\eta \in {\mathbb C}^{*}}$ of invariant curves
of $\omega=0$ fanning
$\Gamma$ to genus $j$ and provided by Proposition \ref{pro:jG}.
The curves $\Gamma_{\eta}$ and $\Gamma_{\eta'}$ 
satisfy $P_{\Gamma_{\eta}, \iota} = P_{\Gamma_{\eta'}, \iota}$
if and only if 
$(\eta' \eta^{-1})^{n_j} =1$ by Remark \ref{rem:coef_dir}.
Hence, the invariance property implies that $\Gamma_{\eta}$ and $\Gamma_{\eta'}$ 
coincide if and only if $(\eta' \eta^{-1})^{n_j} =1$.
\end{rem}
\begin{rem} \label{rem:max_inter} All the curves in the families
  ${(\Gamma_{\eta})}_{\eta \in {\mathbb C}^{*}}$ above, and in
  ${(\Gamma_{j, \eta^{\prime}})}_{\eta^{\prime} \in {\mathbb C}^{*}}$ (cf. Equation
  \eqref{equ:G_j_s}) are equisingular. By Remark \ref{rem:unique_fan} and Noe\-ther's
  formula (cf. \cite[Theorem 3.3.1]{Casas}), given given $\xi \in \mathbb{C}^{*}$,
  the curve $\Gamma_{\xi}$ is the only one in the family fanning $\Gamma$ to
  genus $j$ for $\omega$, whose intersection multiplicity
  $(\Gamma_{\xi}, \Gamma_{j, \xi})$, at the origin, is maximal.
%
%
%
%
\end{rem}
By Proposition \ref{pro:unique_companion}, we already knew that $\omega$ has a
unique formal companion curve. Proposition \ref{pro:jG} ensures that this companion
curve is also analytic:
\begin{cor}
\label{cor:wh0}
Let $\omega \in \hatDD{{\mathfrak n}(\beta_{j}), <}{0}{j}$.  Then $\omega$
\touches{} $\Gamma$ to order $ \Beg_{\Gamma}(\omega)$ and has a unique
analytic companion curve.
\end{cor}  
\subsubsection{Geometry of the $1$-forms in $\DD{\beta_{j}, \leq}{n_{j}-1}{ j}$}
The following proposition gives a thorough description of the geometric relation
between the $1$-forms in $\DD{\beta_{j}, \leq}{n_{j}-1}{ j}$, which is the other
main family we shall encounter, and $\Gamma$. The condition on
$\Beg_{\Gamma}(\omega)$ will be satisfied in most of the cases, so it is not really
restrictive.
\begin{pro}
  \label{pro:wu}
  Fix $j \in {\mathbb N}_{g}$. Let $\omega \in \DD{\beta_{j}, <}{n_{j}-1}{j}$ with
  $\Beg_{\Gamma}(\omega)> \beta_{j}$.  Then $\omega$ \touches{} $\Gamma$ to order
  $ \Beg_{\Gamma}(\omega)$ and has a unique analytic companion curve
  $\overline{\Gamma}$.  Moreover, ${\mathfrak D}_j$ is invariant by the strict
  transform ${\mathfrak F}$ of the foliation $\omega =0$ by
  $\pi_{1} \circ \ldots \circ \pi_{\iota}$ where
  $\iota =\hat{\theta}^{-1}(\beta_{j})$.  All trace points of ${\mathfrak D}_j$ are
  regular points of ${\mathfrak F}$ except
  $P_{\Gamma, \iota}=P_{\overline{\Gamma}, \iota}$, which is an elementary (or
  equivalently log-canonical) singularity of ${\mathfrak F}$.
\end{pro}
\begin{proof}
  We divide the proof in several parts for the reader's convenience.
  
  \strut

  \textbf{1. The form $\omega$ accompanies $\Gamma$ to order
    $\Beg_{\Gamma}(\omega)$}. This fact, except for the analyticity of the
  companion curve $\overline{\Gamma}$, is a consequence of Proposition
  \ref{pro:unique_companion}.  Recall the notations of Definition \ref{def:max}.
  The solutions of $C_{\beta_{j}}=0$ for the family $\fgb{\Gamma}{\beta_j}$ are the
  complex numbers $\xi a_{\beta_{j}, \Gamma}$ where $\xi$ is a $n_j$-th root of
  unity, so they correspond to the same trace point $P_{\Gamma, \iota}$ of
  ${\mathfrak D}_j$ (cf. Remark \ref{rem:coef_dir}).
  
  \strut
  
  \textbf{2. The divisor $\mathfrak{D}_j$ is invariant for
    $\omega$.} Let $T$ be the set of trace points of ${\mathfrak D}_j$, and
  ${\mathfrak C}$ the set of formal curves $\Gamma' (t) \in \fgb{\Gamma}{\beta_j}$
  such that $\Gamma_{\iota}^{\prime}$ is smooth and transverse to ${\mathfrak D}_j$
  at $P_{\Gamma', \iota} \in T$.  Since $\Beg_{\Gamma}(\omega) \leq \beta_{j+1}$ by
  Proposition \ref{pro:unique_companion} (and Property \ref{it:lead-k} in
  Definition \ref{def:max}), any curve
  $\Gamma^{\prime}\in\fgb{\Gamma}{\beta_j}$ is equisingular to the companion curve
  $\overline{\Gamma}$.  Let ${\mathfrak C}_{0}\subset \mathfrak{C}$ be the subset of
  curves $\Gamma^{\prime}$ such that $P_{\Gamma', \iota} \neq P_{\Gamma, \iota}$.
  Notice that $\Beg_{\Gamma}(\omega)> \beta_{j}$ implies
  $\overline{\Gamma} \in {\mathfrak C} \setminus {\mathfrak C}_{0}$.  Finally, set
  \begin{equation*}
    {\mathfrak C}_{1} = \{ \Gamma' (t) \in {\mathfrak C} \setminus
    {\mathfrak C}_{0} : C_{\beta_j+e_j} (\Gamma') \neq 0 \}
  \end{equation*}
  Let $k$ be as in Definition
  \ref{def:max}. By construction, given $\Gamma' \in {\mathfrak C}$ we
  have
  \begin{equation}
    \label{equ:gen_pt}
    \nu_{\Gamma'} (\omega) = (k+\beta_j)/e_j \Leftrightarrow
    \Gamma' \in {\mathfrak C}_{0} 
  \end{equation}
  because $\Gamma'(t)$ is of the form $\gamma (t^{e_j})$ where $\gamma (t)$ is an
  irreducible pa\-ra\-me\-trization (cf. Remark \ref{rem:order}).
  Equation \eqref{equ:gen_pt} implies that $\mathfrak{F}$ has no invariant
  curves transverse to $\mathfrak{D}_j$ at points in
  $T \setminus \{ P_{\Gamma, \iota} \}$; thus, ${\mathfrak D}_j$ is
  ${\mathfrak F}$-invariant.
  
  \strut

  \textbf{3. $T \setminus \{ P_{\Gamma, \iota} \}$ has only regular points
    for $\mathfrak{F}$}. By an argument similar to the above, 
    given $\Gamma' \in {\mathfrak C}$ we have
  \begin{equation}
    \label{equ:sp_pt}
    \nu_{\Gamma'} (\omega)
    = 1+(k+\beta_j)/e_j \Leftrightarrow \Gamma' \in {\mathfrak C}_{1}  . 
  \end{equation}
 Take a trace point $P \in T$ and consider local
  coordinates $(x', y')$ centered at $P$. 
  As $\mathfrak{D}_j$ is $\mathfrak{F}$-invariant, there exists
  some $m_{\iota}\in \mathbb{N}$ such that
  \begin{equation*}
    (\pi_1 \circ \ldots \circ \pi_{\iota})^{*} \omega = (x')^{m_{\iota}} \omega' =
    (x')^{m_{\iota}} (a(x',y') dx' + x' b(x',y') dy')
  \end{equation*}
  where $x'=0$ is a local equation of ${\mathfrak D}_j$ and
  $a(0, y') \not \equiv 0$.  Given $\Gamma' (t) \in {\mathfrak C}$ with
  $P_{\Gamma', \iota}$ in a neighborhood of $P$, we have
  $\nu (x' \circ \Gamma_{\iota}' (t)) = e_j$ where $\Gamma_{\iota}' (t)$ is the
  irreducible parametrization of $\Gamma^{\prime}_{\iota}$ in
  coordinates $(x',y')$.  Thus, we have $\nu_{\Gamma'} (\omega) = m_{\iota}+1$ if
  and only if $a (P_{\Gamma', \iota} ) \neq 0$, and by Equation
  \eqref{equ:gen_pt} we deduce $m_{\iota}+1 = (k+\beta_j)/e_j$.  Equations
  \eqref{equ:gen_pt} and \eqref{equ:sp_pt} imply that $a (P_{\Gamma', \iota}) =0$
  if and only if $P_{\Gamma', \iota} = P_{\Gamma, \iota}$.  Therefore
  $T \setminus \{ P_{\Gamma, \iota} \}$ consists of regular points of
  ${\mathfrak F}$.

\strut
 
  \textbf{4. $P_{\Gamma,\iota}$ is an elementary singularity of
    $\mathfrak{F}$, and $\overline{\Gamma}$ is analytic.}
  For any $\Gamma^{\prime}\in \mathfrak{C}_{1}$, we have 
  \begin{equation*}
    \nu_{\Gamma_{\iota}'} (\omega') =
    \nu_{\Gamma^{\prime}_{\iota}} (\omega)  - m_{\iota}
    =  1+(k+\beta_j)/e_j  - (-1 +(k+\beta_j)/e_j) =2 .
  \end{equation*}
  Since there are two $\omega^{\prime}$-invariant curves through
  $P_{\Gamma, \iota}$, namely $x'=0$ and $\overline{\Gamma}_{\iota}$, it follows
  that $\omega^{\prime} (P_{\Gamma,\iota}) =0$.  Moreover, since
  $\nu_{\Gamma_{\iota}'} (\omega') =2$ for any
  $\Gamma^{\prime} \in {\mathfrak C}_1$, we deduce that $\omega^{\prime}$ has
  non-vanishing linear part.  
There exists $\eta \in {\mathbb C}^{*}$ such that for any 
$\Gamma^{\prime} \in \fgb{\Gamma}{\mathfrak{n} (\beta_j)}$
there is a reparametrization $\phi_{\Gamma'} (s) = \eta s + h.o.t.$ satisfying
\[ \Gamma_{\iota}' \circ \phi_{\Gamma'} (s) = 
\left(s^{e_{j}}, \sum_{l \in {\mathcal E}_{\Gamma}, \ l > \beta_{j}} b_{l, \Gamma'} s^{l - \beta_{j}} \right) \]
in coordinates $(x',y')$. Moreover, there exists $\rho \in{\mathbb C}^{*}$ such that 
for any $l \in {\mathcal E}_{\Gamma}$, with  $l > \beta_{j}$, 
the map $\Gamma' \mapsto b_{l, \Gamma'} - \rho \eta^{l} a_{l, \Gamma'}$, 
defined on  $\fgb{\Gamma}{\mathfrak{n} (\beta_j)}$, 
depends just on $a_{k, \Gamma'}$ for $\beta_{j} <k <l$, or equivalently on $b_{k, \Gamma'}$ for $\beta_{j} <k <l$.
We have
\[  (\Gamma_{\iota}' \circ \phi_{\Gamma'})^{*} ( (\pi_1 \circ \ldots \circ \pi_{\iota})^{*} \omega) = (\phi_{\Gamma'})^{*} 
((\Gamma')^{*} \omega) = 
 \left(  s^{k} \sum_{r=1}^{\infty} D_{{\mathfrak n}^{r} (\beta_{j})} s^{{\mathfrak n}^{r} (\beta_{j})} \right) \frac{ds}{s},  \]
for $\Gamma' \in \fgb{\Gamma}{\mathfrak{n} (\beta_j)}$. We obtain 
\[ D_{{\mathfrak n}^{r} (\beta_{j})} = \eta^{k +{\mathfrak n}^{r}(\beta_{j})}
  \tilde{c} c_{{\mathfrak n}^{r}(\beta_{j})} a_{\beta_{j}, \Gamma}^{n_{j}-1}
  a_{{\mathfrak n}^{r} (\beta_j)} + Q_{{\mathfrak n}^{r} (\beta_j)} \] for all
$\Gamma' \in \fgb{\Gamma}{\mathfrak{n} (\beta_j)}$ and $r \in {\mathbb Z}_{\geq 1}$,
cf. Definition \ref{def:max}.  Thus, there exists $c' \in {\mathbb C}^{*}$ such that
\begin{equation}
\label{equ:lvs}
D_{{\mathfrak n}^{r} (\beta_{j})}    =
c' \tilde{c} c_{{\mathfrak n}^{r}(\beta_{j})} a_{\beta_{j}, \Gamma}^{n_{j}-1} 
b_{{\mathfrak n}^{r} (\beta_j)}  + Q_{{\mathfrak n}^{r} (\beta_j)}'  
\end{equation}
for all $\Gamma' \in \fgb{\Gamma}{\mathfrak{n} (\beta_j)}$ and
$r \in {\mathbb Z}_{\geq 1}$, where $Q_{{\mathfrak n}^{r} (\beta_j)}'$ depends on
$b_{k, \Gamma'}$ for $\beta_{j} <k <{\mathfrak n}^{r} (\beta_j)$.

Consider the linear part $(\mu_1 x + \mu_2 y) dx + \lambda x dy$ of $\omega'$.  We
claim that $\lambda \neq 0$. Otherwise, since $j^{1} \omega' \neq 0$ and $\omega'=0$
has 2 transverse smooth invariant curves, it follows that $\mu_{2} \neq 0$.  We
deduce that the coefficient of $b_{{\mathfrak n}^{r} (\beta_j)}$ in expression
\eqref{equ:lvs} is constant for $r \geq 1$, contradicting that
${(c_{{\mathfrak n}^{r}(\beta_{j})})}_{r \geq 1}$ is a strictly increasing sequence
(recall that we are assuming $\omega\in \DD{\beta_j,<}{n_j-1}{j}$).

Since the tangent space of $\overline{\Gamma}_{\iota}$ is an invariant linear
subspace of
\begin{equation}
\label{equ:lvf}
\lambda x \frac{\partial }{\partial x} -
(\mu_1 x + \mu_2 y) \frac{\partial}{\partial y}
\end{equation} 
associated to the eigenvalue $\lambda \neq 0$, it follows that
$\overline{\Gamma}_{\iota}$ is a germ of analytic curve at $P_{\Gamma, \iota}$. Thus
$\overline{\Gamma}$ is a germ of analytic curve at the origin \cite[Proposition
3.2]{LRRS:stable}.  Since the linear vector field \eqref{equ:lvf} is non-nilpotent,
$P_{\Gamma, \iota}$ is an elementary singularity of ${\mathfrak F}$.
\end{proof}
\begin{rem}
Consider the notations in the fourth point of the proof of Proposition \ref{pro:wu}.
  Since the coefficient of $s^{{\mathfrak n}^{r}(\beta_{j})- \beta_{j}+e_{j}-1} ds$ in
  \[ (\Gamma_{\iota}' \circ \phi_{\Gamma'})^{*}  [ \lambda^{-1} \mu_2 y dx + x dy] \]
  is equal to 
  \[ b_{{\mathfrak n}^{r}(\beta_{j}), \Gamma'} ({\mathfrak n}^{r}(\beta_{j})-\beta_{j} + e_{j} \lambda^{-1} \mu_2), \]
  the sequences ${({\mathfrak n}^{r}(\beta_{j})-\beta_{j} + e_{j} \lambda^{-1} \mu_2)}_{r \geq 1}$ and
  ${(c_{{\mathfrak n}^{r}(\beta_{j})})}_{r \geq 1}$ 
  are equal up to a common 
  multiplicative factor $\rho$. We deduce $ \lambda^{-1} \mu_{2} \in {\mathbb Q}$ as a consequence of 
  $\omega \in \DD{\beta_{j}, <}{n_{j}-1}{j}$. In particular, we get $\rho \in {\mathbb Q}^{+}$
  and hence the former sequence consists of 
  positive rational numbers. This implies $e_{j} +  e_{j} \lambda^{-1} \mu_2 >0$ and in particular
  $- \lambda^{-1} \mu_{2} \in {\mathbb Q}_{<1}$.
\end{rem}

\section{Construction of the ${\mathbb C}[[x]]$-basis}
\label{sec:construction_cx_bases}
At this point, we have all the tools required for the construction of a
$\mathbb{C}[[x]]$-basis, which we divide into $g+1$ levels: from what we call the
$0$-level to the $g$-level, as explained in Algorithm \ref{alg:construction}.  Each
level will consist of a number of stages. Recall that we are assuming $\Gamma$ has
the parametrization \eqref{equ:param}, so that $y=0$ is has maximal contact with it
among the set of smooth curves. The starting point is the \emph{terminal
  family} $\mathcal{T}_0$ of the $0$-level, defined as the pair:
\begin{equation*}
  {\mathcal T}_{\Gamma, 0} = ( dx, dy )
\end{equation*}
Assuming we have constructed a terminal family
${\mathcal T}_{\Gamma, l}$ of the $l$-level (which will be defined later) with
$l < g$, we construct the stage $0$ family ${\mathcal G}_{\Gamma, l+1}^{0}$ of the
$l+1$-level using Equation \eqref{equ:ini_fam}, by means of the 
$l+1$-th
approximate root of $\Gamma$. Then, as detailed in Algorithm
\ref{alg:construction}, we then construct a sequence of families
\begin{equation*}
  {\mathcal G}_{\Gamma, l+1}^{1}, \ldots, {\mathcal G}_{\Gamma, l+1}^{s}
  = {\mathcal T}_{\Gamma, l+1}
\end{equation*}
which will be \emph{good} (another term to be defined) until we find the terminal
family ${\mathcal T}_{\Gamma, l+1}$ of level $l+1$. At level $g$, we shall obtain
the final terminal family ${\mathcal T}_{\Gamma, g}$, from which we shall
extract the ${\mathbb C}[[x]]$-basis.

The properties of \emph{good} families are stable under the \textsc{Main
  Transformation} of Algorithm~\ref{alg:construction}, which is what allows us to
iterate the operation until a terminal family is found at each
level.

\subsection{Distinguished families of $1$-forms} 
\label{subsec:families}
This subsection contains the definitions of the main types of families
(terminal, good, etc.) of $1$-forms which we shall need to carry out
the inductive process.
\begin{defi}
  \label{def:nice}
  We say that a family $\mathbf{\Omega}= (\Omega_1, \hdots, \Omega_{2 \nu_l} )$ in
  ${\Omega}^{1} (\mathbb{C}^2,0)$ is a {\it\nice{} family of level}
  $0 \leq l \leq g$ (for $\Gamma$) if $\mathbf{\Omega}$ is a free basis of the
  ${\mathbb C}[[x]]$-module ${\mathcal M}_{ 2 \nu_{l}}$, and there exists a subset
  $\low\subset{\mathbb N}_{2 \nu_{l}}$ of $\min(2 \nu_{l},n)$ elements such
  that $(\nu_{\Gamma} (\Omega_j))_{j \in \low}$ defines pairwise different
  classes modulo $n$ and
\begin{equation} \label{cond4}
   \nu_{\Gamma} (\Omega_k) > \max \{ \nu_{\Gamma} (\Omega_j) : j \in \low \} 
\end{equation} 
 for any $k \in {\mathbb N}_{2 \nu_{l}} \setminus \low$.
\end{defi}
\begin{rem}
  Notice that $\low=\mathbb{N}_{2\nu_l}$ whenever $l<g$, so that condition
  \eqref{cond4} is only relevant for $l=g$.
\end{rem}
The interest of separated
families is that, at level $g$, they provide $\mathbb{C}[[x]]$-bases:
\begin{pro}
  Consider a separated family of level $g$ for $\Gamma$. Then
  $(\Omega_j)_{j \in \low}$ is a ${\mathbb C}[[x]]$-basis of $\Gamma$.
\end{pro}
\begin{proof}
  Let $\omega \in \hat{\Omega}^{1} (\mathbb{C}^2,0)$.  Since
  \begin{equation*}
    {\mathcal M}_{ 2 \nu_{g}} + (f_{\Gamma, g+1})
    \hat{\Omega}^{1} (\mathbb{C}^2,0) = \hat{\Omega}^{1} (\mathbb{C}^2,0),
  \end{equation*}
  $\omega$ must be of the form
  \begin{equation*}
    \omega =  \sum_{1 \leq j \leq 2n} a_{j}(x) \Omega_j  + f_{\Gamma, g+1} \omega',
  \end{equation*}
  where $a_1, \hdots, a_{2n} \in {\mathbb C}[[x]]$ and
  $\omega' \in \hat{\Omega}^{1} (\mathbb{C}^2,0)$.  Take $\omega$ such that
  $\nu_{\Gamma} (\omega)$ is minimal among the values
  $\nu_{\Gamma} (\hat{\omega})$ for
  $\hat{\omega} \in \hat{\Omega}^{1} (\mathbb{C}^2,0)$, and such that
  $\nu_{\Gamma} (\omega) - \nu_{\Gamma} (\hat{\omega}) \in (n)$.  It suffices to
  show that $\nu_{\Gamma} (\omega) = \nu_{\Gamma} (\Omega_l)$ for some
  $k \in \low$.

  Obviously, there exists $l \in \low$ with
  $\nu_{\Gamma} (\omega) - \nu_{\Gamma} (\Omega_l) \in (n)$.  Since
  $\nu_{\Gamma} (\omega) \leq \nu_{\Gamma} (\Omega_l)$, Equation \eqref{cond4}
  implies that $ \nu_{\Gamma} (a_{l}(x) \Omega_l)$ is the minimum of
  \begin{equation*}
    \{ \nu_{\Gamma} (a_{j}(x) \Omega_j) : 1 \leq j \leq 2n \}.
  \end{equation*}
  Since
  $\nu_{\Gamma} (\omega) = \nu_{\Gamma} (a_{l}(x) \Omega_l) \geq \nu_{\Gamma}
  (\Omega_l)$, we obtain $\nu_{\Gamma} (\omega) = \nu_{\Gamma} ( \Omega_l)$.
\end{proof} 
The definition of \emph{terminal family} includes conditions 
on $\Gamma$-orders and leading variables
which will allow us to argue inductively.
 \begin{defi}
 \label{def:level}
 We say that ${\bf \Omega} = ( \Omega_1, \ldots, \Omega_{2 \nu_{l}} )$ is a {\it
   terminal family of level} $l$ if
\begin{enumerate}[label=(\alph*)]
 \item \label{ta}  ${\bf \Omega}$ is a separated family of level $l$,
 \item \label{tb} $\nu_{\Gamma} (\Omega_{j}) \leq \overline{\beta}_{l+1}$ for any
   $j \in {\mathbb N}_{2 \nu_{l}}$, where $\overline{\beta}_{g+1}=\infty$,
 \item \label{tc}
   ${\bf \Omega} \subset \bigcup_{\beta' \in {\mathcal E}_{\Gamma}}
   \oDd{\beta',<}{l+1}$, and, specifically, we have
   \begin{equation*}
     {\bf \Omega} \subset \oDD{\beta_{l+1}, <}{0}{ l+1} \bigcup
     \cup_{\beta' \in {\mathcal E}(\Gamma), \ \beta' < \beta_{l+1}} \oDd{\beta',
       <}{l+1}
   \end{equation*}
   if $0 \leq l <g$.
 \end{enumerate}
\end{defi}
\begin{rem} \label{rem:ov_imp_t}
Notice that $\oDD{\beta_{l+1}, <}{0}{ l+1} \subset  \tDd{\beta_{l+1}}{l+1}$
by Proposition \ref{pro:unique_companion} (or its Corollary \ref{cor:f_omega}).
It is probably the most important use of such result being fundamental to 
obtain ${\mathcal T}_{\Gamma, l+1}$ from ${\mathcal T}_{\Gamma, l}$.
\end{rem}
We shall need to study the set
$\nu_{\Gamma} ({\bf \Omega}) = \{ \nu_{\Gamma}(\Omega_j): j \in {\mathbb N}_{2
  \nu_{l}} \}$ when $\mathbf{\Omega}$ is a terminal family. The following
result is more general, as it covers all the possible types of $1$-forms we shall
encounter in our construction, not just the terminal ones. Property \ref{it:lead-k}
of Definition \ref{def:max} is essential, as the reader will notice.
\begin{lem}
  \label{lem:cong_nu}
  Let $\Omega \in \hat{\Omega}^{1} \cn{2}$. Then:
  \begin{itemize}
  \item If $\Omega \in \oDd{\beta', <}{l+1}$ for some
    $\beta' \in {\mathcal E}(\Gamma)$ with $\beta' < \beta_{l+1}$, then
    \begin{equation*}
      \nu_{\Gamma} (\Omega) \in (e_{l}) .
    \end{equation*}
  \item If $\Omega \in \oDD{\beta_{l+1}, <}{0}{ l+1}$, then
    \begin{equation*}
      \nu_{\Gamma} (\Omega) \in [\beta_{l+1}]_{e_{l}} .
    \end{equation*}
  \item If
    $\Omega \in \Dd{\beta_{l+1}, <}{l+1} \cup \hatDD{{\mathfrak n}(\beta_{l+1}),
      <}{0}{l+1}$, then
    \begin{equation*}
      \nu_{\Gamma} (\Omega) \in [\Beg_{\Gamma}(\Omega)]_{e_{l+1}} .
    \end{equation*}
  \end{itemize}
\end{lem}
\begin{proof}
  Consider the notations in Definition \ref{def:max}.  Suppose we are in the first
  case.  We have $[k]_{e_{\ell}} = [m \beta']_{e_{\ell}}$ for some
  $m \in {\mathbb Z}_{\geq 0}$, $\beta' < \beta_{l+1}$ and $\ell \leq l$. Since
  $\beta' \in (e_l)$, $k$ belongs to $(e_l)$.  We get
  $\nu_{\Gamma} (\Omega) = k + \beta' \in (e_l)$.
  
  Assume that  $\Omega \in \oDD{\beta_{l+1}, <}{0}{ l+1}$. Thus, we obtain 
 $[k]_{e_{l}} = [0]_{e_{l}}$ by definition and hence $k \in (e_l)$. This implies 
 $[\nu_{\Gamma} (\Omega_j)]_{e_l} = [k + \beta_{l+1}]_{e_l} =  [\beta_{l+1}]_{e_{l}}$.

  Suppose $\Omega \in \Dd{\beta_{l+1}, <}{l+1}$. We have $[k]_{e_{l}} = [m \beta_{l+1}]_{e_l}$ for some 
 $m \in {\mathbb Z}_{\geq 0}$. Thus $k$ belongs to $(e_{l+1})$. 
 Since $\nu_{\Gamma} (\Omega) = k + \Beg_{\Gamma}(\Omega)$,  it follows that $\nu_{\Gamma} (\Omega) \in [\Beg_{\Gamma}(\Omega)]_{e_{l+1}}$.
  
  Suppose $\Omega \in \hatDD{{\mathfrak n}(\beta_{l+1}), <}{0}{ l+1}$. We have
  $[k]_{e_{l+1}} = [0]_{e_{l+1}}$. We conclude as in the previous case. 
  \end{proof}
  \begin{cor}
  \label{cor:cong_term}
  Let ${\bf \Omega} = ( \Omega_1, \ldots, \Omega_{2 \nu_{l}} )$ be a terminal family of level $0 \leq l < g$.
  Then $\nu_{\Gamma} ({\bf \Omega}) \subset (e_l) \cup [\beta_{l+1}]_{e_l}$. In particular, we have
 
\begin{equation*} [\nu_{\Gamma} ({\bf \Omega}) ]_{n} =  [(e_l)]_{n} \cup [\beta_{l+1} + (e_l)]_{n} . \end{equation*}

  \end{cor} 
 \begin{proof}
   The inclusion is a consequence of Condition \ref{tc} and Lemma
   \ref{lem:cong_nu}.  Since ${\bf \Omega}$ is separated and $l <g$, the set
   $[\nu_{\Gamma} ({\bf \Omega}) ]_{n}$ has $2 \nu_{l}$ elements.  As a consequence
   of $n \in (e_l)$ and $\beta_{l+1} \not \in (e_l)$, we deduce that
 \begin{equation*}
   \sharp [(e_l)]_{n} = \sharp [\beta_{l+1} + (e_l)]_{n} = \frac{n}{e_{l}} =
   \nu_l \ \ \mathrm{and} \ \
   [(e_l)]_{n} \cap [\beta_{l+1} + (e_l)]_{n} = \emptyset.
 \end{equation*}
 As $[\nu_{\Gamma} ({\bf \Omega}) ]_{n}$ is contained in
 $ [(e_l)]_{n} \cup [\beta_{l+1} + (e_l)]_{n}$ and has the same cardinal, they
 coincide.
\end{proof}
The inner loop in Algorithm \ref{alg:construction} requires the following
quite technical definition in order to reason inductively. Assume we have a
terminal family $\mathcal{T}_{\Gamma,l}=(\Omega_1, \hdots, \Omega_{2 \nu_{l}})$ of
level $l$ (which he have at level $0$, with $\mathcal{T}_{\Gamma,0}=(dx,dy)$, as
explained above). The inductive stages will consist of constructing \emph{good}
families iteratively:
\begin{defi}
\label{def:und}
Denote
\[
  \overline\upper_{\Gamma, l+1} = \{ j + 2 (n_{l+1}-1) \nu_l : j\in \mathbb{N}_{2\nu_l} \land \Omega_j \in
  \oDD{\beta_{l+1}, <}{0}{l+1} \}
\] 
and
$\upper_{l+1} = \{ j + 2 (n_{l+1}-1) \nu_l : j \in {\mathbb N}_{2
  \nu_{l}} \}$. Notice that this set does not depend on $\Gamma$: it
is just the last $2\nu_l$ elements of $\mathbb{N}_{2\nu_{l+1}}$.  The
former set is the part of $\upper_{l+1}$ ``coming from'' forms
$\Omega_{j}$ whose leading variable is equal to $\beta_{l+1}$.

In the next definition and everywhere else, we assume implicitly that
$\infty-d=\infty+d=\infty$ for any $d\in \mathbb{Z}$.
\end{defi}
\begin{defi}
\label{def:good}
We say that ${\bf \Omega} = (\Omega_{1}^{s}, \hdots, \Omega_{2 \nu_{l+1}}^{s})$ is
a {\it good family of level $l+1$ at stage $s \geq 1$} if the following properties
are satisfied:
\begin{enumerate}[label=(\Roman*)]
\item \label{good1} ${\mathcal T}_{\Gamma, l} \subset {\bf \Omega}$ and more
  precisely $\Omega_{j}^{s} = \Omega_{j}$ for any $j \in {\mathbb N}_{2 \nu_{l}}$;
  our construction keeps the previous terminal family intact at all
  stages;
\item \label{good2} ${\bf \Omega}$ is a free basis of the ${\mathbb C}[[x]]$-module
  ${\mathcal M}_{ 2 \nu_{l+1}}$;
\item \label{good6} the following inclusion holds:
  \begin{multline*}
    {\bf \Omega} \subset
    \cup_{\beta^{\prime} \leq \beta_{l+1}}
    \oDd{\beta^{\prime}, <}{l+1}
    \bigcup \DD{\beta_{l+1},<}{n_{l+1}-1}{l+1}\\
    \bigcup
    \hatDD{{\mathfrak n}(\beta_{l+1}), <}{0}{ l+1}
  \end{multline*}
  (notice the parallelism between this condition, Definition
  \ref{def:level}, and Lemma \ref{lem:cong_nu});
\item \label{good3} for any
  $j \in {\mathbb N}_{2 \nu_{l+1}} \setminus \overline\upper_{\Gamma, l+1}$ with
  $\Beg_{\Gamma}(\Omega_{j}^{s}) > \beta_{l+1}$, we have
  \begin{equation*}
    \Omega_{j}^{s} \in \hatDD{{\mathfrak n}(\beta_{l+1}), <}{0}{ l+1};
  \end{equation*}
\item \label{good4} for any
  $j \in \overline\upper_{\Gamma, l+1}$, we have
  \begin{equation*}
    \Omega_{j}^{s} \in \DD{\beta_{l+1},<}{n_{l+1}-1}{l+1}
  \end{equation*}
  (the last two conditions emphasize the relevance of the set
  $\overline\upper_{\Gamma,l+1}$);
\item \label{good5} there is a partition of ${\mathbb N}_{2 \nu_{l+1}}$ in two sets
  $\low_{\Gamma, l+1}^{s}$, and $\upper_{\Gamma, l+1}^{s}$, each containing $\nu_{l+1}$
  elements, such that ${\mathbb N}_{2 \nu_{l}} \subset \low_{\Gamma, l+1}^{s}$, and
  $\upper_{l+1} \subset \upper_{\Gamma, l+1}^{s}$, satisfying:
 \begin{itemize}
 \item
   $\Beg_{\Gamma}(\Omega_{j}^{s}) < {\mathfrak n}^{s} (\beta_{l+1}) \leq
   \beta_{l+2}$ if $j \in \low_{\Gamma, l+1}^{s}$ and
 \item
   $\beta_{l+2} \geq \Beg_{\Gamma}(\Omega_{j}^{s}) \geq {\mathfrak n}^{s}
   (\beta_{l+1}) $ if $j \in \upper_{\Gamma, l+1}^{s}$,
 \end{itemize}
 where $\beta_{g+1} = \infty$; (thus, at stage $s$, there are $\nu_{l+1}$
 elements $\Omega$ in $\mathbf{\Omega}$ with
 $\Beg_{\Gamma}(\Omega)<\mathfrak{n}^s(\beta_{l+1})$, and other $\nu_{l+1}$ elements with
 $\Beg_{\Gamma}(\Omega)\geq \mathfrak{n}^{s}(\beta_{l+1})$ ---the other bounds are of less
 importance);
\item \label{good8}
  $0 \leq \nu_{\Gamma} (\Omega_{k}^{s}) - \nu_{\Gamma} (\Omega_{k}^{0}) \leq
  \Beg_{\Gamma}(\Omega_{k}^{s}) - \Beg_{\Gamma}(\Omega_{k}^{0})$ for any
  $k \in {\mathbb N}_{2 \nu_{l+1}}$ and the second inequality is an equality if
  $k \in \upper_{l+1}$ (thus, the \textsc{Main Transformation} in
  the inner loop increases the values of $\Beg_{\Gamma}$ at least as much as the
  values of $\nu_{\Gamma}$);
 \item \label{good7} Denote by
   $\upper_{\Gamma, l+1}^{s,<\infty} = 
   \{ j \in \upper_{\Gamma, l+1}^{s} : \Beg_{\Gamma}(\Omega_{j}^{s}) \neq \infty \}$.  
   Then
   \begin{equation*}
     \{ [\nu_{\Gamma}(\Omega_{j}^{s}) - \Beg_{\Gamma}(\Omega_{j}^{s})]_{n}
     : j \in \upper_{\Gamma, l+1}^{s,<\infty}  \}
   \end{equation*}
   is a subset of cardinal $\sharp \upper_{\Gamma, l+1}^{s,<\infty}$ of
   $[(e_{l+1})]_{n}$ and
   \begin{equation*} \{ [\nu_{\Gamma}(\Omega_{j}^{s})]_{n} : j \in \low_{\Gamma,
       l+1}^{s} \} = [(e_{l+1})]_{n} .
   \end{equation*}
   In particular
   $[\nu_{\Gamma}(\Omega_{j}^{s})]_{n} \neq [\nu_{\Gamma}(\Omega_{k}^{s})]_{n}$ if
   $j \neq k$ and $j, k \in \low_{\Gamma, l+1}^{s}$; 
   (notice that the value $\nu_{\Gamma}(\Omega_j^s)-\Beg_{\Gamma}(\Omega_j^s)$ 
   is the exponent $k$ in Definition \ref{def:max}, so that all those exponents are 
   pairwise different modulo $n$ in the family $\upper_{\Gamma,l+1}^{s,<\infty}$).
  \end{enumerate}
\end{defi}
\begin{rem}
  Notice that $\overline\upper_{\Gamma,l+1}^{s,<\infty}=\upper_{\Gamma,l+1}^s$ for
  $l+1<g$ due to condition \ref{good5}.
\end{rem}
\subsection{From stage $0$ to stage $1$ in level $l+1$}
\label{subsec:from_0_to_1}
The family $\mathcal{G}_{l+1}^{0}$ of stage $0$ constructed in Algorithm
\ref{alg:construction} is not good in general: there 
are properties in Definition \ref{def:good} that are not satisfied. 
For instance, there are $\nu_l$ elements $\Omega$ in  
$\mathcal{G}_{l+1}^{0}$ with $\Beg_{\Gamma}(\Omega_{j}^{0}) < \beta_{l+1}$ 
whereas the other $2 \nu_{l+1} - \nu_{l}$
satisfy $\Beg_{\Gamma}(\Omega_{j}^{0}) = \beta_{l+1}$ 
and thus Condition \ref{good5} does not hold. 
This precludes using it as the
base step in the induction reasoning. However, a careful analysis shows that,
despite that, the \textsc{Main Transformation} applied to $\mathcal{G}_{l+1}^0$
does in fact produce a good family. This is why the passage from stage $0$ to
stage $1$ requires a specific study, which we tackle in this subsection.

Consider a terminal family
${\mathcal T}_{\Gamma, l} = (\Omega_1,\ldots \Omega_{2 {\nu}_l})$ of level
$0 \leq l < g$.  Recall that the initial family
${\mathcal G}_{\Gamma, l+1}^{0} = (\Omega_{1}^{0}, \ldots, \Omega_{2
  \nu_{l+1}}^{0})$ of level $l+1$ is defined using the $l+1$-th approximate root of
$\Gamma$, by the formula
\begin{equation*}
  \Omega^0_{2a\nu_l+b} =
  f_{\Gamma, l+1}^a\Omega_b,\ \mathrm{for}\ 0\leq a < n_{l+1}
  \ \mathrm{and}\
 1\leq b \leq 2 {\nu}_l .
\end{equation*}
Since ${\mathcal T}_{\Gamma, l}$ is a free basis of the ${\mathbb C}[[x]]$-module
${\mathcal M}_{2 \nu_{l}}$ and $f_{\Gamma, l+1}$ is a monic polynomial on $y$ of
degree $\nu_{l}$ (cf. Remark \ref{rem:app_roots}), it follows that
${\mathcal G}_{\Gamma, l+1}^{0}$ is a free basis of the ${\mathbb C}[[x]]$-module
${\mathcal M}_{2 \nu_{l+1}}$.  Moreover, $\Omega_{j}^{0} = \Omega_{j}$ holds for
any $j \in {\mathbb N}_{2 \nu_{l}}$.  
The next remarks are a consequence of
Remark \ref{rem:family_restriction} and Corollaries \ref{cor:f_dominates_omega} and
\ref{cor:f_omega}.
\begin{rem}
 \label{rem:f_dominates_omega} 
 We have
 $\ \Omega_{2 a \nu_{l}+j}^{0} \in \oDD{\beta_{l+1}, \leq}{a-1}{l+1}$
 if $a\geq 1$ and $\Beg_{\Gamma}(\Omega_{j}) < \beta_{l+1}$.
\end{rem}
\begin{rem}
\label{rem:f_omega}
We have
$\ \Omega_{2 a \nu_{l}+j}^{0} \in \oDD{\beta_{l+1}, <}{a}{l+1}$ if
$\Beg_{\Gamma}(\Omega_{j}) = \beta_{l+1} $.
\end{rem}
The idea of the construction of a good family ${\mathcal G}_{\Gamma, l+1}^{1}$ of
level $l+1$ from $\mathcal{G}^0_{\Gamma,l+1}$ is to modify the initial family to
make it closer to be separated. In order to carry out this step, we need to identify
the pairs $\Omega_{j}^{0}, \Omega_{k}^{0}$ with $j \neq k$ such that
$\nu_{\Gamma} (\Omega_{j}^{0}) - \nu_{\Gamma} (\Omega_{k}^{0}) \in (n)$. 
The use of the approximate roots, and more precisely that 
$l+1$-th root has contact $\overline{\beta}_{l+1}$ with $\Gamma$,
is the key to the following result.
\begin{pro}
\label{pro:conflicting_pairs}
We  have $[\nu_{\Gamma} ({\mathcal G}_{\Gamma, l+1}^{0})]_{n} = [(e_{l+1})]_n$.
Let $h \in [(e_{l+1})]_{n}$. Then, the set 
\begin{equation*}
  S_h = \{ d \in {\mathbb N}_{2 \nu_{l+1}} : [\nu_{\Gamma}(\Omega_{d}^{0})]_{n} = h \}
\end{equation*}  
has exactly two elements $j, k$ (with $j < k$) and
$\nu_{\Gamma} (\Omega_{j}^{0}) < \nu_{\Gamma} (\Omega_{k}^{0})$ holds. Moreover,
if $\Omega_{j}^{0} = f_{\Gamma, l+1}^{r_a} \Omega_a$ and
$\Omega_{k}^{0} = f_{\Gamma, l+1}^{r_b} \Omega_b$, then  either
\begin{enumerate}
\item \label{emv} $r_b=r_a+1$, $ \nu_{\Gamma} (\Omega_b) \in (e_l)$ and
 $\nu_{\Gamma} (\Omega_a) \in [\beta_{l+1}]_{e_l}$ or
\item \label{dmv} $r_b=n_{l+1} -1$, $r_a=0$,
  $ \nu_{\Gamma} (\Omega_b) \in [\beta_{l+1}]_{e_l}$ and
  $ \nu_{\Gamma} (\Omega_a) \in (e_l)$.
\end{enumerate}
\end{pro}
\begin{proof}
  We have $[\nu_{\Gamma} ({\mathcal T}_{\Gamma, l})]_{n} \subset [(e_{l+1})]_n$ by
  Corollary \ref{cor:cong_term}.  We deduce that
  $[\nu_{\Gamma} ({\mathcal G}_{\Gamma, l+1}^{0})]_{n} \subset [(e_{l+1})]_n$ since
  $\nu_{\Gamma} (f_{\Gamma, l+1}) = \overline{\beta}_{l+1} \in (e_{l+1})$.

  Suppose that $\Omega_{j}^{0} = f_{\Gamma, l+1}^{r_a} \Omega_a$ and
  $\Omega_{k}^{0} = f_{\Gamma, l+1}^{r_b} \Omega_b$ satisfy $j <k$ and
  $[\nu_{\Gamma}(\Omega_{j}^{0})]_{n} = [\nu_{\Gamma}(\Omega_{k}^{0})]_{n}$.
  This implies
  \begin{equation*}
    (r_b-r_a) \overline{\beta}_{l+1} + (\nu_{\Gamma}(\Omega_{b})
    - \nu_{\Gamma}(\Omega_{a})) \in (n) .
  \end{equation*}
  We have
  $\nu_{\Gamma}(\Omega_{a}), \nu_{\Gamma}(\Omega_{b}) \in (e_l) \cup
  [\beta_{l+1}]_{e_l}$ by Corollary \ref{cor:cong_term}. We claim that one of them
  belongs to $(e_l)$ and the other one to $[\beta_{l+1}]_{e_l}$.  Otherwise,
  $\nu_{\Gamma}(\Omega_{b}) - \nu_{\Gamma}(\Omega_{a})$ belongs to $(e_l)$ and hence
  $(r_b-r_a ) \overline{\beta}_{l+1} \in (e_l)$.  Since $0 \leq r_b-r_a < n_{l+1}$,
  we obtain $r_b = r_a $ by Equation \eqref{equ:rec_beta} in
  Remark \ref{rem:prop_beta}.  But then
  $[\nu_{\Gamma}(\Omega_{a})]_{n} = [\nu_{\Gamma}(\Omega_{b})]_{n}$, so that $a=b$,
  since $l<g$ and a terminal family is separated. In particular, $j=k$, which is a
  contradiction.

  First, assume
  $\nu_{\Gamma}(\Omega_{a}) \in [\beta_{l+1}]_{e_l} =
  [\overline{\beta}_{l+1}]_{e_l}$, so that $\nu_{\Gamma}(\Omega_{b}) \in (e_l)$.
  This gives $(r_b-r_a-1) \overline{\beta}_{l+1} \in (e_l)$, 
  which by Remark \ref{rem:prop_beta} implies that
  $r_b=r_a+1 > r_a$. Since
  $\nu_{\Gamma} (\Omega_a) \leq \overline{\beta}_{l+1}$ by Definition
  \ref{def:level} (terminal family), we get
  $\nu_{\Gamma} (\Omega_{j}^{0}) < \nu_{\Gamma} (\Omega_{k}^{0})$.

  Now, assume $\nu_{\Gamma}(\Omega_{a}) \in (e_l)$,
  $\nu_{\Gamma}(\Omega_{b}) \in [\overline{\beta}_{l+1}]_{e_l}$.  Since
  $1 \leq r_b-r_a+1 \leq n_{l+1}$, it follows that $r_b-r_a+1 = n_{l+1}$ and hence
  $r_b=n_{l+1}-1$ and $r_a=0$.  Since $r_b>r_a$, we also deduce
  $\nu_{\Gamma} (\Omega_{j}^{0}) < \nu_{\Gamma} (\Omega_{k}^{0})$.

  The fact that there is no other $\Omega_{p}^{0} = f_{\Gamma, l+1}^{r_c} \Omega_c$
  with $p \not \in \{j,k\}$ such that
  $[\nu_{\Gamma}(\Omega_{j}^{0})]_{n} = [\nu_{\Gamma}(\Omega_{p}^{0})]_{n}$ holds
  because either $\nu_{\Gamma}(\Omega_{a}) - \nu_{\Gamma}(\Omega_{c}) \in (e_l)$ or
  $\nu_{\Gamma}(\Omega_{b}) - \nu_{\Gamma}(\Omega_{c}) \in (e_l)$. Therefore,
  $S_{h}$ contains at most $2$ elements for any $h \in [(e_{l+1})]_{n}$.

  Finally, as $\sharp [(e_{l+1})]_n = \nu_{l+1}$, and we know that
  $\sharp {\mathcal G}_{\Gamma, l+1}^{0} = 2 \nu_{l+1}$,
  $[\nu_{\Gamma} ({\mathcal G}_{\Gamma, l+1}^{0})]_{n} \subset [(e_{l+1})]_n$ and
  $\sharp S_{h} \leq 2$ for any $h \in [(e_{l+1})]_{n}$, we get
  $\sharp S_{h} = 2$ for any $h \in [(e_{l+1})]_{n}$, and
  $[\nu_{\Gamma} ({\mathcal G}_{\Gamma, l+1}^{0})]_{n} = [(e_{l+1})]_n$.
\end{proof}
Proposition~\ref{pro:conflicting_pairs} has bearing on the relation 
between the elements $\Omega_j$ and $\Omega_k$, 
and their leading variables, and those of $\Omega_a$ and $\Omega_b$:
\begin{lem}
  \label{rem:conflicting_pairs}
Let $\Omega_{j}^{0} = f_{\Gamma, l+1}^{r_a} \Omega_a$,
$\Omega_{k}^{0} = f_{\Gamma, l+1}^{r_b}\Omega_b$ satisfy the condition
$[\nu_{\Gamma}(\Omega_{j}^{0})]_{n} =
[\nu_{\Gamma}(\Omega_{k}^{0})]_{n}$.  Consider both cases in
Proposition \ref{pro:conflicting_pairs}. We have 
\begin{enumerate}
\item If $\nu_{\Gamma}(\Omega_b)\in(e_l)$, then:
  \begin{itemize}
  \item $\Omega_{a} \in \oDD{\beta_{l+1}, <}{0}{ l+1}$,
\item $\Omega_{b} \in \oDd{\beta', <}{l+1}$ for some
    $\beta' < \beta_{l+1}$,
    \item $\Omega_{j}^{0} \in \oDD{\beta_{l+1}, <}{r_{a}}{ l+1}$, and
    \item $\Omega_{k}^{0} \in \oDD{\beta_{l+1}, \leq}{r_{a}}{ l+1}$;
  \end{itemize}
\item If $\nu_{\Gamma}(\Omega_b)\in[\beta_{l+1}]_{e_l}$ then:
  \begin{itemize}
  \item $\Omega_{b} \in \oDD{\beta_{l+1}, <}{0}{ l+1}$,
    \item $\Omega_{a} =\Omega_{j}^{0} \in \oDd{\beta', <}{l+1} $ for some
    $\beta' < \beta_{l+1}$,
\item $\Omega_{k}^{0} \in \oDD{\beta_{l+1}, <}{n_{l+1}-1}{ l+1}$
  \end{itemize}
\end{enumerate}
\end{lem}
\begin{proof}
The properties of $\Omega_a$ and $\Omega_b$ are a consequence of 
Definition \ref{def:level} and Lemma \ref{lem:cong_nu}.
The results on $\Omega_{j}^{0}$ and $\Omega_{k}^{0}$ 
are derived from Remarks \ref{rem:f_dominates_omega} and \ref{rem:f_omega}.
\end{proof}
We are now in position to define the family of stage $1$:
\begin{defi}
\label{def:good1}
Given the family
${\mathcal G}_{\Gamma, l+1}^{0} = (\Omega_{1}^{0}, \ldots, \Omega_{2
  \nu_{l+1}}^{0})$, of stage $0$ above, we define the family of stage $1$, that is
${\mathcal G}_{\Gamma, l+1}^{1} = (\Omega_{1}^{1}, \ldots, \Omega_{2
  \nu_{l+1}}^{1})$, as follows: for any $h \in [(e_{l+1})]_{n}$, there
are exactly two $\Omega_{j}^{0}$, $\Omega_{k}^{0}$ with $j <k$ such
that
$[\nu_{\Gamma}(\Omega_{j}^{0})]_{n} =
[\nu_{\Gamma}(\Omega_{k}^{0})]_{n} = h$.  We set
\begin{equation*}
  \Omega_{j}^{1} = \Omega_{j}^{0} \ \ \mathrm{and} \ \  \Omega_{k}^{1}
  = \Omega_{k}^{0} - c x^{d} \Omega_{j}^{0}
\end{equation*}
where
$\nu_{\Gamma}(\Omega_{k}^{0}) - \nu_{\Gamma}(\Omega_{j}^{0}) = dn$ and
$c$ is the unique element of ${\mathbb C}^{*}$ such that
$\nu_{\Gamma} (\Omega_{k}^{0} - c x^{d} \Omega_{j}^{0}) > \nu_{\Gamma}
(\Omega_{k}^{0})$.
 \end{defi}
\begin{rem}
\label{rem:stage1_free}
Note that $d$ in Definition \ref{def:good1} always belong to
${\mathbb Z}_{+}$ by Proposition \ref{pro:conflicting_pairs}.
Therefore, ${\mathcal G}_{\Gamma, l+1}^{1}$ is still a free basis of
${\mathcal M}_{2 \nu_{l+1}}$.
\end{rem} 
Our next goal is to show that ${\mathcal G}_{\Gamma, l+1}^{1}$ is a good family of
level $l+1$ at stage $1$.  The next two alternative lemmas are key in the
proof. Notice how Proposition \ref{pro:conflicting_pairs} and Lemma
\ref{rem:conflicting_pairs} are essential tools in both proofs.
\begin{lem}
 \label{lem:non_dom}
 Suppose that
 $[\nu_{\Gamma} (\Omega_{j}^{0})]_{n}=[\nu_{\Gamma} (\Omega_k^0)]_n$
 for $j < k$, and that
 $\Beg_{\Gamma}(\Omega_{j}^{0}) = \Beg_{\Gamma}(\Omega_{k}^{0})
 =\beta_{l+1}$.  Then the following properties hold:
 $\Omega_{j}^{1} = \Omega_{j}^{0}$,
 $\Omega_{k}^{1} \in \hatDD{{\mathfrak n}(\beta_{l+1}),<}{0}{l+1}$ and
 \begin{equation}
\label{equ:step1}
\nu_{\Gamma} (\Omega_{k}^{1}) -   \nu_{\Gamma} (\Omega_{k}^{0})  =
\Beg_{\Gamma}(\Omega_{k}^{1})  - \Beg_{\Gamma}(\Omega_{k}^{0}).
\end{equation}
\end{lem}
\begin{proof}
  Denote $\Omega_{j}^{0} = f_{\Gamma, l+1}^{r_a} \Omega_a$ and
  $\Omega_{k}^{0} = f_{\Gamma, l+1}^{r_b} \Omega_b$.  We are in case
  \eqref{emv} of Proposition \ref{pro:conflicting_pairs} because
  otherwise we would have
  $\Beg_{\Gamma}(\Omega_{j}^{0}) < \beta_{l+1}$ by Lemma
  \ref{rem:conflicting_pairs}.

  As a consequence, $r_b=r_a+1$, $ \nu_{\Gamma} (\Omega_b) \in (e_l)$
  and $\nu_{\Gamma} (\Omega_a) \in [\beta_{l+1}]_{e_l}$ by Proposition
  \ref{pro:conflicting_pairs}.  The $R$-terms
    (cf. Definition \ref{def:max}) of $f_{\Gamma, l+1} \circ \Gamma'$
    and $t (\Gamma')^{*} \Omega_a$, where $\Gamma'$ runs on
    $\fgb{\Gamma}{\beta_{l+1}}$, are non-vanishing complex numbers.
  Thus, the $R$-terms of
    $f_{\Gamma, l+1}^{r_a} \Omega_a$ and
    $f_{\Gamma, l+1}^{r_a +1} \Omega_b$ are of the form
    $c a_{\beta_{l+1}}^{r_a}$ where $c \in {\mathbb C}^{*}$.  The
    $Q$-terms (cf. Definition \ref{def:max}) are both equal to $0$
    since $\Omega_{j}^{0}, \Omega_{k}^{0} \in \tDd{\beta_{l+1}}{l+1}$
    by Corollaries \ref{cor:f_dominates_omega} and \ref{cor:f_omega}.
    Thus $\Omega_{k}^{1}$ belongs to
    $\DD{{\mathfrak n}(\beta_{l+1})}{0}{l+1}$ and its $R$-term is of
    the form $c a_{\beta_{l+1}}^{r_a}$ where $c \in {\mathbb C}^{*}$.
    Moreover, the $c$-sequence of $\Omega_{k}^{0}$ is
    $1, r_{a} + 1, r_{a}+1, \ldots$ by Corollary \ref{cor:f_dominates_omega}
    whereas the $c$-sequence of $\Omega_{j}^{0}$ is
    \begin{equation*}
      1, c_{{\mathfrak n}(\beta_{l+1})} + r_{a} ,
      c_{{\mathfrak n}^{2}(\beta_{l+1})} +r_{a} ,
      \ldots
    \end{equation*}
    by Corollary \ref{cor:f_omega}, where
    $(c_{{\mathfrak n}^{r}(\beta_{l+1})})_{r \geq 0}$ is the
    $c$-sequence of $\Omega_{a}$.  Therefore, the $c$-sequence of
    $\Omega_k^1$ is
    $((c_{{\mathfrak n}^{r}(\beta_{l+1})}-1) (c_{{\mathfrak n}(\beta_{l+1})}-1)^{-1} )_{r \geq 1}$. Hence,
    $\Omega_{j}^{1} = \Omega_{j}^{0} \in \oDD{\beta_{l+1}, <}{r_{a}}{l+1}$ and
    $\Omega_{k}^{1} \in \hatDD{{\mathfrak n}(\beta_{l+1}),<}{0}{l+1}$.
    The last property is immediate.
\end{proof} 
 \begin{lem} 
 \label{lem:dom0}
 Assume that
 $[\nu_{\Gamma} (\Omega^{0}_j)]_{n}=[\nu_{\Gamma}
 (\Omega^{0}_{k})]_{n}$ for $j < k$ and that either
 $\Beg_{\Gamma}(\Omega_{j}^{0})$ or $\Beg_{\Gamma}(\Omega_{k}^{0})$ is
 not $\beta_{l+1}$.  Then we get $\Omega_{j}^{1}=\Omega_{j}^{0}$ and
 $ \Omega_{k}^{1} \in \DD{\beta_{l+1},<}{n_{l+1}-1}{l+1}$.  Even more,
 the $c$-sequences of $\Omega_{k}^{1}$ and $\Omega_{k}^{0}$ coincide,
 and
 \begin{equation*}
   \nu_{\Gamma} (\Omega_{k}^{1}) - \nu_{\Gamma} (\Omega_{k}^{0}) =
   \Beg_{\Gamma}(\Omega_{k}^{1})  - \Beg_{\Gamma}(\Omega_{k}^{0})
 \end{equation*}
holds if  $ \Beg_{\Gamma}(\Omega_{k}^{1}) \neq \infty$. 

\end{lem}
\begin{proof}
  Denote $\Omega_{j}^{0} = f_{\Gamma, l+1}^{r_a} \Omega_a$ and
  $\Omega_{k}^{0} = f_{\Gamma, l+1}^{r_b} \Omega_b$.  Since in case \eqref{emv} of
  Proposition \ref{pro:conflicting_pairs} we have
  $\Beg_{\Gamma}(\Omega_{j}^{0})=\Beg_{\Gamma}(\Omega_{k}^{0})=\beta_{l+1}$, we are
  in case \eqref{dmv}, so that $r_b=n_{l+1}-1$, $r_a=0$,
  $ \nu_{\Gamma} (\Omega_b) \in [\beta_{l+1}]_{e_l}$ and
  $\nu_{\Gamma} (\Omega_a) \in (e_l)$.  Lemma \ref{rem:conflicting_pairs} implies
  $\Omega_{k}^{0} \in \oDD{\beta_{l+1}, <}{n_{l+1}-1}{ l+1}$ and
  $\Omega_{j}^{0} \in \oDD{\beta'}{<}{l+1}$ for some $\beta' < \beta_{l+1}$.  Since
  $\beta_{l+1} + ({\mathfrak n}^{r}(\beta') - \beta') \in {\mathcal E}_{\Gamma}$ for
  any $r \in {\mathbb Z}_{\geq 0}$, it follows, from the definition of $\Omega^1_k$,
  that $t (\Gamma')^{*} (\Omega_{k}^{1})$ is of the form  
  $t^{\kappa} \sum_{r \geq 0}^{\infty} C_{{\mathfrak n}^{r}(\beta_{l+1})}
  t^{{\mathfrak n}^{r}(\beta_{l+1})} dt$ when $\Gamma'$ runs on
  $\fgb{\Gamma}{\beta_{l+1}}$.

  The $R$-term of $\Omega_{k}^{0}$ (cf. Definition \ref{def:max}) is of the form
  $c a_{\beta_{l+1}}^{n_{l+1}-1}$ where $c \in {\mathbb C}^{*}$.  Since
  $\beta' < \beta_{l+1}$, it follows that
  $\Omega_{k}^{1} \in \DD{\beta_{l+1},<}{n_{l+1}-1}{l+1}$ and the $c$-sequences
  of $\Omega_{k}^{0}$ and $\Omega_{k}^{1}$ coincide.  The last property is
  immediate.
\end{proof}
Since we want to prove that ${\mathcal G}_{\Gamma, l+1}^{1}$ is good, we need to
define the sets $\low_{\Gamma, l+1}^{1}$ and $\upper_{\Gamma, l+1}^{1}$ in
Definition \ref{def:good} (good family). For each element
$j\in \mathbb{N}_{2\nu_l}$, we are going to specify to which set the elements
$j,j+2\nu_l,\ldots,j+2(n_{l+1}-1)\nu_l$ belong:
\begin{defi}
  \label{def:lu}
  Given $j\in \mathbb{N}_{2\nu_{l}}$:
  \begin{enumerate}
  \item \label{ml1} If $\Omega_j\in \oDd{\beta', <}{l+1}$ for some
    $\beta'<\beta_{l+1}$, then $j \in \low_{\Gamma, l+1}^{1}$ and
    $j + 2 i \nu_l \in \upper_{\Gamma, l+1}^{1}$ for any $1 \leq i < n_{l+1}$.
  \item \label{mu1} If $\Omega_{j} \in \oDD{\beta_{l+1}, <}{0}{ l+1}$, then
    $j + 2 i \nu_l \in \low_{\Gamma, l+1}^{1}$ for any $0 \leq i < n_{l+1} -1$,
    and $j + 2 (n_{l+1}-1) \nu_l \in \upper_{\Gamma, l+1}^{1}$.
  \end{enumerate}
\end{defi}
\begin{rem}
\label{rem:lu}
Assume that
$[\nu_{\Gamma} (\Omega^{0}_j)]_{n}=[\nu_{\Gamma} (\Omega^{0}_{k})]_{n}$ for
$j < k$.  Definition \ref{def:lu} guarantees that $j \in \low_{\Gamma, l+1}^{1}$
and $k \in \upper_{\Gamma, l+1}^{1}$ by Lemma \ref{rem:conflicting_pairs}.  As
a consequence, both $\low_{\Gamma, l+1}^{1}$ and $\upper_{\Gamma, l+1}^{1}$ have
cardinal $\nu_{l+1}$.  By construction, $\overline\upper_{\Gamma, l+1}$ has
$\nu_{l}$ elements, and the inclusions
${\mathbb N}_{2 \nu_{l}} \subset \low_{\Gamma, l+1}^{1}$, and
$\overline\upper_{\Gamma, l+1} \subset \upper_{l+1} \subset
\upper_{\Gamma, l+1}^{1}$ hold.
\end{rem}
The following technical result will let us prove that
$\mathcal{G}^1_{\Gamma,l+1}$ is indeed a good family. 
\begin{pro}
  \label{pro:aux1}
  The family ${\mathcal G}_{\Gamma, l+1}^{1}$ satisfies (where in all statements, $j,k\in \mathbb{N}_{2\nu_{l+1}}$):
  \begin{enumerate}
  \item \label{g1p1} $j \in \low_{\Gamma, l+1}^{1}$ if and only if
    $\Omega_{j}^{1} = \Omega_{j}^{0}$;
  \item \label{g1p2} if $j \in \low_{\Gamma, l+1}^{1}$, then
    $\Omega_{j}^{1} \in \cup_{\beta^{\prime} \leq \beta_{l+1}}
    \oDd{\beta^{\prime},<}{l+1}$;
  \item \label{g1p3} if
    $k \in \upper_{\Gamma, l+1}^{1} \setminus \overline\upper_{\Gamma,
      l+1}$, then
    $\Omega_{k}^{1} \in \hatDD{{\mathfrak n}(\beta_{l+1}),<}{0}{l+1}$ ;
  \item \label{g1p4} if $k \in \overline\upper_{\Gamma, l+1}$, then
    $\Omega_{k}^{1} \in \DD{\beta_{l+1},<}{n_{l+1}-1}{l+1}$;
  \item \label{g1p5} if
    $k \in \mathbb{N}_{2\nu_{l+1}}$, then
    $\nu_{\Gamma} (\Omega_{k}^{1}) - \nu_{\Gamma} (\Omega_{k}^{0}) =
   \Beg_{\Gamma}(\Omega_{k}^{1})  - \Beg_{\Gamma}(\Omega_{k}^{0})$;
  \item \label{g1p6}
    $ \upper_{\Gamma, l+1}^{1} = \{ j \in {\mathbb N}_{2 \nu_{l+1}} :
    \Beg_{\Gamma}(\Omega_{j}^{1}) > \beta_{l+1} \}$;
  \item \label{g1p7}
    $\{ [\nu_{\Gamma}(\Omega_{j}^{1})]_{n} : j \in \low_{\Gamma, l+1}^{1} \} =
    [(e_{l+1})]_{n}$;
  \item \label{g1p8}
    the set $\{ [\nu_{\Gamma}(\Omega_{k}^{1}) - \Beg_{\Gamma}(\Omega_{k}^{1})]_{n} : k
    \in \upper_{\Gamma, l+1}^{1,<\infty} \}$ is contained in $[(e_{l+1})]_{n}$ and has
    $\sharp \upper_{\Gamma, l+1}^{1,<\infty}$ elements (cf. Definition
    \ref{def:good}).
	\end{enumerate}
      \end{pro}
      Notice, in the following proof, the relevance of Lemma
      \ref{rem:conflicting_pairs}.
\begin{proof}
  Property \eqref{g1p1} is a consequence of Remark \ref{rem:lu}.  This remark
  together with Lemma \ref{rem:conflicting_pairs} imply Property \eqref{g1p2}.

  Let $k \in \upper_{\Gamma, l+1}^{1} \setminus \overline\upper_{\Gamma, l+1}$.
  There is an integer $j \in \low_{\Gamma, l+1}^{1}$ with
  $[\nu_{\Gamma} (\Omega^{0}_j)]_{n}=[\nu_{\Gamma} (\Omega^{0}_{k})]_{n}$. By
  Lemma \ref{rem:conflicting_pairs}, we are not in case \eqref{dmv} of
  Proposition \ref{pro:conflicting_pairs}, and Lemma
  \ref{lem:non_dom} gives
  $\Omega_{k}^{1} \in \hatDD{{\mathfrak n}(\beta_{l+1}),<}{0}{l+1}$, which is
  Property \eqref{g1p3}.

  If $k \in \overline\upper_{\Gamma, l+1}$, we are in case \eqref{dmv} of
  Proposition \ref{pro:conflicting_pairs}, and by Lemma
  \ref{lem:dom0}, we obtain
  $\Omega_{k}^{1} \in \DD{\beta_{l+1},<}{n_{l+1}-1}{l+1}$, i.e. Property \eqref{g1p4}.

  Property \eqref{g1p5} follows from Lemmas \ref{lem:non_dom} and
  \ref{lem:dom0}. 
  
  Lemma \ref{rem:conflicting_pairs} implies
  $\Beg_{\Gamma}(\Omega_{j}^{1}) = \Beg_{\Gamma}(\Omega_{j}^{0}) \leq
  \beta_{l+1}$ for any $j \in \low_{\Gamma, l+1}^{1}$.
  Moreover, notice that
  $\Beg_{\Gamma}(\Omega_{k}^{0}) = \beta_{l+1}$ 
  for any  $k \in \upper_{\Gamma, l+1}^{1}$  by Lemma
  \ref{rem:conflicting_pairs} and hence $\Beg_{\Gamma}(\Omega_{k}^{1}) > \beta_{l+1}$
  by Property \eqref{g1p5}. This concludes the proof of Property \eqref{g1p6}.

  The set
  $S:=\{ [\nu_{\Gamma}(\Omega_{j}^{1})]_{n} : j \in \low_{\Gamma, l+1}^{1} \}$
  is contained in $[\nu_{\Gamma} ({\mathcal G}_{\Gamma, l+1}^{0})]_{n}$ by
  Property \eqref{g1p1}, and thus in $[(e_{l+1})]_{n}$ by Proposition
  \ref{pro:conflicting_pairs}.  Moreover, $S$ has cardinal $\nu_{l+1}$ by Remark
  \ref{rem:lu}.  Since $\sharp S = \sharp [(e_{l+1})]_{n} = \nu_{l+1}$ and
  $S \subset [(e_{l+1})]_{n}$, it follows that $S = [(e_{l+1})]_{n}$, i.e. Property
  \eqref{g1p7}.

  Finally,
  \begin{equation*}
    \nu_{\Gamma} (\Omega_{k}^{1}) -
    \Beg_{\Gamma}(\Omega_{k}^{1}) = \nu_{\Gamma} (\Omega_{k}^{0}) -
    \Beg_{\Gamma}(\Omega_{k}^{0}) = \nu_{\Gamma} (\Omega_{k}^{0}) -
    \beta_{l+1}
  \end{equation*}
  for any $k \in \upper_{\Gamma, l+1}^{1,<\infty}$ by Property \eqref{g1p5}.
  Since
  \begin{equation*}
    \{ [\nu_{\Gamma}(\Omega_{k}^{0})]_{n} : k \in
    \upper_{\Gamma, l+1}^{1} \} = \{ [\nu_{\Gamma}(\Omega_{j}^{0})]_{n} : j \in
    \low_{\Gamma, l+1}^{1} \} =[(e_{l+1})]_{n},
  \end{equation*}
  Property \eqref{g1p8} follows.
\end{proof} 
\begin{cor}
  \label{cor:good-to-good0}
  $\mathcal{G}^{1}_{\Gamma, l+1}$ is a good family of level $l+1$ and of stage
  $1$.
\end{cor}
\begin{proof}
  In order to prove Condition \ref{good1} of Definition \ref{def:good}, notice
  that ${\mathbb N}_{2 \nu_{l}} \subset \low_{\Gamma, l+1}^{1}$ by Remark
  \ref{rem:lu}, so that $\Omega_{j} = \Omega_{j}^{0} = \Omega_{j}^{1}$ for any
  $j \in {\mathbb N}_{2 \nu_{l}}$.  Condition \ref{good2} is a consequence of
  Remark \ref{rem:stage1_free}.

  The set $\upper_{\Gamma, l+1}^{1}$ is composed of the $1$-forms $\omega$ in
  $\mathcal{G}^{1}_{\Gamma, l+1}$ satisfying
  $\Beg_{\Gamma}(\omega)>\beta_{l+1}$.  Thus, Conditions \ref{good6},
  \ref{good3} and \ref{good4} follow from Properties
  \eqref{g1p2}, \eqref{g1p3} and \eqref{g1p4} in Proposition \ref{pro:aux1}.

  The set $\upper_{l+1}$ is contained in $\upper_{\Gamma, l+1}^{1}$ by
  Remark \ref{rem:lu}.  We also have
  $\Beg_{\Gamma}(\Omega_{k}^{1}) \leq \beta_{l+2}$ for any
  $k \in \upper_{\Gamma, l+1}^{1}$ by Properties \eqref{g1p3} and \eqref{g1p4}
  in Proposition \ref{pro:aux1} and Proposition \ref{pro:unique_companion}.  The
  remaining part of property \ref{good5} holds by construction.

  Condition \ref{good8} is a consequence of Property \eqref{g1p5} in Proposition
  \ref{pro:aux1}.  Notice that the second inequality is, in this case, always an
  equality.

  Finally, Condition \ref{good7} follows from the last two properties in
  Proposition \ref{pro:aux1}.
\end{proof}

\subsection{From stage $s \geq 1$ to stage $s+1$ in level $l+1$}
\label{subsec:from_s_to_s+1}
In this subsection, given a good family $\mathcal{G}^{s}_{\Gamma, l+1}$ of level
$l+1$ and stage $s \geq 1$ with ${\mathfrak n}^{s}(\beta_{l+1}) < \beta_{l+2}$,
we construct a family $\mathcal{G}^{s+1}_{\Gamma, l+1}$ of level $l+1$ and stage
$s+1$, and prove that it is good. This is the \textsc{Main
  Transformation} in Algorithm~\ref{alg:construction}.

In the next subsection, we shall prove that, after a finite number of
iterations, we obtain a terminal family $\mathcal{G}^{d}_{\Gamma, l+1}$ of level
$l+1$.
\begin{rem}
\label{rem:next}
Consider $k \in \upper_{\Gamma, l+1}^{s}$ with
$\Beg_{\Gamma}(\Omega_{k}^{s}) = {\mathfrak n}^{s} (\beta_{l+1})$.  The property
${\mathfrak n}^{s} (\beta_{l+1}) < \beta_{l+2}$ implies
$\nu_{\Gamma}(\Omega_{k}^{s}) \in (e_{l+1})$ by Condition \ref{good7} of good
families. By the same Condition, there exists $j \in \low_{\Gamma, l+1}^{s}$
with $[\nu_{\Gamma}(\Omega_{j}^{s})]_{n} = [\nu_{\Gamma}(\Omega_{k}^{s})]_{n}$. 
The \textsc{Main Transformation} in Algorithm \ref{alg:construction} is only carried out 
for these forms, at this stage, as follows. 
\end{rem}
\begin{defi}
\label{def:next}
Fix $s \geq 1$.  Let $k \in \upper_{\Gamma, l+1}^{s}$ with
$\Beg_{\Gamma}(\Omega_{k}^{s}) = {\mathfrak n}^{s} (\beta_{l+1})$, let
$j \in \low_{\Gamma, l+1}^{s}$ satisfying
$[\nu_{\Gamma}(\Omega_j^s)]_{n}=[\nu_{\Gamma}(\Omega_k^s)]_{n}$, as by
Remark \ref{rem:next}, and
$d = ( \nu_{\Gamma} (\Omega_{k}^{s}) -\nu_{\Gamma} (\Omega_{j}^{s}))/n$.  We
define
\begin{itemize}
\item {\it If $d \geq 0$}, then
  \begin{equation*}
    \Omega_{j}^{s+1} = \Omega_{j}^{s}, \ \ \Omega_{k}^{s+1} =
    \Omega_{k}^{s} - r x^{d} \Omega_{j}^{s},
  \end{equation*}
  where $r\in{\mathbb C}^{*}$ is such that
  $\nu_{\Gamma} (\Omega_{k}^{s} - r x^{d} \Omega_{j}^{s}) > \nu_{\Gamma}
  (\Omega_{k}^{s})$. We set $j \in \low_{\Gamma, l+1}^{s+1}$, and
  $k \in \upper_{\Gamma, l+1}^{s+1}$, and  define $\xi_{s}(k)=k$.
\item {\it If $d < 0$}, then
  \begin{equation*}
    \Omega_{k}^{s+1} = \Omega_{k}^{s}, \ \ \Omega_{j}^{s+1} =
    \Omega_{j}^{s} - r x^{-d} \Omega_{k}^{s},
  \end{equation*}
  where $r\in {\mathbb C}^{*}$ satisfies
  $\nu_{\Gamma} (\Omega_{j}^{s} - r x^{-d} \Omega_{k}^{s}) > \nu_{\Gamma}
  (\Omega_{j}^{s})$.  We set $k \in \low_{\Gamma, l+1}^{s+1}$, and $j \in
  \upper_{\Gamma, l+1}^{s+1}$, and define $\xi_{s}(k)=j$.
\end{itemize}
All other members $j$ of $\low_{\Gamma, l+1}^{s}$ (resp.
$\upper_{\Gamma, l+1}^{s}$) are set in $\low_{\Gamma, l+1}^{s+1}$
(resp. $\upper_{\Gamma, l+1}^{s+1}$) and we define
$\Omega_{j}^{s+1} = \Omega_{j}^{s}$.  This way, we obtain a map
$\xi_{s} : \upper_{\Gamma, l+1}^{s} \to \upper_{\Gamma, l+1}^{s+1}$ by extending
$\xi_s$ as $\xi_{s}(k) = k$ for any $k \in \upper_{\Gamma, l+1}^{s}$ such that
$\Beg_{\Gamma}(\Omega_{k}^{s}) > {\mathfrak n}^{s} (\beta_{l+1})$.
\end{defi}
\begin{defi}
\label{def:zeta} 
Given $s \in {\mathbb Z}_{\geq 1}$, define the map
$\zeta_{s}: \upper_{\Gamma, l+1}^{1} \to \upper_{\Gamma, l+1}^{s}$ as
$\zeta_{s} = \mathrm{id}_{\upper_{\Gamma, l+1}^{1}}$ if $s=1$ and
$\zeta_{s} = \xi_{s-1} \circ \ldots \circ \xi_{1}$ if $s>1$.
\end{defi}
Note that Definition \ref{def:next} guarantees that 
$\nu_{\Gamma} (\Omega_{k}^{s+1}) > \nu_{\Gamma} (\Omega_{k}^{s}) $ whenever
$\Omega_{k}^{s} \neq \Omega_{k}^{s+1}$.

The following result will allow us to prove that the \textsc{Main
  Transformation} maps good families $\mathcal{G}_{\Gamma, l+1}^s$ 
  to good families $\mathcal{G}_{\Gamma, l+1}^{s+1}$, as long as  
  ${\mathfrak n}^{s}(\beta_{l+1}) < \beta_{l+2}$.
 \begin{lem} 
 \label{lem:dom}
 Assume that $[\nu_{\Gamma} (\Omega^{s}_j)]_{n}=[\nu_{\Gamma} (\Omega^{s}_{k})]_{n}$ for $j \neq k$, 
 and that  $\Beg_{\Gamma}(\Omega_{j}^{s}) < \Beg_{\Gamma}(\Omega_{k}^{s}) = {\mathfrak n}^{s} (\beta_{l+1})$. 
  Then 
 \begin{itemize}
 \item $\Omega_{j^{\prime}}^{s+1}=\Omega_{j^{\prime}}^s$ where $j'$ is the element in
   $\{j, k\} \setminus \{ \xi_{s}(k) \}$ (the index corresponding to the $1$-form which does not change);
 \item If $\Omega_{k}^{s} \in \hatDD{{\mathfrak n}(\beta_{l+1}),<}{0}{l+1}$,
   then we also have
   \[
     \Omega_{\xi_{s}(k)}^{s+1} \in \hatDD{{\mathfrak n}(\beta_{l+1}),<}{0}{l+1};\]
 \item If $\Omega_{k}^{s} \in \DD{\beta_{l+1},<}{n_{l+1}-1}{l+1}$, then we also have
   \[
     \Omega^{s+1}_{\xi_s(k)}\in \DD{\beta_{l+1},<}{n_{l+1}-1}{l+1}
   \]
 \item the $c$-sequences of $\Omega_{\xi_{s}(k)}^{s+1}$ and $\Omega_{k}^{s}$ coincide;
 \item
   $\nu_{\Gamma} (\Omega_{\xi_{s}(k)}^{s+1}) - \nu_{\Gamma} (\Omega_{\xi_{s}(k)}^{s}) =
   \Beg_{\Gamma}(\Omega_{\xi_{s}(k)}^{s+1}) -
   \Beg_{\Gamma}(\Omega_{k}^{s})$;
    \item
   $\nu_{\Gamma} (\Omega_{j^{\prime}}^{s+1}) - \nu_{\Gamma} (\Omega_{j}^{s}) \in n
   {\mathbb Z}_{\leq 0}$;
  \end{itemize}
\end{lem}
Notice how the second and third properties underscore the importance of the sets $\hatDD{{\mathfrak n}(\beta_{l+1}),<}{0}{l+1}$ and $\DD{\beta_{l+1},<}{n_{l+1}-1}{l+1}$, as they are stable under the \textsc{Main Transformation}. The fifth property allows us to keep track of $\nu_{\Gamma}(\Omega_{\xi_s(k)}^{s+1})$ and $\Beg_{\Gamma}(\Omega_{\xi_s(k)}^{s+1})$, 
as required in Definition \ref{def:good} (good family).
\begin{proof}
  The first property follows by construction.  We have
  $j \in \low_{\Gamma, l+1}^{s}$ and $k \in \upper_{\Gamma, l+1}^{s}$ by Condition
  \ref{good5} in Definition \ref{def:good}.  Moreover,
  $\Omega_{j}^{s} \in \Dd{\beta', <}{l+1}$ for some
  $\beta' \leq {\mathfrak n} (\beta_{l+1})$ and
  $\Omega_{k}^{s} \in \Dd{\beta_{l+1}, <}{l+1} \cup \hatDD{{\mathfrak
      n}(\beta_{l+1}), <}{0}{ l+1}$.  Thus, we deduce that
  $\Omega_{k}^{s} \in \Dd{\beta'', <}{l+1}$ for some
  $\beta'' \in \{ \beta_{l+1}, {\mathfrak n} (\beta_{l+1}) \}$.

  Let $d = ( \nu_{\Gamma} (\Omega_{k}^{s}) -\nu_{\Gamma} (\Omega_{j}^{s}))/n$, and
  define
  \begin{equation*}
    \overline{\Omega}_{j}^{s} = x^{\max (d,0)} {\Omega}_{j}^{s} \ \ \mathrm{and} \ \  
    \overline{\Omega}_{k}^{s} = x^{\max(-d,0)} {\Omega}_{k}^{s}
  \end{equation*}
  We have
  \begin{equation*}
    t (\Gamma')^{*} (\overline{\Omega}_{j}^{s}) = 
    \left( t^{a} \sum_{r=0}^{\infty}
      C_{{\mathfrak n}^{r}(\beta')} t^{{\mathfrak n}^{r}(\beta')} \right) dt
  \end{equation*}
and 
\begin{equation}
\label{equ:exp_k}
 t (\Gamma')^{*} (\overline{\Omega}_{k}^{s}) = 
 \left( t^{b} \sum_{r=0}^{\infty}
   D_{{\mathfrak n}^{r}(\beta'')} t^{{\mathfrak n}^{r}(\beta'')} \right) dt 
\end{equation}
where $\Gamma'$ runs in $\fgb{\Gamma}{\beta_{l+1}}$ (cf. Definition \ref{def:level}).
As
$\nu_{\Gamma} (\overline{\Omega}_{j}^{s}) = \nu_{\Gamma} (\overline{\Omega}_{k}^{s})$,
it follows that
\begin{equation}\label{eq:a-and-b}
  a  + \Beg_{\Gamma}(\Omega_{j}^{s}) = b   + \Beg_{\Gamma}(\Omega_{k}^{s}) .
\end{equation}
We have $a > b$ and moreover $a-b \in (e_{l+1})$, for example because since
$\Beg_{\Gamma}(\Omega_{j}^{s}) < \Beg_{\Gamma}(\Omega_{k}^{s}) < \beta_{l+2}$, then
$\Beg_{\Gamma}(\Omega_{j}^{s}), \Beg_{\Gamma}(\Omega_{k}^{s})\in (e_{l+1})$.

We want to show that $t (\Gamma')^{*} (\overline{\Omega}_{k'}^{s+1})$ admits an
expression of type \eqref{equ:exp_k} for different coefficients in
${\mathbb C}[a_{\hat{\beta}}]_{\hat{\beta} \geq \beta_{l+1}}$.  It suffices to prove
that given $r \in {\mathbb Z}_{\geq 0}$, the expression
$a + {\mathfrak n}^{r} (\beta')$ is of the form $b + \overline{\beta}$ for some
$\overline{\beta} \in {\mathcal E}_{\Gamma}$ such that
$\overline{\beta} \geq \beta''$.  First, assume
${\mathfrak n}^{r} (\beta') \geq \Beg_{\Gamma}(\Omega_{j}^{s})$. This is case is
simple, as using~\eqref{eq:a-and-b}, we get
\begin{multline*} 
  \overline{\beta} =
  \Beg_{\Gamma}(\Omega_{k}^{s}) + (\overline{\beta} - \Beg_{\Gamma}(\Omega_{k}^{s}) )
  =\\
  \Beg_{\Gamma}(\Omega_{k}^{s}) + ({\mathfrak n}^{r} (\beta')  -
  \Beg_{\Gamma}(\Omega_{j}^{s}) ) \geq \Beg_{\Gamma}(\Omega_{k}^{s}) \geq \beta''
\end{multline*}
so that $\overline{\beta} \in {\mathcal E}_{\Gamma}$ by Remark \ref{rem:no_fw_hole}.
Now, suppose ${\mathfrak n}^{r} (\beta') < \Beg_{\Gamma}(\Omega_{j}^{s})$. Since
$\Beg_{\Gamma}(\Omega_{j}^{s}) > \beta'$, we obtain $\beta' \geq \beta_{l+1}$ by
Remark \ref{rem:l_b}. The inequalities
\begin{equation*}
  \beta_{l+1} \leq \beta' \leq {\mathfrak n}^{r} (\beta')  <
  \Beg_{\Gamma}(\Omega_{j}^{s}) < \Beg_{\Gamma}(\Omega_{k}^{s}) < \beta_{l+2}
\end{equation*}
imply that $ {\mathfrak n}^{r} (\beta')  \in (e_{l+1})$ and hence $\overline{\beta}  \in (e_{l+1})$. 
This property, together with $\overline{\beta} > {\mathfrak n}^{r} (\beta') \geq \beta_{l+1}$ imply 
$\overline{\beta} \in {\mathcal E}_{\Gamma}$. Moreover, we get $\overline{\beta} \geq \beta''$ since
$\overline{\beta} \geq {\mathfrak n}(\beta_{l+1})$.

All other properties in the lemma are straightforward.
\end{proof}  
Now, we are ready to study the properties of the family
$\mathcal{G}_{\Gamma, l+1}^{s+1}$, as long as
$\mathfrak{n}^s(\beta_{l+1})<\beta_{l+2}$:
\begin{defi}
  If $\mathfrak{n}^s(\beta_{l+1}) )<\beta_{l+2}$, then we define
  \begin{multline*}
    \upper_{\Gamma, l+1}^{s+1} =
    \{ j \in \upper_{\Gamma,
      l+1}^{s} :
    \Beg_{\Gamma}(\Omega_{j}^{s}) > {\mathfrak n}^{s} (\beta_{l+1}) \}
    \cup \\
    \{ k \in {\mathbb N}_{2 \nu_{l+1}} :
    \Omega_{k}^{s+1} \neq \Omega_{k}^{s} \}
  \end{multline*}
  and $\low_{\Gamma, l+1}^{s+1}$ as its complement in ${\mathbb N}_{2 \nu_{l+1}} $.
\end{defi}
The next result will imply that $\mathcal{G}_{\Gamma,l+1}^{s+1}$ is
good. The importance of the sets
$\hatDD{{\mathfrak n}(\beta_{l+1}),<}{0}{l+1}$ and
$\DD{\beta_{l+1},<}{n_{l+1}-1}{l+1}$ is stressed again.
\begin{pro}
  \label{pro:auxm}
The family  ${\mathcal G}_{\Gamma, l+1}^{s+1}$ satisfies
\begin{itemize}
\item Both $\low_{\Gamma, l+1}^{s+1}$ and $\upper_{\Gamma, l+1}^{s+1}$ have
  $\nu_{l+1}$ elements;
\item ${\mathbb N}_{2 \nu_{l}} \subset \low_{\Gamma, l+1}^{s+1}$ and
  $\upper_{l+1} \subset \upper_{\Gamma, l+1}^{s+1}$;
\item if $j \in \low_{\Gamma, l+1}^{s+1}$, then
  $\Beg_{\Gamma} (\Omega_{j}^{s+1}) < {\mathfrak n}^{s+1} (\beta_{l+1})$;
\item if $k \in \upper_{\Gamma, l+1}^{s+1}$, then
  $\Beg_{\Gamma}(\Omega_{k}^{s+1}) \geq
  {\mathfrak n}^{s+1}(\beta_{l+1})$ ;
\item If
  $k \in \upper_{\Gamma, l+1}^{s}$, then
  \begin{equation*}
    \nu_{\Gamma} (\Omega_{\xi_{s}(k)}^{s+1}) - \nu_{\Gamma} (\Omega_{\xi_{s}(k)}^{s})
    = \Beg_{\Gamma}(\Omega_{\xi_{s}(k)}^{s+1}) -\Beg_{\Gamma}(\Omega_{k}^{s});
  \end{equation*}
\item if $k \in \upper_{\Gamma, l+1}^{s} \setminus \overline\upper_{\Gamma, l+1}$,
  then
  \begin{equation*}
    \Omega_{\xi_{s}(k)}^{s+1} \in \hatDD{{\mathfrak n}(\beta_{l+1}),<}{0}{l+1};
  \end{equation*}
\item if $k \in \overline\upper_{\Gamma, l+1}$, then
  \begin{equation*}
    \Omega_{\xi_{s}(k)}^{s+1} = \Omega_{k}^{s+1} \in
    \DD{\beta_{l+1},<}{n_{l+1}-1}{l+1}.
  \end{equation*}
 \end{itemize}
 \end{pro}
 \begin{proof}
   We have
   $\sharp \low_{\Gamma, l+1}^{s} = \sharp \upper_{\Gamma, l+1}^{s} = \nu_{l+1}$.  In
   order to show that the analogous property holds for $\low_{\Gamma, l+1}^{s+1}$ and
   $\upper_{\Gamma, l+1}^{s+1}$, it suffices to check out that in every pair $(j,k)$
   as in Definition \ref{def:next}, there is an element in $\low_{\Gamma, l+1}^{s+1}$
   and another in $\upper_{\Gamma, l+1}^{s+1}$, which is clear by definition.
 
 The following equality
 \begin{equation*}
   \nu_{\Gamma} (\Omega_{j}^{s}) =
   \nu_{\Gamma} (\Omega_{j})  \leq \overline{\beta}_{l+1}
 \end{equation*}
 holds for any $j \in {\mathbb N}_{2 \nu_{l}}$ by Condition \ref{good1} of good
 families and Condition \ref{tb} of terminal families.  Since for
 $k \in {\mathbb N}_{2 \nu_{l+1}} \setminus {\mathbb N}_{2 \nu_{l}}$, we have
 \begin{equation*}
   \nu_{\Gamma} (\Omega_{k}^{s}) \geq  \nu_{\Gamma} (\Omega_{k}^{0}) >
   \nu_{\Gamma} (f_{\Gamma, l+1}) = \overline{\beta}_{l+1},
 \end{equation*}
 we get ${\mathbb N}_{2 \nu_{l}}  \subset  \low_{\Gamma, l+1}^{s+1}$.
 
 Let $k \in \upper_{l+1}$. Since
 $\upper_{l+1} \subset \upper_{\Gamma, l+1}^{s}$, it is clear that
 $k \in \upper_{\Gamma, l+1}^{s+1}$ if
 $\Beg_{\Gamma}(\Omega_{k}^{s}) > {\mathfrak n}^{s} (\beta_{l+1})$. So, suppose that
 $\Beg_{\Gamma}(\Omega_{k}^{s}) = {\mathfrak n}^{s} (\beta_{l+1})$. Take
 $j \in \low_{\Gamma, l+1}^{s}$ from Remark \ref{rem:next}.  Since
 $j \not \in \upper_{l+1}$, we get
 $\Omega_{j}^{0} = f_{\Gamma, l+1}^{r} \Omega_a$ for some
 $a \in {\mathbb N}_{2 \nu_{l}}$ and $0 \leq r \leq n_{l+1}-2$. Moreover,
 $\Omega_{k}^{0}$ is of the form
 $\Omega_{k}^{0} = f_{\Gamma, l+1}^{n_{l+1}-1} \Omega_b$ for some
 $b \in {\mathbb N}_{2 \nu_{l}}$ because $k \in \upper_{l+1}$. Since
 $\nu_{\Gamma} (\Omega_{a}) \leq \overline{\beta}_{l+1}$ by Condition \ref{tb} of
 terminal families, we deduce 
 \begin{equation*}
   \nu_{\Gamma} (\Omega_{j}^{0})  \leq (n_{l+1} -1) \overline{\beta}_{l+1} \ \
   \mathrm{and} \ \ 
   \nu_{\Gamma} (\Omega_{k}^{0})  > (n_{l+1} -1) \overline{\beta}_{l+1} .
 \end{equation*}
 In particular $\nu_{\Gamma} (\Omega_{k}^{0}) > \nu_{\Gamma} (\Omega_{j}^{0})$ holds.
 Since
 \begin{equation*}
   \Beg_{\Gamma} (\Omega_{j}^{s}) - \Beg_{\Gamma} (\Omega_{j}^{0}) <
   {\mathfrak n}^{s} (\beta_{l+1})- \beta_{l+1} 
   = \Beg_{\Gamma} (\Omega_{k}^{s}) - \Beg_{\Gamma} (\Omega_{k}^{0}),
 \end{equation*}
  we deduce 
  \begin{multline*}
    \nu_{\Gamma} (\Omega_{j}^{s})
    \leq \nu_{\Gamma} (\Omega_{j}^{0}) + (\Beg_{\Gamma} (\Omega_{j}^{s})
    - \Beg_{\Gamma} (\Omega_{j}^{0})  )
    < \\
    \nu_{\Gamma} (\Omega_{k}^{0}) + (\Beg_{\Gamma} (\Omega_{k}^{s}) -
    \Beg_{\Gamma} (\Omega_{k}^{0}))   =  \nu_{\Gamma} (\Omega_{k}^{s})
  \end{multline*}
  by Condition \ref{good8} of good families. In particular
  $k \in \upper_{\Gamma, l+1}^{s+1}$ holds.
 
  The five remaining properties are a consequence of Conditions \ref{good3} and
  \ref{good4} in Definition \ref{def:good} (good family) and Lemma \ref{lem:dom}.
 \end{proof}
 We can now prove that being good is a hereditary property.
 \begin{pro}
 \label{pro:good-to-good}
 Assume $\mathcal{G}^{s}_{\Gamma, l+1}$ is good and
 ${\mathfrak n}^{s}(\beta_{l+1}) < \beta_{l+2}$.  Then
 $\mathcal{G}^{s+1}_{\Gamma, l+1}$ is good.
\end{pro}
\begin{proof}
  Condition \ref{good1} in Definition \ref{def:good} for
  $\mathcal{G}^{s}_{\Gamma, l+1}$ and
  ${\mathbb N}_{2 \nu_{l}} \subset \low_{\Gamma, l+1}^{s+1}$ (Proposition
  \ref{pro:auxm}) imply Condition \ref{good1} for $\mathcal{G}^{s+1}_{\Gamma, l+1}$.
  Condition \ref{good2} holds by construction. Conditions \ref{good6}, \ref{good3}
  and \ref{good4} for $\mathcal{G}^{s+1}_{\Gamma, l+1}$ are consequences of the
  analogous properties for $\mathcal{G}^{s}_{\Gamma, l+1}$ and the two last
  properties in Proposition \ref{pro:auxm}.

  We have ${\mathbb N}_{2 \nu_{l}} \subset \low_{\Gamma, l+1}^{s+1}$ and
  $\upper_{l+1} \subset \upper_{\Gamma, l+1}^{s+1}$ by Proposition
  \ref{pro:auxm}. Moreover, $\Beg_{\Gamma}(\Omega_{k}^{s+1}) \leq \beta_{l+2}$ holds
  for any $k \in \upper_{\Gamma, l+1}^{s+1}$ by Condition \ref{good5} for
  $\mathcal{G}^{s}_{\Gamma, l+1}$, the two last properties in Proposition
  \ref{pro:auxm} and Proposition \ref{pro:unique_companion}
   (this is one of the key consequences of that result together with
  Corollary \ref{cor:f_omega}: see Remark \ref{rem:ov_imp_t}).
  All other properties in Condition \ref{good5} for
  $\mathcal{G}^{s+1}_{\Gamma, l+1}$ are satisfied by construction.

  We claim that
\begin{equation}
\label{equ:good8}
0 \leq \nu_{\Gamma} (\Omega_{k}^{s+1}) - \nu_{\Gamma} (\Omega_{k}^{s}) 
\leq  \Beg_{\Gamma}(\Omega_{k}^{s+1})  - \Beg_{\Gamma}(\Omega_{k}^{s})  
 \end{equation}
 holds for any $k \in {\mathbb N}_{2 \nu_{l+1}}$.  Assume that
 $\Omega_{k}^{s} \neq \Omega_{k}^{s+1}$ since it is obviously true otherwise.  We
 have $k \in \upper_{\Gamma, l+1}^{s+1}$,
 \begin{multline*}
   0 \leq \nu_{\Gamma} (\Omega_{k}^{s+1}) - \nu_{\Gamma} (\Omega_{k}^{s}) 
   =
   \Beg_{\Gamma}(\Omega_{k}^{s+1})   -   \Beg_{\Gamma}(\Omega_{\xi_{s}^{-1}(k)}^{s})
   \leq \\
   \Beg_{\Gamma}(\Omega_{k}^{s+1}) -  \Beg_{\Gamma}(\Omega_{k}^{s})
 \end{multline*}
 by Proposition \ref{pro:auxm}.  The last inequality is an equality if
 $k \in \upper_{\Gamma, l+1}^{s}$, and in particular if $k \in \upper_{l+1}$.
 As a consequence, Condition \ref{good8} for $\mathcal{G}^{s}_{\Gamma, l+1}$ implies
 Condition \ref{good8} for $\mathcal{G}^{s+1}_{\Gamma, l+1}$.

 Let us consider $k \in \upper_{\Gamma, l+1}^{s}$ with
 $\Beg_{\Gamma}(\Omega_{k}^{s}) = {\mathfrak n}^{s} (\beta_{l+1})$. Let $j$ be the
 element of $\low_{\Gamma, l+1}^{s}$ provided by Remark \ref{rem:next}.  Denote by
 $j'$ the element in $\{j, k \} \setminus \{ \xi_{s} (k) \}$.  Note that
 $j' \in \low_{\Gamma, l+1}^{s+1}$ and $\xi_{s}(k) \in \upper_{\Gamma, l+1}^{s+1}$.
 We obtain
 \begin{equation*}
   \nu_{\Gamma}({\Omega}_{j'}^{s+1}) - \nu_{\Gamma}({\Omega}_{j}^{s})
   \in n {\mathbb Z}_{\leq 0}
 \end{equation*}
 and
 \begin{equation*}
   \nu_{\Gamma}({\Omega}_{\xi_{s}(k)}^{s+1}) -
   \Beg_{\Gamma}({\Omega}_{\xi_{s}(k)}^{s+1})   - 
   (\nu_{\Gamma}({\Omega}_{k}^{s}) -
   \Beg_{\Gamma}({\Omega}_{k}^{s})) \in n {\mathbb Z}_{\geq 0}
 \end{equation*}
 by Lemma \ref{lem:dom}. Now, Condition \ref{good7} for
 $\mathcal{G}^{s+1}_{\Gamma, l+1}$ is an immediate consequence of Condition
 \ref{good7} for $\mathcal{G}^{s}_{\Gamma, l+1}$.
\end{proof}
\subsection{Finiteness of the Algorithm}
We can finally prove that a terminal family of level $l+1$ is obtained after a finite number of iterations of Algorithm \ref{alg:construction}.
\begin{pro}
\label{pro:from-l-to-next}  
Consider the terminal family
${\mathcal T}_{\Gamma, l} = (\Omega_1, \ldots, \Omega_{2 \nu_{l}})$ of level
$0 \leq l < g$.  Then there is $s \in {\mathbb Z}_{\geq 1}$ such that
${\mathcal G}_{\Gamma, l+1}^{s}= (\Omega_{1}^{s}, \ldots, \Omega_{2 \nu_{l+1}}^{s})$
is a terminal family of level $l+1$ with $\Omega_{j}^{s} = \Omega_{j}$ for any
$j \in {\mathbb N}_{2 \nu_{l}}$.  Moreover, we have
\begin{itemize}
\item ${\mathfrak n}^{s} (\beta_{l+1}) = \beta_{l+2}$ if $l+1 < g$;
\item $(\Omega_{j}^{s} : j \in \low_{\Gamma, g}^{s})$ is a ${\mathbb C}[[x]]$-basis
  for $\Gamma$ if $l+1 = g$.
\end{itemize}
\end{pro} 
\begin{proof}
  We know that any good family ${\mathcal G}_{\Gamma, l+1}^{d}$, with $d \geq 1$, is
  a free basis of ${\mathcal M}_{2 \nu_{l+1}}$.
  
  $\bullet$ Case $l+1 < g$. Consider $s \in {\mathbb Z}_{\geq 1}$ with
  ${\mathfrak n}^{s} (\beta_{l+1}) = \beta_{l+2}$. Let $1 \leq d < s$.  Since
  $\nu_{\Gamma} ({\mathcal G}_{\Gamma, l+1}^{d}) \subset (e_{l+1})$ by Condition
  \ref{good7}, ${\mathcal G}_{\Gamma, l+1}^{d}$ has $2 \nu_{l+1}$ elements, and
  $\sharp [(e_{l+1})]_{n} = \nu_{l+1}$, it follows that
  ${\mathcal G}_{\Gamma, l+1}^{d}$ is non-separated and thus non-terminal. So, we can
  focus on ${\mathcal G}_{\Gamma, l+1}^{s}$.  The good family
  ${\mathcal G}_{\Gamma, l+1}^{s}$ is separated since
  $\Beg_{\Gamma}(\Omega_{k}^{s}) = \beta_{l+2}$ for any
  $k \in \upper_{\Gamma, l+1}^{s}$, $\beta_{l+2} \not \in (e_{l+1})$ and
  Condition \ref{good7} of good families for ${\mathcal G}_{\Gamma,
    l+1}^{s}$. Moreover, since any $\Omega_{j}^{0}$ is of the form
  $f_{\Gamma, l+1}^{r} \Omega_{a}$ for some $0 \leq r < n_{l+1}$ and
  $a \in {\mathbb N}_{2 \nu_{l}}$, we deduce
  \begin{multline*}
    \nu_{\Gamma} (\Omega_{j}^{s}) \leq \nu_{\Gamma} (\Omega_{j}^{0}) +
    (\Beg_{\Gamma}(\Omega_{j}^{s}) - \Beg_{\Gamma}(\Omega_{j}^{0})) 
    \leq \\
    n_{l+1} \overline{\beta}_{l+1} + (\beta_{l+2} - \beta_{l+1}) =
    \overline{\beta}_{l+2} \end{multline*}
  for any $j \in {\mathbb N}_{2 \nu_{l+1}}$ by equation \eqref{equ:rec_beta}
  (page \pageref{equ:rec_beta}).
  It only remains to prove Condition \ref{tc} of terminal families.
  Let $j \in {\mathbb N}_{2 \nu_{l+1}}$. 
  Suppose first $\Beg_{\Gamma}(\Omega_{j}^{s}) \leq \beta_{l+1}$, which implies
  $\Omega_{j}^{s} = \Omega_{j}^{0}$, 
  i.e. $j \in \cap_{a \in {\mathbb N}_{s}} \low_{\Gamma, l+1}^{a}$.
  As a consequence, we obtain 
\begin{equation*}
  \Omega_{j}^{s} \in \cup_{\beta^{\prime} \leq \beta_{l+1}}
  \oDd{\beta^{\prime}, <}{l+1} 
  \subset \cup_{\beta^{\prime} < \beta_{l+2}} \oDd{\beta^{\prime}, <}{l+2}
\end{equation*}
by Lemma \ref{rem:conflicting_pairs} and Remark \ref{rem:family_restriction}.  Now,
suppose $\Beg_{\Gamma}(\Omega_{j}^{s}) > \beta_{l+1}$. Then
$\Omega_{j}^{s} \in \Dd{\beta_{l+1}, <}{l+1} \cup \hatDD{{\mathfrak n}(\beta_{l+1}),
  <}{0}{ l+1}$ by Conditions \ref{good3} and \ref{good4} of good families. Since
$\Beg_{\Gamma}(\Omega_{j}^{s}) \leq \beta_{l+2}$, we obtain
\begin{multline*}
  \Omega_{j}^{s} \in \oDD{\Beg_{\Gamma}(\Omega_{j}^{s}), <}{0}{l+2} \subset \\
  \oDD{\beta_{l+2}, <}{0}{ l+2}
  \bigcup \cup_{\beta' \in {\mathcal E}(\Gamma), \ \beta' < 
  \beta_{l+2}} \oDd{\beta', <}{l+2}
\end{multline*}
by Remark \ref{rem:nb_puiseux}. Therefore, Condition \ref{tc} holds and
${\mathcal G}_{\Gamma, l+1}^{s}$ is a terminal family of level $l+1$.

$\bullet$ Assume now that $l+1 = g$. Condition \ref{tb} is empty. The proof of
Condition \ref{tc} for ${\mathcal G}_{\Gamma, l+1}^{d}$, and any
$d \in {\mathbb Z}_{\geq 1}$, is analogous to the previous case.  So it remains to
prove that there exists $s \in {\mathbb Z}_{\geq 1}$ such that
${\mathcal G}_{\Gamma, l+1}^{s}$ is separated. Let us remark that
$\sharp \low_{\Gamma, l+1}^{d} = n$ and
$(\nu_{\Gamma} (\Omega_{j}^{d}))_{j \in \low_{\Gamma, l+1}^{d}}$ defines pairwise
different classes modulo $n$ for any $d \in {\mathbb Z}_{\geq 1}$ by Condition
\ref{good7} of Definition \ref{def:good}.  It is clear, by definition, that
\begin{equation*}
  \max_{j \in \low_{\Gamma, l+1}^{d+1}}  \nu_{\Gamma} (\Omega_{j}^{d+1}) \leq
  \max_{j \in \low_{\Gamma, l+1}^{d}}  \nu_{\Gamma} (\Omega_{j}^{d})
\end{equation*}
and that
\begin{equation*}
  \min_{j \in \upper_{\Gamma, l+1}^{d}}  \nu_{\Gamma} (\Omega_{j}^{d}) \geq
  \mathfrak{n}^{d} (\beta_{l+1})
\end{equation*}
for any $d \in {\mathbb Z}_{\geq 1}$. Hence,
$\min_{j \in \upper_{\Gamma, l+1}^{s}} \nu_{\Gamma} (\Omega_{j}^{s}) > \max_{j \in
  \low_{\Gamma, l+1}^{s}} \nu_{\Gamma} (\Omega_{j}^{s})$ for some
$s \in {\mathbb Z}_{\geq 1}$. Thus, ${\mathcal G}_{\Gamma, l+1}^{s}$ is separated,
completing the proof.
\end{proof}
\begin{rem}
\label{rem:chu}
Consider $0 \leq l < g$ and $k \in \upper_{l+1}$. By construction, we have
$\Omega_{k}^{0} = f_{\Gamma, l+1}^{n_{l+1}-1} \Omega_{j}$ where
$\Omega_{j} \in {\mathcal T}_{\Gamma, l}$.  Condition \ref{good8} of Definition
\ref{def:good} implies
\begin{equation*}
  \nu_{\Gamma} (\Omega_{k}^{s}) = \nu_{\Gamma} (\Omega_j) +
  (n_{l+1} -1) \overline{\beta}_{l+1} + 
  \Beg_{\Gamma}(\Omega_{k}^{s}) - \beta_{l+1}.
\end{equation*}
Now, suppose $l+1<g$. Let $s_{l+1} \in {\mathbb Z}_{\geq 1}$ with
${\mathfrak n}^{s_{l+1}} (\beta_{l+1}) = \beta_{l+2}$.  The $1$-form
$\Omega_{k}^{s_{l+1}}$ belongs to ${\mathcal T}_{\Gamma, l+1}$ by Proposition
\ref{pro:from-l-to-next} and satisfies
\begin{equation*}
  \nu_{\Gamma} (\Omega_{k}^{s_{l+1}}) = \nu_{\Gamma} (\Omega_j) +
  (n_{l+1} -1) \overline{\beta}_{l+1} + 
  \beta_{l+2} - \beta_{l+1} =
  \nu_{\Gamma} (\Omega_j) + \overline{\beta}_{l+2} - \overline{\beta}_{l+1}
\end{equation*}
by equation \eqref{equ:rec_beta}.
\end{rem}
\begin{rem}
\label{rem:generators_construction_0} 
For $0 \leq l < g$, the terminal family
${\mathcal T}_{\Gamma, l} = (\Omega_1, \ldots, \Omega_{2 \nu_{l}})$ we have
computed satisfies $\nu_{\Gamma} (\Omega_{2 \nu_{l}}) =
\overline{\beta}_{l+1}$: this clearly holds for $l=0$ and if it does for $l < g-1$,
the family 
${\mathcal T}_{\Gamma, l+1} = (\Omega_1, \ldots, \Omega_{2 \nu_{l+1}})$ we have computed satisfies
$\nu_{\Gamma} (\Omega_{2 \nu_{l+1}}) = \overline{\beta}_{l+2}$ by Remark
\ref{rem:chu}.
\end{rem}\begin{rem}
\label{rem:tracking_leading}
Consider $0 \leq l < g$ and $k \in  \upper_{\Gamma, l+1}^{1}$. We have
\begin{equation*}
  \nu_{\Gamma} (\Omega_{k}^{1}) - \nu_{\Gamma} (\Omega_{k}^{0}) =
  \Beg_{\Gamma}(\Omega_{k}^{1})  - \beta_{l+1}
\end{equation*}
 by property \eqref{g1p5} in Proposition \ref{pro:aux1}. Moreover,
 \begin{equation*}
   \nu_{\Gamma} (\Omega_{\xi_{s}(k')}^{s+1}) - \nu_{\Gamma}
   (\Omega_{k'}^{s}) - ( \Beg_{\Gamma}(\Omega_{\xi_{s}(k')}^{s+1})
   -\Beg_{\Gamma}(\Omega_{k'}^{s})) \in n {\mathbb Z}_{\geq 0}
 \end{equation*}
 holds by Proposition \ref{pro:auxm} for all $k' \in \upper_{\Gamma, l+1}^{s}$,
 $s \in {\mathbb Z}_{\geq 1}$ such that
 $\Beg_{\Gamma}(\Omega_{\xi_{s}(k)}^{s+1}) \neq \infty$.  In particular, it is
 satisfied for $k' = \zeta_{s}(k)$ (cf. Definition \ref{def:zeta}).  By adding the
 previous expressions, we obtain
 \begin{equation*}
   \nu_{\Gamma} (\Omega_{\zeta_{s}(k)}^{s}) - \nu_{\Gamma} (\Omega_{k}^{0}) -
   ( \Beg_{\Gamma}(\Omega_{\zeta_{s}(k)}^{s})   - \beta_{l+1})
   \in n {\mathbb Z}_{\geq 0}
 \end{equation*}
 for all $k \in \upper_{\Gamma, l+1}^{1}$ and $s \in {\mathbb Z}_{\geq 1}$ with
 $\Beg_{\Gamma}(\Omega_{\zeta_{s}(k)}^{s}) \neq \infty$. Consider now the case
 $l+1 <g$.  The previous expression particularizes to
 \begin{equation*}
   \nu_{\Gamma} (\Omega_{\zeta_{s_{l+1}}(k)}^{s_{l+1}}) - \nu_{\Gamma} (\Omega_{k}^{0}) - 
   ( \beta_{l+2}  - \beta_{l+1})  \in n {\mathbb Z}_{\geq 0}
 \end{equation*}
 for all $k \in \upper_{\Gamma, l+1}^{1}$, where
 ${\mathfrak n}^{s_{l+1}} (\beta_{l+1}) = \beta_{l+2}$.  Since the image of
 $\zeta_{s_{l+1}}$ is $\upper_{\Gamma, l+1}^{s_{l+1}}$, we obtain, in this way, an
 expression for $\nu_{\Gamma}(\Omega)$ for the $1$-forms
 $\Omega \in {\mathcal T}_{\Gamma, l+1}$ with $\Beg_{\Gamma}(\Omega) = \beta_{l+2}$.
\end{rem}
\section{Geometric Analysis of the Construction}
\label{sec:analysis_construction}
Given the singular branch $\Gamma$, we have constructed the terminal families
\begin{equation*}
  {\mathcal T}_{\Gamma, 0}, \ldots, {\mathcal T}_{\Gamma, g-1}, \mathcal{T}_{\Gamma,g}
\end{equation*}
from the intermediate good families
\begin{equation*}
    {\mathcal G}_{\Gamma, l}^{0},  {\mathcal G}_{\Gamma, l}^{1}, \hdots, \mathcal{G}_{\Gamma,l}^{s_l}
\end{equation*}
where ${\mathfrak n}^{s_{l}} (\beta_{l}) = \beta_{l+1}$ for $0<l<g$ in Section
\ref{sec:construction_cx_bases}.  Recall that
$\mathcal{T}_{\Gamma,g}={\mathcal G}_{\Gamma, g}^{s_g}$ is a terminal family of level $g$
for some sufficiently large $s_g \in {\mathbb Z}_{\geq 1}$ by the proof of Proposition
\ref{pro:from-l-to-next}.
\begin{rem}
\label{rem:generators_construction} 
The ${\mathbb C}[[x]]$-basis $(\Omega_{j}^{s_{g}} : j \in \low_{\Gamma, g}^{s_{g}})$ for
$\Gamma$ provided by our method satisfies
$\nu_{\Gamma} ( \Omega_{2 \nu_{l}}^{s_g}) = \overline{\beta}_{l+1}$ for any $0 \leq l <g$
by Remark \ref{rem:generators_construction_0}.  In particular, we could replace
$\Omega_{2 \nu_{l}}^{s_g}$ with $d f_{\Gamma, l+1} \in {\mathcal M}_{2 \nu_{l}}$ in our
${\mathbb C}[[x]]$-basis.
\end{rem}
\subsection{Step by step analysys}
\label{subsec:step_by_step}
In this subsection we give an explicit interpretation of our step-by-step
construction: what we shall see is that, the \textsc{Main Transformation} in each
stage factors out the contribution of a coefficient $a_{\beta}$ with
$\beta \in {\mathcal E}_{\Gamma}$. This has important geometric implications, as the next
subsection shows.

Let us first describe explicitly the (non-vanishing) terms of lowest order
of the $1$-forms $t (\overline{\Gamma})^{*} (\Omega_j^{s})$ for
$\Omega_{j}^{s} \in {\mathcal G}_{\Gamma, l+1}^{s}$.
\begin{defi}
\label{def:gen_ord_s}
Given $\Omega \in \hat{\Omega}^{1} \cn{2}$, we define
\begin{equation*}
  \nu_{\Gamma, s} (\Omega) = \min \{ \nu_{\overline{\Gamma}} (\Omega) :
  \overline{\Gamma} \in  \fgba{\Gamma}{{\mathfrak n}^{s}(\beta_{l+1})} \}.
\end{equation*}
It is the generic order of $\nu_{\overline{\Gamma}} (\Omega)$ for
$\overline{\Gamma} \in \fgba{\Gamma}{{\mathfrak
    n}^{s}(\beta_{l+1})}$.
\end{defi}  
\begin{lem} 
\label{lem:lot}
Let $0 \leq l < g$ and consider an index $s \geq 1$ and a $1$-form
$\Omega_{j}^{s} \in {\mathcal G}_{\Gamma, l+1}^{s}$.  Let  
$\iota = \min(\Beg_{\Gamma}(\Omega_{j}^{s}), {\mathfrak n}^{s}(\beta_{l+1}) )$.  Then
\begin{equation*}
  t (\overline{\Gamma})^{*} (\Omega_j^{s})  =
  \left( C_{\iota} (\overline{\Gamma}) t^{\nu_{\Gamma, s} (\Omega_{j}^{s})} 
  + O \left( t^{\nu_{\Gamma, s} (\Omega_{j}^{s})+1} \right) \right) dt
\end{equation*}
for any $\overline{\Gamma} \in \fgb{\Gamma}{{\mathfrak n}^{s}(\beta_{l+1})}$.
Moreover, $C_{\iota}$ is a non-zero linear function of
$a_{\iota, \overline{\Gamma}}$ in
$\fgb{\Gamma}{{\mathfrak n}^{s}(\beta_{l+1})}$, and it is constant if and
only if $j \in \low_{\Gamma, l+1}^{s}$.
 \end{lem}    
\begin{proof}
  We have $\Omega_{j}^{s} \in \Dd{\beta', <}{l+1}$ for some
  $\beta' \in {\mathcal E}_{\Gamma}$ by Condition \ref{good6} of Definition
  \ref{def:level} (notice that $s\geq 1$). So, we obtain
  \begin{equation*}
    t (\overline{\Gamma})^{*} (\Omega_j^{s}) = \left( t^{k} \sum_{r=0}^{\infty}
      C_{{\mathfrak n}^{r} (\beta')} t^{{\mathfrak n}^{r} (\beta')} \right) dt,
  \end{equation*}
  cf. Definition \ref{def:max}. Since $C_{\eta}$ depends on
  $(a_{\beta, \overline{\Gamma}})_{\beta \leq \eta}$ for
  $\overline{\Gamma} \in \fgb{\Gamma}{\beta_{l+1}}$, we deduce that for any
  $\eta < {\mathfrak n}^{s}(\beta_{l+1})$,
  $ C_{\eta}\equiv C_{\eta}(\overline{\Gamma}) \equiv C_{\eta}(\Gamma)$ whenever
  $\overline{\Gamma}\in \fgb{\Gamma}{{\mathfrak n}^{s}(\beta_{l+1})}$. Notice that
  $C_{\eta}(\Gamma) =0$ for any $\eta < \Beg_{\Gamma}(\Omega_{j}^{s})$.  In
  particular, for $\eta <\iota$, we obtain
  $C_{\eta} \equiv C_{\eta}(\overline{\Gamma})\equiv 0$ for
  $\overline{\Gamma}\in \fgb{\Gamma}{{\mathfrak n}^{s}(\beta_{l+1})}$.

  Take $j \in \low_{\Gamma, l+1}^{s}$. Since
  $\Beg_{\Gamma}(\Omega_{j}^{s}) < {\mathfrak n}^{s}(\beta_{l+1})$ by Condition
  \ref{good5} of Definition \ref{def:good}, it follows that
  $C_{\Beg_{\Gamma}(\Omega_j^s)}
  \equiv C_{\Beg_{\Gamma}(\Omega_{j}^{s})}(\Gamma) \neq 0$ for any branch
  $\overline{\Gamma}\in\fgb{\Gamma}{{\mathfrak n}^{s}(\beta_{l+1})}$.

  Finally, consider $j \in \upper_{\Gamma, l+1}^{s}$. Condition \ref{good5} implies
  $\iota = {\mathfrak n}^{s}(\beta_{l+1})$.  Also,by
  Conditions \ref{good3} and \ref{good4}, $\Omega_{j}^{s}$, belongs to either $\Dd{\beta_{l+1}, <}{l+1}$ or to
  $\hatDD{{\mathfrak n}(\beta_{l+1}), <}{0}{l+1}$.  In both cases, the restriction of
  $C_{\iota}(\overline{\Gamma})$ to $\fgb{\Gamma}{{\mathfrak n}^{s}(\beta_{l+1})}$ is a
  non-constant linear polynomial on $a_{\iota, \overline{\Gamma}}$ since
  $\Omega_{j}^{s}$ has leading variable of degree $0$.  Recall that if
  $\Omega_j^s\in \Dd{\beta_{l+1}, <}{l+1}$, this last property follows because
  $j\in \upper_{\Gamma,l+1}^s$ implies by definition that
  $\Beg_{\Gamma}(\Omega_{j}^{s}) > \beta_{l+1}$.
\end{proof}
\begin{defi}
\label{def:sing_dir_j}
Let $ 0\leq l < g$, $s \geq 1$ and
$\Omega_{j}^{s} \in {\mathcal G}_{\Gamma, l+1}^{s}$ with
$j \in \upper_{\Gamma, l+1}^{s}$.  We define
\begin{equation*}
  {\mathcal D}_{\Gamma} ({\mathfrak n}^{s}(\beta_{l+1}), j) =
  \{ c \in \mathbb{C} : C_{{\mathfrak n}^{s}(\beta_{l+1})}(c) = 0 \} ,
\end{equation*}
cf. Lemma \ref{lem:lot}. It is the set of values
$a_{{\mathfrak n}^{s}(\beta_{l+1}), \overline{\Gamma}}$ such that any curve
$\overline{\Gamma} \in \fgba{\Gamma}{{\mathfrak n}^{s}(\beta_{l+1})}$ 
with such coefficient in the term $t^{\mathfrak{n}^s(\beta_{l+1})}$
satisfies $\nu_{\overline{\Gamma}} (\Omega_{j}^{s}) > \nu_{\Gamma, s}
(\Omega_{j}^{s})$. It is, by Lemma \ref{lem:lot} either a singleton or the empty
set, but the set notation will simplify the notation later on.
\end{defi}
\begin{rem}
\label{rem:sd_l}
Let
$\overline{\Gamma} \in \fgba{\Gamma}{{\mathfrak n}^{s}(\beta_{l+1})}$. 
Then, given $j \in \upper_{\Gamma, l+1}^{s}$, we have the following sequence of equivalences:
\begin{multline*}
  \Beg_{\overline{\Gamma}} (\Omega_{j}^{s}) > {\mathfrak n}^{s}(\beta_{l+1})
  \Leftrightarrow
  a_{{\mathfrak n}^{s}(\beta_{l+1}), \overline{\Gamma}}
  \in {\mathcal D}_{\Gamma} ({\mathfrak n}^{s}(\beta_{l+1}), j)\\
  \Leftrightarrow
  \nu_{\overline{\Gamma}} (\Omega_{j}^{s}) >  \nu_{\Gamma, s} (\Omega_{j}^{s}).
\end{multline*}
\end{rem}
Next, we describe which coefficients the good and terminal families that appear in
the construction of a ${\mathbb C}[[x]]$-basis depend on.
\begin{pro}
  \label{pro:term_good_dep}
  Let $0 \leq l <g$ and $s \geq 0$. Then the family ${\mathcal G}_{\Gamma, l+1}^{s}$
  depends just on
  $(a_{\beta, \overline{\Gamma}})_{\beta \in {\mathcal E}_{\Gamma}, \ \beta < {\mathfrak
      n}^{s}(\beta_{l+1})}$ for any $\overline{\Gamma} \in \fga{\Gamma}$.  In other
  words: ${\mathcal G}_{\Gamma, l+1}^{s} = {\mathcal G}_{\overline{\Gamma}, l+1}^{s}$ for
  any
  $\overline{\Gamma} \in \fgba{\Gamma}{{\mathfrak n}^{s}(\beta_{l+1})}$.  
  Moreover, the family ${\mathcal T}_{\Gamma, l}$ depends just
  on
  $(a_{\beta, \overline{\Gamma}})_{\beta \in {\mathcal E}_{\Gamma}, \ \beta < \beta_{l+1}}$ 
  for any $\overline{\Gamma} \in \fga{\Gamma}$.
\end{pro}
\begin{proof} 
  Fix $0 \leq l < g$.  Let us assume that ${\mathcal G}_{\Gamma, l+1}^{s}$
  satisfies the thesis for some $s \in {\mathbb Z}_{\geq 0}$ with
  ${\mathfrak n}^{s}(\beta_{l+1}) < \beta_{l+2}$.  Consider
  $j, k \in {\mathbb N}_{2 \nu_{l+1}}$ with $j \in \low_{\Gamma, l+1}^{s}$ and
  $k \in \upper_{\Gamma, l+1}^{s}$.
  Note that $\Beg_{\overline{\Gamma}}(\Omega_{j}^{s})$,
  $\nu_{\overline{\Gamma}} (\Omega_{j}^{s})$ and the non-vanishing term of lowest
  order of $(\overline{\Gamma})^{*} \Omega_{j}^{s}$ are constant for those
  $\overline{\Gamma} \in \fgba{\Gamma}{{\mathfrak n}^{s}(\beta_{l+1})}$ 
  by Lemma \ref{lem:lot}. It also implies that the
  non-vanishing term of lowest order of $(\overline{\Gamma})^{*} \Omega_{k}^{s}$
  is constant in the family
  $\overline{\Gamma} \in \fgb{\Gamma}{{\mathfrak n}^{s+1}(\beta_{l+1})}$, and that
  $\Beg_{\overline{\Gamma}} (\Omega_{k}^{s}) = {\mathfrak n}^{s}(\beta_{l+1})$
  holds either for all
  $\overline{\Gamma} \in \fgba{\Gamma}{{\mathfrak n}^{s+1}(\beta_{l+1})}$, or for none.  
  If $\Beg_{\Gamma} (\Omega_{k}^{s}) = {\mathfrak n}^{s}(\beta_{l+1})$ 
  then the function
  $\overline{\Gamma} \mapsto \nu_{\overline{\Gamma}} (\Omega_{k}^{s})$ is also
  constant in
  $\fgba{\Gamma}{{\mathfrak n}^{s+1}(\beta_{l+1})}$ by Lemma
  \ref{lem:lot}.  Definition \ref{def:next} (the \textsc{Main
    Transformation}) makes clear  that
  ${\mathcal G}_{\overline{\Gamma}, l+1}^{s+1} = {\mathcal G}_{\Gamma, l+1}^{s+1}$
  for any
  $\overline{\Gamma} \in \fgba{\Gamma}{{\mathfrak n}^{s+1}(\beta_{l+1})}$.

  Consider the case $l + 1 < g$.  Let $s_{l+1} \in {\mathbb Z}_{\geq 1}$ with
  ${\mathfrak n}^{s_{l+1}} (\beta_{l+1}) = \beta_{l+2}$.  The previous discussion,
  together with an induction argument, imply that if
  ${\mathcal G}_{\Gamma, l+1}^{0}$ satisfies the thesis then
  ${\mathcal T}_{\Gamma, l+1}$ does. It is clear that ${\mathcal T}_{\Gamma, 0}$
  satisfies the thesis since it depends on nothing.  So in order to end the proof
  by induction on $l$, it suffices to show that if ${\mathcal T}_{\Gamma, l}$
  satisfies the thesis for some $0 \leq l < g$ so it does
  ${\mathcal G}_{\Gamma, l+1}^{0}$. This is clear by construction, as the
  $j$-th approximate root depends only on the coefficients before
  $\beta_j$, that is: $f_{\overline{\Gamma}, l+1} = f_{\Gamma, l+1}$ for
  any $\overline{\Gamma} \in \fgba{\Gamma}{\beta_{l+1}}$.
\end{proof} 
\begin{rem}
  \label{rem:coef_to_coef}
  In the process of building a ${\mathbb C}[[x]]$-basis of
  $\overline{\Gamma} \in \fga{\Gamma}$, we can think of
  ${\mathcal T}_{\overline{\Gamma}, l}$ (resp.
  ${\mathcal G}_{\overline{\Gamma}, l+1}^{s}$) as a result of the process when we
  just consider coefficients $a_{\beta, \overline{\Gamma}}$ with
  $\beta < \beta_{l+1}$ (resp. $\beta <{\mathfrak n}^{s}(\beta_{l+1})$) for
  $0 \leq l < g$.  Thus, when we build
  ${\mathcal G}_{\overline{\Gamma}, l+1}^{s+1}$ from
  ${\mathcal G}_{\overline{\Gamma}, l+1}^{s}$, we are considering the contribution
  to the process of the coefficient
  $a_{{\mathfrak n}^{s}(\beta_{l+1}), \overline{\Gamma}}$.
\end{rem}
Since  ${\mathcal G}_{\overline{\Gamma}, l+1}^{s}$ does not depend on
$\overline{\Gamma} \in \fgba{\Gamma}{{\mathfrak n}^{s}(\beta_{l+1})}$, 
it is natural to
study whether the functions
$\overline{\Gamma} \mapsto \nu_{\overline{\Gamma}} (\Omega_{j}^{s})$ are constant in
$\fgba{\Gamma}{{\mathfrak n}^{s}(\beta_{l+1})}$.  This divides the
$1$-forms $\Omega_{j}^{s}$ in ${\mathcal G}_{\Gamma, l+1}^{s}$, with $s \geq 1$, in two
sets depending on that property. The non-constant case is interesting, because then
$\nu_{\overline{\Gamma}} (\Omega_{j}^{s})$ depends on coefficients that will be processed
in future steps of the method.  This alternative, which is the same as Definition
\ref{def:good}, is introduced in the following remark, which is a direct combination of
the results in Lemma \ref{lem:lot} and Proposition \ref{pro:term_good_dep}.  
\begin{rem}
  Neither of the sets $\low_{\overline{\Gamma}, l+1}^{s}$ and 
  $\upper_{\overline{\Gamma}, l+1}^{s}$ depends on
  $\overline{\Gamma} \in \fgba{\Gamma}{{\mathfrak n}^{s}(\beta_{l+1})}$. 
  Moreover, for $j\in \low_{\Gamma,l+1}^s$, the analogous property holds for the integers
  $\Beg_{\overline{\Gamma}}(\Omega_{j}^{s})$ and
  $\nu_{\overline{\Gamma}} (\Omega_{j}^{s})$, and the non-vanishing term of lowest
  order of $(\overline{\Gamma})^{*} \Omega_{j}^{s}$.

  Note also that, for $k \in \upper_{\Gamma, l+1}^{s}$, if
  ${\mathfrak n}^{s}(\beta_{l+1}) < \beta_{l+2}$ then
  $\Beg_{\overline{\Gamma}}(\Omega_{k}^{s})$ and
  $\nu_{\overline{\Gamma}} (\Omega_{k}^{s})$ are non-constant in
  $\fgba{\Gamma}{{\mathfrak n}^{s}(\beta_{l+1})}$.
  Indeed, the coefficient of the non-vanishing term of lowest order of
  $(\overline{\Gamma})^{*} \Omega_{k}^{s}$ is a non-constant linear polynomial in
  $a_{{\mathfrak n}^{s}(\beta_{l+1}), \overline{\Gamma}}$ for
  $\overline{\Gamma} \in \fgba{\Gamma}{{\mathfrak n}^{s}(\beta_{l+1})}$.
\end{rem}
\begin{rem}
  Let $0 \leq l <g$. There is a set of special $1$-forms in
  ${\mathcal G}_{\Gamma, l+1}^{1}$, namely the $1$-forms $\Omega_{k}^{1}$ with
  $k \in \upper_{\Gamma, l+1}^{1} \setminus \overline\upper_{\Gamma, l+1}$.  These
  $1$-forms belong to $\hatDD{{\mathfrak n}(\beta_{l+1}),<}{0}{l+1}$ by property
  \eqref{g1p3} of Proposition \ref{pro:aux1}. We can improve the result in Proposition
  \ref{pro:term_good_dep}.  Indeed $\Omega_{k}^{1}$ does not depend on 
  $\overline{\Gamma} \in \fgba{\Gamma}{\beta_{l+1}}$.  This is because when we define
  $\Omega_{k}^{1} = \Omega_{k}^{0} - c x^{d} \Omega_{j}^{0}$ (cf. Definition
  \ref{def:good1}, the \textsc{Main Transformation} at stage $0$), we have
  $\Omega_{j}^{0}, \Omega_{k}^{0} \in \oDD{\beta_{l+1}, <}{r}{ l+1} \cap
  \tDd{\beta_{l+1}}{l+1}$ for some $0 \leq r \leq n_{l+1}-2$ by Lemma
  \ref{rem:conflicting_pairs} and Corollaries \ref{cor:f_dominates_omega} and
  \ref{cor:f_omega}.  Hence, even if the non-vanishing terms of lowest order of
  $(\overline{\Gamma})^{*} \Omega_{j}^{0}$ and $(\overline{\Gamma})^{*} \Omega_{k}^{0}$
  are non-constant in $\fgba{\Gamma}{\beta_{l+1}}$, their quotient
  is of the form $\frac{\kappa_1 a_{\beta_{l+1}}}{k_2 a_{\beta_{l+1}}}$, which is
  constant. This independence on $a_{\beta_{l+1}, \overline{\Gamma}}$ is another
  side of the dicritical nature of $\Omega_{k}^{1}$ described in Proposition \ref{pro:jG}.
\end{rem}  
\subsection{Properties of the $1$-forms in the construction}
\label{subsec:properties_forms_construction}
In order to prove our main results, 
namely  Theorems \ref{teo:collection} and \ref{teo:lambda_on_s_b},
we need to understand what types of forms appear in the construction of a
${\mathbb C}[[x]]$-basis of $\Gamma$.  Indeed, our construction provides interesting
additional structure. For instance, we can associate a formal invariant curve to every
$1$-form in a terminal family, a so called companion curve. 
Recall, at this point, Definition \ref{def:Dj}, and Propositions \ref{pro:jG} and
\ref{pro:wu} in Subsection \ref{subsec:foliation}.
\begin{pro}
\label{pro:type_forms}
Consider the construction of a ${\mathbb C}[[x]]$-basis ${\mathcal B}$ for $\Gamma$
in section \ref{sec:construction_cx_bases}.  Then $\Omega$ \touches{} $\Gamma$ to
order $ \Beg_{\Gamma}(\Omega)$ for any $\Omega \in {\mathcal T}_{\Gamma,
  g}$. Moreover, we have ${\mathcal T}_{\Gamma, 0} = (dx, dy)$ and given
$0 \leq l <g$, the $1$-forms in
$({\mathcal B} \cap {\mathcal T}_{\Gamma, l+1}) \setminus {\mathcal T}_{\Gamma, l}$
are of one of the following types:
\begin{enumerate}
\item \label{type:dicritical}
  $\Omega_{k}^{s_{l+1}} \in \hatDD{{\mathfrak n}(\beta_{l+1}),<}{0}{l+1}$ with
  $\beta_{l+1} < \Beg_{\Gamma}(\Omega_{k}^{s_{l+1}}) \leq \beta_{l+2}$ and
  $k \not \in \overline\upper_{\Gamma, l+1}$. There is a unique companion curve of
  $\Omega_{k}^{s_{l+1}} $ that is equisingular to $\Gamma_{<\beta_{l+2}}$.
  Moreover, $\Omega_{k}^{s_{l+1}} $ fans $\Gamma$ to genus $l+1$.
\item \label{type:induced_c} $ \Omega_{k}^{s_{l+1}} \in \Dd{\beta_{l+1},<}{l+1}$ if
  $k \in \overline\upper_{\Gamma, l+1}$.  In this case, we have $l+1 < g$ and
  $\Beg_{\Gamma}(\Omega_{k}^{s_{l+1}}) = \beta_{l+2}$.  There is a unique formal
  companion curve $\overline{\Gamma}$ of $\Omega_{k}^{s_{l+1}}$ that is
  equisingular to $\Gamma_{<\beta_{l+2}}$.  We have
\begin{equation}
  \label{equ:no_collection} 
  \nu_{\Gamma} (\Omega_{k}^{s_{l+1}} ) = \nu_{\Gamma}(\Omega_j) +
  \overline{\beta}_{l+2} - \overline{\beta}_{l+1} 
\end{equation}
for some $\Omega_j \in {\mathcal T}_{\Gamma, l}$ with
$\nu_{\Gamma} (\Omega_j) \in [\beta_{l+1}]_{e_l}$. Moreover, ${\mathfrak D}_{l+1}$
is invariant by the strict transform ${\mathfrak F}$ of the foliation
$\Omega_{k}^{s_{l+1}}=0$ by $\pi =\pi_{1} \circ \ldots \circ \pi_{\iota}$ where
$\iota =\hat{\theta}^{-1}(\beta_{l+1})$.  Indeed, all trace points of
${\mathfrak D}_{l+1}$ in $E_{\iota}$ are regular points of ${\mathfrak F}$ except
$P_{\Gamma, \iota} = P_{\overline{\Gamma}, \iota}$.
\item \label{type:induced_s} $\Omega_{k}^{s_{l+1}} =f_{\Gamma, l+1}^{m} \omega$
  with $0 < m < n_{l+1}-1$, $\omega \in {\mathcal T}_{\Gamma, l}$,
  $\Beg_{\Gamma}(\omega)= \beta_{l+1}$ and
  $\nu_{\Gamma} (\omega) \in [\beta_{l+1}]_{e_{l}}$. Any formal companion curve of
  $\Omega_{k}^{s_{l+1}}$ is equisingular to $\Gamma_{<\beta_{l+1}}$.
\end{enumerate}
\end{pro}  
\begin{proof}
  Since ${\mathcal T}_{\Gamma, 0} = (dx, dy)$ by construction and both $1$-forms
  \touch{} $\Gamma$ to some order, let us focus on the $1$-forms
  $\Omega \in {\mathcal T}_{\Gamma, l+1} \setminus {\mathcal T}_{\Gamma, l}$ for
  any $0 \leq l < g$.

  Recall that
  ${\mathfrak n}^{s_{l+1}} (\beta_{l+1}) = \beta_{l+2}$ for $0 < l < g-1$ and $s_g$
  is chosen so that ${\mathcal G}_{\Gamma, g}^{s_g}$ is a terminal family of level
  $g$.  We have ${\mathcal T}_{\Gamma, l+1} = {\mathcal G}_{\Gamma, l+1}^{s_{l+1}}$
  and in particular ${\mathcal T}_{\Gamma, l+1}$ is a good family of level $l+1$
  for $0 \leq l < g$.

  Consider $\Omega:=\Omega_{k}^{s_{l+1}}$ in
  ${\mathcal T}_{\Gamma, l+1} \setminus {\mathcal T}_{\Gamma, l}$.  
  Firstly, if $\Beg_{\Gamma}(\Omega) > \beta_{l+1}$, then it is of one of the first
  two types, by Conditions \ref{good3} and \ref{good4} in Definition
  \ref{def:good}. If $\Omega \in \hatDD{{\mathfrak n}(\beta_{l+1}),<}{0}{l+1}$, the
  properties stated in the first item for $\Omega$ are a consequence of Proposition
  \ref{pro:jG} and Corollary \ref{cor:wh0}.  If, otherwise,
  $k \in \overline\upper_{\Gamma, l+1}$, the properties stated for $\Omega$ in the
  second item derive from Proposition \ref{pro:wu}, except $l+1<g$ and equation
  \eqref{equ:no_collection}.  Notice that since $k \in \low_{\Gamma, g}^{s_g}$ for
  $l+1=g$ by Proposition \ref{pro:from-l-to-next} and
  $\overline\upper_{\Gamma, g} \subset \upper_{\Gamma, g}^{s_g}$, we have
  necessarily that $l+1 <g$.  In particular, $\Beg_{\Gamma}(\Omega)= \beta_{l+2}$.
  We also know that $\Omega_{k}^{0} = f_{\Gamma, l+1}^{n_{l+1}-1} \Omega_j$ for
  some $\Omega_{j} \in {\mathcal T}_{\Gamma, l}$ with
  $\Omega_j \in \oDD{\beta_{l+1}, <}{0}{l+1}$ by Definition \ref{def:und}; this
  gives $\nu_{\Gamma} (\Omega_j) \in [\beta_{l+1}]_{e_l}$ by Lemma
  \ref{lem:cong_nu}. Now, Equation \eqref{equ:no_collection} holds by Remark
  \ref{rem:chu}.
 
  Let us consider the remaining case: $\Beg_{\Gamma}(\Omega)=\beta_{l+1}$.  By
  construction, $\Omega = \Omega_{k}^{s_{l+1}} = \Omega_{k}^{0}$ and
  $k \in \cap_{s \in {\mathbb N}_{s_{l+1}}} \low_{\Gamma, l+1}^{s}$. Thus
  $\Omega= f_{\Gamma, l+1}^{a} \Omega_{j}$ for some $0 \leq a < n_{l+1}$ and
  $\Omega_{j} \in {\mathcal T}_{\Gamma, l}$.  Since
  $\Omega \not \in {\mathcal T}_{\Gamma, l}$, we have $a \geq 1$.  Moreover, $k$
  does not belong to $\upper_{l+1}$ since $k \in \low_{\Gamma, l+1}^{1}$
  and Condition \ref{good5} of Definition \ref{def:good} holds. Hence,
  $a < n_{l+1}-1$.  Also, the fact that $k \in \low_{\Gamma, l+1}^{1}$ implies
  $\Beg_{\Gamma}(\Omega_j) = \beta_{l+1}$ by Definition \ref{def:lu}.  Using Lemma
  \ref{lem:cong_nu}, and the properties of terminal families, we conclude that
  $\Omega_j \in \oDD{\beta_{l+1}, <}{0}{ l+1}$ and that
  $\nu_{\Gamma} (\Omega_{j}) \in [\beta_{l+1}]_{e_{l}}$.  Thus, $\Omega_j$
  \touches{} $\Gamma$ to order $ \beta_{l+1}$ and has a unique formal companion
  curve equisingular to $\Gamma_{<\beta_{l+1}}$ by Proposition
  \ref{pro:unique_companion}. As a consequence, $\Omega$ \touches{} $\Gamma$ to
  order $ \beta_{l+1}$ because
  $\{ f_{\Gamma, l+1} = 0 \} = \Gamma_{< \beta_{l+1}}$.  
\end{proof}
%
%
\subsection{Properties of the semimodule}
\label{subsec:properties_semimodule}
Now, let us introduce a couple of consequences of the construction of a
${\mathbb C}[[x]]$-basis of $\Gamma$ for the semimodule $\Lambda_{\Gamma}$.
\begin{cor}
  We have
  $\Lambda_{\Gamma} \cap {\mathbb Z}_{\leq \overline{\beta}_{l+1}} \subset (e_l)
  \cup [\beta_{l+1}]_{e_l}$
  for any $0 \leq l < g$.
\end{cor}
\begin{proof}
  Consider a terminal family ${\mathcal T}_{\Gamma, g}$ constructed as described in
  section \ref{sec:construction_cx_bases}.  It contains a ${\mathbb C}[[x]]$-basis.
  Since $\nu_{\Gamma} (\Omega) > \overline{\beta}_{l+1}$ for any
  $\Omega \in {\mathcal T}_{\Gamma, g} \setminus {\mathcal T}_{\Gamma, l}$ and
  $n \in (e_l)$, it suffices to show
  $\nu_{\Gamma} ({\mathcal T}_{\Gamma, l}) \subset (e_l) \cup [\beta_{l+1}]_{e_l}$.
  This an immediate consequence of Corollary \ref{cor:cong_term}.
\end{proof}
Let us restate here Theorem \ref{teo:collection}
\begin{pro}
\label{pro:L_is_c_coll}
Let $\Gamma$ be a singular branch of planar curve.  Then $\Lambda_{\Gamma}$ is a
${\mathcal C}$-collection.
\end{pro}
\begin{proof}
  Consider the construction of a terminal family ${\mathcal T}_{\Gamma, g}$ of
  $\Gamma$ in section \ref{sec:construction_cx_bases}. It contains a
  ${\mathbb C}[[x]]$-basis of $\Gamma$.  It suffices to show that if
  $\lambda \in \Lambda_{\Gamma}$ with $\lambda \leq \overline{\beta}_{l}$ and
  $l \in {\mathbb N}_{g}$ then
  $\lambda + \overline{\beta}_{r} - \overline{\beta}_{l} \in \Lambda_{\Gamma}$ for
  any $l < r \leq g$.  However, it suffices to show the result for $r=l+1$ since,
  if we do, we obtain
  \begin{equation*}
    \lambda + \overline{\beta}_{l+1} - \overline{\beta}_{l} \in
    {\mathbb Z}_{\leq \overline{\beta}_{l+1}} \cap \Lambda_{\Gamma}
  \end{equation*}
and thus 
\begin{equation*}
  \lambda + \overline{\beta}_{l+2} - \overline{\beta}_{l} =
  (\lambda + \overline{\beta}_{l+1} - \overline{\beta}_{l} ) + 
  (\overline{\beta}_{l+2} - \overline{\beta}_{l+1} )
  \in {\mathbb Z}_{\leq \overline{\beta}_{l+2}} \cap \Lambda_{\Gamma}
\end{equation*}
if $l+2 \leq g$ and so on.

We have $\lambda = \nu_{\Gamma} (\Omega_{j}) + dn$ for
some $\Omega_{j} \in {\mathcal T}_{\Gamma, l-1}$ and $d \in {\mathbb Z}_{\geq 0}$
since $\nu_{\Gamma} (\Omega) > \overline{\beta}_{l}$ for any
$\Omega \in {\mathcal T}_{\Gamma, g} \setminus {\mathcal T}_{\Gamma, l-1}$.  It
suffices to prove the result for $\lambda = \nu_{\Gamma} (\Omega_{j})$.  This is a
direct consequence of Remark \ref{rem:chu}.
\end{proof}
\subsection{Hidden coefficients}
The dimension of the analytic moduli of $\Gamma$ is equal to 
the cardinal of ${\mathcal E}_{\Gamma}  \setminus (\Lambda - n)$
where $\Lambda$ is the generic semimodule $\Lambda_{\overline{\Gamma}}$ 
for $\overline{\Gamma} \in \fga{\Gamma}$ \cite[Remark 5.11]{ayuso2024construction}. 
Let us see
that in many cases we can find ``hidden" coefficients $\beta \in {\mathcal E}_{\Gamma}$
such that $\beta + n \not \in \Lambda_{\overline{\Gamma}}$ for any $\overline{\Gamma} \in \fga{\Gamma}$. 
In this way, our methods recovers results by Zariski \cite[Chapter IV.2]{Zariski:moduli} in a simpler way. 

Let $\Gamma$ be a singular branch of planar curve. Assume $g \geq 2$ and define 
\[ \beta^{*} = \min \{ \beta \in {\mathcal E}_{\Gamma} : \beta > \beta_{2} \ \& \ \beta \not \in (e_1) \} . \]
Notice that since $\beta_{2} + e_2$ and $\beta_2 + 2 e_2$ belong to ${\mathcal E}_{\Gamma}$ and at least one of
them is not a multiple of $e_1$, we deduce $\beta^{*} \leq \beta_2 + 2 e_2$.
\begin{pro}
Assume that  $n_1=n_2=2$, $\beta^{*} - \beta_2 = 2e_2$ and 
$\beta_1 = (n_1 + 1) e_1$ do not hold simultaneously. Then 
$\beta^{*} +n \not \in \Lambda_{\Gamma}$. In particular,   $\beta^{*} +n \not \in \Lambda_{\overline{\Gamma}}$
for any $\overline{\Gamma} \in \fga{\Gamma}$. 
\end{pro}
\begin{proof}
Consider a terminal family ${\mathcal T}_{\Gamma, g}$ and a ${\mathbb C}[[x]]$-basis ${\mathcal B}$ for $\Gamma$. 
Assume that $\beta^{*} +n \in \Lambda_{\Gamma}$. 
Since $\nu_{\Gamma} (\Omega) \geq n + \overline{\beta}_2$ for 
$\Omega \in {\mathcal T}_{\Gamma, g} \setminus {\mathcal T}_{\Gamma, 1}$ and 
\begin{equation}
\label{equ:aux_comb}
\beta^{*} + n  \leq \beta_2 + 2 e_{2} +n  \leq \beta_2 + n_{2} e_{2} + n \leq   \beta_2 + e_1 + n \leq \beta_2 + \beta_1
 \end{equation}
 and then 
 \begin{equation*}
\beta^{*} + n    \leq \beta_2 + \beta_1
  \leq \beta_2 + (n_1 -1) \beta_1 = \overline{\beta}_2, 
 \end{equation*}
it follows that $\beta^{*} + n = \nu_{\Gamma} (\Omega) + l n$ for some
$\Omega \in {\mathcal T}_{\Gamma, 1}$ and $l \in {\mathbb Z}_{\geq 0}$.
So, we can consider the families ${\mathcal G}_{\Gamma,1}^{s} = (\Omega_{1}^{s}, \hdots, \Omega_{2 \nu_1}^{s})$
for $0 \leq s \leq s_1$ where ${\mathfrak n}^{s_1}(\beta_1) = \beta_2$.
The property $\beta^{*} \not \in (e_1)$ implies $\Omega \in  \oDD{\beta_{2}, <}{0}{2}$ by Lemma \ref{lem:cong_nu}
and thus $\Omega = \Omega_{k}^{s_1}$ for some $k \in \upper_{\Gamma, 1}^{s_1}$. We obtain 
 \begin{equation*}
   \nu_{\Gamma} (\Omega_{k}^{s_1}) - \nu_{\Gamma} (\Omega_{\zeta_{s_{1}}^{-1}(k)}^{0}) -
   ( \Beg_{\Gamma}(\Omega_{k}^{s_{1}})   - \beta_{1})
   \in n {\mathbb Z}_{\geq 0}
 \end{equation*}
 by Remark \ref{rem:tracking_leading}. As a consequence, 
 \[  \nu_{\Gamma} (\Omega) - \beta_{2} =
 \nu_{\Gamma} (\Omega_{k}^{s_1}) - \Beg_{\Gamma}(\Omega_{k}^{s_{1}}) = m \beta_1 + r n \]
 holds for some $m, r \in {\mathbb Z}_{\geq 0}$ with $m+r \geq 1$ and thus
 \[  \beta^{*} + n - \beta_2 =  m \beta_1 + (r+l) n .\]
 Since $0 < \beta^{*} - \beta_2 \leq 2  e_1 < n$ by equation \eqref{equ:aux_comb},  
 we get $\beta^{*} - \beta_2 \not \in (n)$ and hence $m \geq 1$.
 Since $\beta^{*} - \beta_{2} +n \leq \beta_1$ by equation \eqref{equ:aux_comb}, 
 we deduce $m=1$, $r=l=0$ and morever, the inequalities in equation 
 \eqref{equ:aux_comb} are indeed equalities. In particular, we obtain 
 $\beta^{*} = \beta_2 + 2 e_2$, $n_2 =2$ and $\beta_1 = (n_1 + 1) e_1$.
 
Notice that $\zeta_{s_{1}}^{-1}(k) \in \upper_{\Gamma, 1}^{1}$ and 
$\nu_{\Gamma} (\Omega_{\zeta_{s_{1}}^{-1}(k)}^{0}) - \beta_1 = \beta_1$ since $r=0$. 
This implies $\zeta_{s_{1}}^{-1}(k) = 4$ and $\Omega_{4}^{0} = f_{\Gamma, 1} dy$. 
Since $4 \in \upper_{\Gamma, 1}^{1}$, we obtain $n_1 =2$ by Definition \ref{def:lu}. 
\end{proof}
\section{Generation of the semimodule $\Lambda_{\Gamma}$}
\label{sec:generation_semimodule}
In  this section, given the singular branch $\Gamma$, we discuss several types
of generator sets for its semimodule $\Lambda_{\Gamma}$ and how they naturally induce
generator sets of $1$-forms with meaningful geometrical properties.  In particular,
we prove Theorems \ref{teo:lambda_on_s_b} and \ref{teo:gen_sg_in_sbasis} (pages \pageref{teo:lambda_on_s_b} and \pageref{teo:gen_sg_in_sbasis}). As will become apparent, all the results
follow from the properties of good and terminal families.

In the Introduction, we defined the concept of ${\mathbb C}[[x]]$-collections and
${\mathcal C}$-collections for subsets of ${\mathbb Z}_{\geq 1}$.  We now define two
other types of collections related to $\Gamma$ and $\Lambda_{\Gamma}$.
\begin{defi}
  Let $S$ be a subset of ${\mathbb Z}_{\geq 1}$. We say that it is an
  ${\mathcal S}_{\Gamma}$-{\it collection} (or just ${\mathcal S}$-collection if
  $\Gamma$ is implicit) if $S + {\mathcal S}_{\Gamma} \subset S$.
\end{defi}
\begin{defi}
  Let $S$ be a subset of ${\mathbb Z}_{\geq 1}$. We say that it is a
  ${\mathcal C}_{\Gamma}^{w}$-{\it collection} (or just
  ${\mathcal C}^{w}$-collection if $\Gamma$ is implicit) if
  $S + {\mathcal S}_{\Gamma} \subset S$ and given any choice of $1 \leq l <g$ and
  $\lambda \leq \overline{\beta}_{l}$ with
  $\lambda \in S \cap [\beta_l]_{e_{l-1}}$, we have
  $\lambda + \overline{\beta}_{l+1} - \overline{\beta}_{l} \in S$.
\end{defi}
\begin{rem}
\label{rem:cw_gen}
Let $S$ be a ${\mathcal C}^{w}$-collection.  Assume
$\lambda \in S \cap [\beta_l]_{e_{l-1}}$ with $\lambda \leq \overline{\beta}_{l}$
and $l < k \leq g$. Then
$\lambda + \overline{\beta}_{k} - \overline{\beta}_{l} \in S$ holds. This is clear
for $k=l+1$ by definition. Moreover, if
$\lambda + \overline{\beta}_{k} - \overline{\beta}_{l} \in S$ for some $l < k < g$
then
\begin{equation*}
  \lambda + \overline{\beta}_{k+1} -  \overline{\beta}_{l} = 
  (\lambda + \overline{\beta}_{k} -  \overline{\beta}_{l}) +
  \overline{\beta}_{k+1} -  \overline{\beta}_{k}.
\end{equation*}
This expression belongs also to $S$ by definition of ${\mathcal C}^{w}$-collection
since
$\lambda + \overline{\beta}_{k} - \overline{\beta}_{l} \leq \overline{\beta}_{k}$
and
$[\lambda + \overline{\beta}_{k} - \overline{\beta}_{l}]_{e_{k-1}} =
[\beta_{k}]_{e_{k-1}}$.
\end{rem}
\begin{rem}
  We have
  \begin{equation*}
    \mathcal{C}-\mathrm{collection} \Rightarrow
    \mathcal{C}^{w}-\mathrm{collection} \Rightarrow \mathcal{S}-\mathrm{collection}
    \Rightarrow {\mathbb C}[[x]]-\mathrm{collection} .
  \end{equation*}
  As a consequence $\Lambda_{\Gamma}$ is a ${\mathcal J}$-collection for
  $\mathcal{J} \in \{ {\mathbb C}[[x]], \mathcal{S}, \mathcal{C}^{w}, \mathcal{C}
  \}$ by Proposition \ref{pro:L_is_c_coll}.
\end{rem}
\begin{rem}
  Let
  $\mathcal{J} \in \{ {\mathbb C}[[x]], \mathcal{S}, \mathcal{C}^{w}, \mathcal{C}
  \}$.  It is clear that an arbitrary intersection of $\mathcal{J}$-collections is
  also a $\mathcal{J}$-collection.
\end{rem}
The previous remark implies that the next definition makes sense.
\begin{defi}
  \label{def:gen_collection}
  Let
  $\mathcal{J} \in \{ {\mathbb C}[[x]], \mathcal{S}, \mathcal{C}^{w}, \mathcal{C}
  \}$.  Given a subset $S$ of ${\mathbb Z}_{\geq 1}$, we denote by
  ${\langle S \rangle}_{\mathcal J}$ the minimal $\mathcal{J}$-collection
  containing $S$. We say that $T$ is a {\it generator set} of a
  ${\mathcal J}$-collection $S$ if ${\langle T \rangle}_{\mathcal J} = S$.
\end{defi} 
\begin{rem}
  \label{rem:gen_c_s}
  Let $S$ be a subset of ${\mathbb Z}_{\geq 1}$. Then we have
  \begin{itemize}
  \item
    ${\langle S \rangle}_{\mathbb{C}[[x]]} = \{ k + rn : k \in S \ \mathrm{and} \ r
    \in {\mathbb Z}_{\geq 0} \}$;
  \item
    ${\langle S \rangle}_{\mathcal S} = \{ k + r : k \in S \ \mathrm{and} \ r \in
    {\mathcal S}_{\Gamma} \cup \{0\} \}$.
\end{itemize}
\end{rem}  
\begin{pro}
  \label{pro:c_coll_gen}
  Let $S$ be a subset of ${\mathbb Z}_{\geq 1}$. Then
  ${\langle S \rangle}_{\mathcal C}$ consists of the expressions of the form
  \begin{equation}
    \label{equ:description_c}
    k + \sum_{j=1}^{r} (\overline{\beta}_{b_j} - \overline{\beta}_{a_j}) + s
  \end{equation}
  where $r \geq 0$, $s \in {\mathcal S}_{\Gamma} \cup \{0\}$,
  $k \leq \overline{\beta}_{a_1} < \overline{\beta}_{b_1} \leq
  \overline{\beta}_{a_2} < \overline{\beta}_{b_2} \leq \ldots <
  \overline{\beta}_{b_r} $ and $k \in S$.
\end{pro}
\begin{proof}
  We denote by $T_{m}$ the set of expressions of the form \eqref{equ:description_c}
  with $r=m$.  It is clear that
  $T_{0} \cup T_{1} \subset {\langle S \rangle}_{\mathcal C}$ by the definition of
  a ${\mathcal C}$-collection.
  
  Let us analyze an expression $k'$ of the form \eqref{equ:description_c} in
  $T_{r}$ for $r \geq 1$. We have
  \begin{equation*}
    \overline{\beta}_{b_{r}-1}  < \overline{\beta}_{b_r} - \overline{\beta}_{a_r}
    <  k' - s =
    (k - \overline{\beta}_{a_1}) +
    \sum_{j=0}^{r-1} (\overline{\beta}_{b_j} - \overline{\beta}_{a_{j+1}}) +
    \overline{\beta}_{b_{r}} \leq \overline{\beta}_{b_{r}}
  \end{equation*}
  where we use $\overline{\beta}_{a_r} \leq \overline{\beta}_{b_{r-1}}$ and
  $\overline{\beta}_{b_r} > 2 \overline{\beta}_{b_{r-1}}$ (cf. Remark
  \ref{rem:prop_beta}). 
  Thus, $k' -s + (\overline{\beta}_{b_{r+1}} - \overline{\beta}_{a_{r+1}})
  \in {\langle T_r \rangle}_{\mathcal C}$ if
  $\overline{\beta}_{b_{r}} \leq
  \overline{\beta}_{a_{r+1}} < \overline{\beta}_{b_{r+1}}$.
  It follows that $T_{r+1} \subset {\langle T_r \rangle}_{\mathcal C}$ for $r \geq 0$.
  This implies 
  $\cup_{r \geq 0} T_r \subset {\langle S \rangle}_{\mathcal C}$.

  We claim that $\cup_{r \geq 0} T_r$ is a ${\mathcal C}$-collection.
  It suffices to show that if $k'$ is of the form \eqref{equ:description_c}
  and satisfies
  $k' \leq \overline{\beta}_{a_{r+1}} < \overline{\beta}_{b_{r+1}}$ then
  $k' + (\overline{\beta}_{b_{r+1}} - \overline{\beta}_{a_{r+1}})$
  belongs to $T_{r+1}$. 
  Since $\overline{\beta}_{b_{r}-1} < k' -s \leq \overline{\beta}_{a_{r+1}}$, 
  we get $\overline{\beta}_{b_{r}} \leq \overline{\beta}_{a_{r+1}}$
  and the claim holds by definition of $T_{r+1}$.

  Since $\cup_{r \geq 0} T_r$ is a ${\mathcal C}$-collection contained in
  ${\langle S \rangle}_{\mathcal C}$ and $S \subset \cup_{r \geq 0} T_r$, we
  get ${\langle S \rangle}_{\mathcal C} = \cup_{r \geq 0} T_r$.
\end{proof}  
\begin{pro}
  \label{pro:cw_coll_gen}
  Let $S$ be a subset of ${\mathbb Z}_{\geq 1}$.
  Then ${\langle S \rangle}_{\mathcal{C}^{w}}$ consists of
  $S + ({\mathcal S}_{\Gamma} \cup \{0\})$ and all the expressions of the form
  \begin{equation}
    \label{equ:description_cw}
    k + \overline{\beta}_{b} - \overline{\beta}_{a} + s
  \end{equation}
  where $k \in S$, $s \in {\mathcal S}_{\Gamma} \cup \{0\}$,
  $k \leq \overline{\beta}_{a} < \overline{\beta}_{b}$ and
  $k \in [\beta_a]_{e_{a-1}}$.
\end{pro}
\begin{proof}
  Denote by $T$ the set consisting of $S + ({\mathcal S}_{\Gamma} \cup \{0\})$ and
  the expressions of the form \eqref{equ:description_cw}.  It is clear that $T$
  contains $S$.  Moreover, it is contained in
  ${\langle S \rangle}_{\mathcal{C}^{w}}$ by Remark \ref{rem:cw_gen}.

  It suffices to show that $T$ is a ${\mathcal C}^{w}$-collection. Consider
  \begin{equation*}
    k' = k  + s  \leq \overline{\beta}_{a}
  \end{equation*}
  with $k \in S$, $s \in {\mathcal S}_{\Gamma} \cup \{0\}$ and
  $k' \in [\beta_a]_{e_{a-1}}$.  Let us show that
  $k' + \overline{\beta}_{b} - \overline{\beta}_a$ belongs to $T$ if $a < b$.  We
  have $s \in (e_{a-1})$ since $s < \overline{\beta}_{a}$ (cf. Remark
  \ref{rem:prop_beta}).  This implies $k \in [\beta_a]_{e_{a-1}}$ and hence
  $k' + \overline{\beta}_{b} - \overline{\beta}_{a} \in T$.  Finally, consider
  \begin{equation*}
    k' =k + \overline{\beta}_{b} - \overline{\beta}_{a} + s
    \leq \overline{\beta}_{c}
  \end{equation*}
  with $k \in S$, $s \in {\mathcal S}_{\Gamma} \cup \{0\}$,
  $k \in [\beta_a]_{e_{a-1}}$ and $k' \in [\beta_c]_{e_{c-1}}$.  Let us show
  $k' + \overline{\beta}_{d} - \overline{\beta}_c \in T$ if $d>c$.  Since
  $\overline{\beta}_{b} - \overline{\beta}_{a} > \overline{\beta}_{b-1}$ it follows
  that $c \geq b > a$.  Since $s < \overline{\beta}_c$, we get $s \in
  (e_{c-1})$. Moreover, we obtain
  $\overline{\beta}_{a}, k \in (e_a) \subset (e_{c-1})$ as a consequence of
  $[\beta_a]_{e_{a-1}} \subset (e_{a})$.  The property $k' \not \in (e_{c-1})$,
  together with $\{ \overline{\beta}_a, s, k \} \subset (e_{c-1})$, implies
  $\overline{\beta}_{b} \not \in (e_{c-1})$.  We deduce $c=b$. Therefore, we have
  that
  \begin{equation*}
    k' + \overline{\beta}_{d} - \overline{\beta}_{c}  =
    k + \overline{\beta}_{d} - \overline{\beta}_{a} + s
  \end{equation*}
  belongs to $T$.
\end{proof}  
\begin{pro}
  \label{pro:min_gen}
  Let $S$ be a ${\mathcal J}$-collection where
  $\mathcal{J} \in \{ {\mathbb C}[[x]], \mathcal{S}, \mathcal{C}^{w}, \mathcal{C}
  \}$.  The set
  \begin{equation*}
    {\mathcal B} = \{ k \in {\mathbb Z}_{\geq 1} : k \not \in
    {\langle S \cap {\mathbb Z}_{<k} \rangle}_{\mathcal J} \}
  \end{equation*}
 satisfies 
 \begin{itemize}
  \item ${\mathcal B}$ is contained in any generator set of $S$ and
 \item ${\mathcal B}$ is a generator set of $S$. 
 \end{itemize}
 In particular ${\mathcal B}$ is the minimal generator set of $S$.
\end{pro}
\begin{proof}
  Let $T$ a generator set of $S$ and $k \in {\mathcal B}$. Assume, aiming at a
  contradiction that $k \not \in T$. We have
  \begin{equation*}
    k \in {\langle T \rangle}_{\mathcal J}
    \subset {\langle T  \cap {\mathbb Z}_{<k}  \rangle}_{\mathcal J}
    \cup {\mathbb Z}_{>k}
    \subset   {\langle S \cap {\mathbb Z}_{<k} \rangle}_{\mathcal J}
    \cup {\mathbb Z}_{>k},
  \end{equation*}
  where the first inclusion is a consequence of
  ${\langle T \cap {\mathbb Z}_{<k} \rangle}_{\mathcal J} \cup {\mathbb Z}_{>k}$
  being a ${\mathcal J}$-collection containing $T$. Since
  $k \not \in {\langle S \cap {\mathbb Z}_{<k} \rangle}_{\mathcal J} \cup {\mathbb
    Z}_{>k}$, we obtain a contradiction.
 
  In order to prove $S={\langle {\mathcal B} \rangle }_{\mathcal J}$ it suffices to
  show
  \begin{equation*}
    {\langle {\mathcal B} \cap {\mathbb Z}_{\leq k} \rangle
    }_{\mathcal J} = {\langle S \cap {\mathbb Z}_{\leq k} \rangle }_{\mathcal J}
  \end{equation*}
  for any $k \in {\mathbb Z}_{\geq 0}$. Let us show this result by induction on
  $k$. The property is satisfied for $k=0$.  Suppose it holds for
  $k \in {\mathbb Z}_{\geq 0}$.  We have
  \begin{equation}
    \label{equ:sbj}
    {\langle S \cap {\mathbb Z}_{\leq k} \rangle }_{\mathcal J} 
    \subset
    {\langle {\mathcal B} \cap {\mathbb Z}_{\leq k} \rangle }_{\mathcal J}
    \subset 
    {\langle {\mathcal B} \cap {\mathbb Z}_{\leq k+1} \rangle }_{\mathcal J}.   
  \end{equation}
  We claim that $k+1 \in S$ implies
  $k+1 \in {\langle {\mathcal B} \cap {\mathbb Z}_{\leq k+1} \rangle }_{\mathcal
    J}$.  This is obvious if $k+1 \in {\mathcal B}$.  Otherwise, we have
  ${ \langle S \cap {\mathbb Z}_{\leq k+1} \rangle}_{\mathcal J} = { \langle S \cap
    {\mathbb Z}_{\leq k} \rangle}_{\mathcal J}$ and it is a consequence of equation
  \eqref{equ:sbj}. We deduce
  \begin{equation*}
    S \cap {\mathbb Z}_{\leq k+1}
    \subset {\langle {\mathcal B} \cap {\mathbb Z}_{\leq k+1} \rangle }_{\mathcal J}
  \end{equation*}
  and hence
  $ {\langle S \cap {\mathbb Z}_{\leq k+1} \rangle }_{\mathcal J} = {\langle
    {\mathcal B} \cap {\mathbb Z}_{\leq k+1} \rangle }_{\mathcal J}$.
\end{proof}  
\begin{defi}
  Let $S$ be a ${\mathcal J}$-collection where
  $\mathcal{J} \in \{ {\mathbb C}[[x]], \mathcal{S}, \mathcal{C}^{w}, \mathcal{C}
  \}$.  We say that the minimal generator set of $S$ is the ${\mathcal J}$-{\it
    basis} of $S$.
\end{defi}
\begin{defi}
  We say that the ${\mathcal J}$-basis of $\Lambda_{\Gamma}$ is the
  ${\mathcal J}$-{\it basis of values for} $\Gamma$.  We say that
  ${\mathcal B} = (\omega_1, \hdots, \omega_m)$ is a ${\mathcal J}$-{\it basis of
    $1$-forms} for $\Gamma$ if
  $(\nu_{\Gamma}(\omega_1), \ldots, \nu_{\Gamma}(\omega_m))$ is the
  ${\mathcal J}$-basis of values.
\end{defi}
We have already studied the properties of the ${\mathbb C}[[x]]$-basis of $1$-forms
for $\Gamma$ that we constructed in section \ref{sec:construction_cx_bases}. Next,
we will see that this leads us to interesting ${\mathcal J}$-basis of $1$-forms for
$\Gamma$ for ${\mathcal J} \in \{ \mathcal{S}, {\mathcal C}^{w}, {\mathcal C} \}$.
\begin{defi}
  Given a $1$-form $\Omega$ with $\overline{\beta}_{1} < \nu_{\Gamma}(\Omega)$, we
  define
  \begin{equation*}
    \kappa_{\Omega} = \max \{ j \in {\mathbb N}_{g} : \overline{\beta}_{j} <
    \nu_{\Gamma} (\Omega) \} .
  \end{equation*}
\end{defi}
\begin{pro}
  \label{pro:gen_s}
  Let $\Gamma$ be a singular branch of plane curve. Then there exists an
  $\mathcal{S}$-basis ${\mathcal B}$ of $1$-forms for $\Gamma$ such that any
  $\Omega \in {\mathcal B} \setminus \{dx, dy\}$ \touches{} $\Gamma$ to some order
  and has a unique formal companion curve $\Gamma_{\Omega}$ equisingular to
  $\Gamma_{< \beta_{\kappa +1}}$, where $\kappa = \kappa_{\Omega}$.  Moreover,
  $\Omega$ is a $1$-form of one of the first two types in Proposition
  \ref{pro:type_forms} with $l = \kappa-1$.
 \end{pro}
 \begin{proof}
   Consider the construction of a terminal family
   \begin{equation*}
     {\mathcal T}_{\Gamma,g} = {\mathcal G}_{\Gamma, g}^{s_g} =
     (\Omega_{1}^{s_g}, \ldots, \Omega_{2n}^{s_g})
   \end{equation*}
   of $\Gamma$ in section \ref{sec:construction_cx_bases}.  Let
   ${\mathcal B}' = (\Omega_{j}^{s_g} : j \in \low_{\Gamma, g}^{s_g})$, it is a
   ${\mathbb C}[[x]]$-basis of $\Gamma$ by Proposition \ref{pro:from-l-to-next}.
   Since an ${\mathcal S}$-collection is a ${\mathbb C}[[x]]$-collection,
   ${\langle {\mathcal B}' \rangle}_{\mathcal S} = \Lambda_{\Gamma}$ and hence
   ${\mathcal B}'$ contains an ${\mathcal S}$-basis of $1$-forms ${\mathcal B}$ of
   $\Gamma$.
 
   Let $0 \leq l < g$.  Note that the $1$-forms
   $\Omega \in ({\mathcal B} \cap {\mathcal T}_{\Gamma, l+1}) \setminus {\mathcal T}_{\Gamma, l}$ satisfy
   $\overline{\beta}_{l+1} < \nu_{\Gamma}(\Omega) \leq \overline{\beta}_{l+2}$ and
   thus $\kappa_{\Omega} = l+1$. Thus, $\Omega$ is of one of the types in
   Proposition \ref{pro:type_forms}.  It suffices to show that $\Omega$ can not be
   of type \eqref{type:induced_s}.  Otherwise, since
   $\nu_{\Gamma} (f_{\Gamma, l+1}^{a} \omega) \in {\mathcal S}_{\Gamma} +
   \Lambda_{\Gamma}$ if $a \geq 1$ (which gives also $\nu_{\Gamma}(f_{\Gamma,l+1}^a)>0$), 
   we deduce that $\nu_{\Gamma} (\Omega)$ does not
   belong to the ${\mathcal S}$-basis of values for $\Gamma$, contradicting
   $\Omega \in {\mathcal B}$.
 \end{proof}
 \begin{rem}
   In the case $g=1$, there exists an ${\mathcal S}$-basis of $1$-forms consisting
   of $dx, dy$ and $1$-forms fanning $\Gamma$, since there are no $1$-forms of type
   \eqref{type:induced_c} in Proposition \ref{pro:type_forms}.  This result was
   already proved in \cite{Cano-Corral-Senovilla:semiroots}.
 \end{rem}
 The next result implies Theorem \ref{teo:lambda_on_s_b}.
 \begin{pro}\label{pro:6-5-theorem}
   Let $\Gamma$ be a singular branch of plane curve. Then there exists a
   $\mathcal{C}^{w}$-basis (resp. $\mathcal{C}$-basis) ${\mathcal B}$ of $1$-forms
   for $\Gamma$ such that any $\Omega \in {\mathcal B} \setminus \{dx, dy\}$
   \touches{} $\Gamma$ to some order, fans $\Gamma$ to genus $\kappa_{\Omega}$, and
   has a unique formal companion curve $\Gamma_{\Omega}$ that is equisingular to
   $\Gamma_{< \beta_{\kappa +1}}$ where $\kappa = \kappa_{\Omega}$.
 \end{pro}
 \begin{proof}
   Let ${\mathcal B}'$ be the ${\mathcal S}_{\Gamma}$-basis of $1$-forms provided
   by Proposition \ref{pro:gen_s}, it consists of $1$-forms of types
   \eqref{type:dicritical} and \eqref{type:induced_c} (cf. Proposition
   \ref{pro:type_forms}).  Since a ${\mathcal C}^{w}$-collection is an
   ${\mathcal S}$-collection, ${\mathcal B}'$ contains a ${\mathcal C}^{w}$-basis
   ${\mathcal B}$ of $1$-forms for $\Gamma$.  It suffices to show that all forms in
   ${\mathcal B} \setminus \{dx, dy\}$ are of type \eqref{type:dicritical}.
   Notice that since
   \begin{equation*}
     \nu_{\Gamma} (\Omega) \in {\langle \Lambda_{\Gamma} \cap
       {\mathbb Z}_{<\nu_{\Gamma} (\Omega) } \rangle}_{{\mathcal C}^{w}}
   \end{equation*}
   if $\Omega$ is of type \eqref{type:induced_c}, we get
   $\Omega \not \in {\mathcal B}$ by Proposition \ref{pro:min_gen}.  Thus, all
   $1$-forms in ${\mathcal B} \setminus \{dx, dy\}$ are of type
   \eqref{type:dicritical}.  Moreover, since any ${\mathcal C}$-collection is a
   ${\mathcal C}^{w}$-collection, ${\mathcal B}$ contains a ${\mathcal C}$-basis
   $\overline{\mathcal B}$ of $1$-forms such that
   $\overline{\mathcal B} \setminus \{dx, dy\}$ consists of $1$-forms of type
   \eqref{type:dicritical}.
\end{proof} 
Let us recover the main theorem in \cite{deAbreu-Hernandes:value}.
\begin{cor}
   $\Lambda_{\Gamma}$ determines ${\mathcal S}_{\Gamma}$.
\end{cor}
\begin{proof}
Consider a terminal family $\mathcal{T}_{\Gamma, g}$ constructed as in section 
 \ref{sec:construction_cx_bases}. Fix $0 \leq l <g$. 
Note that $\nu_{\Gamma} (\Omega) > \overline{\beta}_{l+1}$
for any $\Omega \in \mathcal{T}_{\Gamma, g} \setminus \mathcal{T}_{\Gamma, l}$  
and that $\overline{\beta}_{l+1}$ is the maximum element of 
$\nu_{\Gamma} (\mathcal{T}_{\Gamma, l})$ (Remark \ref{rem:generators_construction}). 
  It follows 
  that $\overline{\beta}_{l+1}$ is the element in position $2 \nu_{l}$ 
  when we order the elements
  of a $\mathbb{C}[[x]]$-basis of values for $\Gamma$ in increasing order. 
  Since $(n, \overline{\beta}_1, \ldots, \overline{\beta}_g)$ is a minimal generator 
  set of $\mathcal{S}_{\Gamma}$, clearly $\Lambda_{\Gamma}$ determines 
  $\mathcal{S}_{\Gamma}$.
\end{proof}
Next, we compare the minimal generator set of the semigroup ${\mathcal S}_{\Gamma}$   
with the ${\mathcal S}_{\Gamma}$-basis of values for $\Gamma$. 
The next result implies Theorem \ref{teo:gen_sg_in_sbasis}.
\begin{pro}
  \label{pro:gen_sg_in_sbasis}
  Let $\Gamma$ be a singular branch of plane curve. Then $n$,
  $\overline{\beta}_1, \ldots, \overline{\beta}_g$ belong to the
  ${\mathcal S}$-basis of values for $\Gamma$.
 \end{pro}
 \begin{proof}
 Consider the construction of a ${\mathbb C}[[x]]$-basis ${\mathcal B}$ of $\Gamma$
 in   section \ref{sec:construction_cx_bases}.
 Note that  $n, \overline{\beta}_1, \hdots, \overline{\beta}_g$ belong to 
 $\nu_{\Gamma} ({\mathcal B})$
 by Remark \ref{rem:generators_construction}. 
 %
 %
 Denote by ${\mathfrak B}_{\mathcal S}$ the ${\mathcal S}$-basis of
 $\Lambda_{\Gamma}$.  We have that $d \in \nu_{\Gamma} ({\mathcal B})$ belongs to
 ${\mathfrak B}_{\mathcal S}$ if and only if there is no
 $b \in \nu_{\Gamma} ({\mathcal B})$ such that $d - b \in {\mathcal S}_{\Gamma}$ by
 Remark \ref{rem:gen_c_s}.
 
 Since $\overline{\beta}_1 - n \not \in {\mathcal S}_{\Gamma}$, it follows that
 $\overline{\beta}_1 \in {\mathfrak B}_{\mathcal S}$.  Now, assume that
 $\overline{\beta}_1, \ldots, \overline{\beta}_l$ belong to the
 ${\mathfrak B}_{\mathcal S}$ for some $1 \leq l < g$. It suffices to show that so
 does $\overline{\beta}_{l+1}$.
 Assume, aiming at a contradiction, that there exists $\Omega \in {\mathcal B}$
 such that
 $\overline{\beta}_{l+1} - \nu_{\Gamma} (\Omega) \in {\mathcal S}_{\Gamma}$.
 Since
 $\nu_{\Gamma} ({\mathcal B} \setminus {\mathcal T}_{\Gamma, l}) \subset {\mathbb
   Z}_{>\overline{\beta}_{l+1}}$, it follows that
 $\Omega \in {\mathcal T}_{\Gamma, l}$.  This implies
 $\nu_{\Gamma} (\Omega) \in (e_{l}) \cup [\beta_{l+1}]_{e_l}$ by Corollary
 \ref{cor:cong_term}.  Indeed, $\nu_{\Gamma} (\Omega) \in [\beta_{l+1}]_{e_l}$
 holds.  Otherwise $\overline{\beta}_{l+1} - \nu_{\Gamma} (\Omega)$ belongs to
 ${\mathcal S}_{\Gamma} \setminus (e_l)$ and it has to be greater than or equal to
 $\overline{\beta}_{l+1}$ by Remark \ref{rem:prop_beta}, contradicting
 $\nu_{\Gamma} (\Omega) >0$.  In particular, we deduce
 $\Beg_{\Gamma}(\Omega) = \beta_{l+1}$ by Condition \ref{good7} of Definition
 \ref{def:good}. As a consequence, given
 \begin{equation*}
   {\mathcal T}_{\Gamma, l-1} = (\Omega_1, \ldots, \Omega_{2\nu_{l-1}}),
 \end{equation*}
 we have $\Omega = \Omega_{k}^{s_l}$ for some $k \in \upper_{\Gamma, l}^{s_l}$,
 where ${\mathfrak n}^{s_l} (\beta_{l}) = \beta_{l+1}$, by the construction of
 ${\mathcal T}_{\Gamma, l}$.  Denote $k' = \zeta_{s_l}^{-1} (k)$ (cf. Definition
 \ref{def:zeta}). Therefore, we obtain
 \begin{equation*}
   \nu_{\Gamma} (\Omega_{k}^{s_{l}}) - \nu_{\Gamma} (\Omega_{k'}^{0}) -
   ( \beta_{l+1}  - \beta_{l})  \in n {\mathbb Z}_{\geq 0}
 \end{equation*}
 by Remark \ref{rem:tracking_leading}.  Since $k' \in \upper_{\Gamma, l}^{1}$, we
 have $\Omega_{k'}^{0} = f_{\Gamma, l}^{m} \Omega_j$, where
 $m \in {\mathbb N}_{n_{l} -1}$, $j \in {\mathbb N}_{2 \nu_{l-1}}$ and either
 $m = n_{l}-1$ or $m<n_{l}-1$ and $\nu_{\Gamma} (\Omega_j) \in (e_{l-1})$ by
 Definition \ref{def:lu} and Lemma \ref{lem:cong_nu}. We obtain
 \begin{equation*}
   \overline{\beta}_{l+1} - \nu_{\Gamma} (\Omega) = 
   \overline{\beta}_{l+1} - m \overline{\beta}_l -
   (\beta_{l+1} - \beta_l) -  \nu_{\Gamma} (\Omega_{j}) - rn
 \end{equation*}
 for some $r \in {\mathbb Z}_{\geq 0}$ and hence
 \begin{equation}
   \label{equ:b_m_nu}
   \overline{\beta}_{l+1} - \nu_{\Gamma} (\Omega) = 
   (n_{l} -  m) \overline{\beta}_l  -
   \nu_{\Gamma} (\Omega_{j}) - rn < \min ((n_{l} -  m) \overline{\beta}_l  ,
   \overline{\beta}_{l+1})  
 \end{equation}
 by equation \eqref{equ:rec_beta}. First, suppose $m = n_{l}-1$. We obtain
 $\overline{\beta}_l - \nu_{\Gamma} (\Omega_{j}) \in {\mathcal S}_{\Gamma}$,
 contradicting $\overline{\beta}_l \in {\mathfrak B}_{\mathcal S}$.

 Finally, assume $m < n_{l}-1$. 
 As a consequence of
 $\overline{\beta}_{l+1} - \nu_{\Gamma} (\Omega) \in {\mathcal S}_{\Gamma}$,
 equation \eqref{equ:b_m_nu} implies
 \begin{equation*}
   \overline{\beta}_{l+1} - \nu_{\Gamma} (\Omega)  =
   m_{0} n + m_1 \overline{\beta}_{1} + \ldots + m_l \overline{\beta}_{l}
 \end{equation*}
 with $m_{0}, \ldots, m_{l} \in {\mathbb Z}_{\geq 0}$ and $m_l < n_{l} - m$. By
 using $\nu_{\Gamma} (\Omega_j) \in (e_{l-1})$ and again equation
 \eqref{equ:b_m_nu}, we obtain that
 \begin{equation*}
   (n_{l} - m - m_{l}) \overline{\beta}_{l} \in (e_{l-1}),
 \end{equation*} contradicting Remark \ref{rem:prop_beta}.
\end{proof} 
\begin{rem}
  The ${\mathcal C}^{w}$-collection and ${\mathcal C}$-collection generated just by
  $n$ and $\beta_1$ contain $\overline{\beta}_2, \ldots, \overline{\beta}_g$, and
  hence ${\mathcal S}_{\Gamma}$, by Propositions \ref{pro:cw_coll_gen} and
  \ref{pro:c_coll_gen}.  In particular $\overline{\beta}_{j}$ does not belong
  neither to the ${\mathcal C}^{w}$-basis nor to the ${\mathcal C}$-basis of values
  for $\Gamma$ if $1 < j \leq g$.
\end{rem}
\begin{rem}
  \label{rem:power2}
  Assume that $\Gamma$ is a curve of genus $g \in {\mathbb Z}_{\geq 1}$ and
  multiplicity $n=2^{g}$.  Hence, we have $n_{j} =2$ for any
  $j \in {\mathbb N}_{g}$. Consider the construction of a ${\mathbb C}[[x]]$-basis
  ${\mathcal B}$ for $\Gamma$ in section \ref{sec:construction_cx_bases}.  Since
  ${\mathbb N}_{2 \nu_{l+1}} = {\mathbb N}_{2 \nu_{l}} \cup \upper_{l+1}$,
  ${\mathcal T}_{\Gamma, l+1}$ is a good family of level $l+1$, and Condition
  \ref{good1} of Definition \ref{def:good} and Remark \ref{rem:chu} hold, it
  follows that
  \begin{equation*}
    \nu_{\Gamma} ({\mathcal T}_{\Gamma, l+1}) =
    \nu_{\Gamma} ({\mathcal T}_{\Gamma, l}) \cup \{ r + \overline{\beta}_{l+2} -
    \overline{\beta}_{l+1} : r \in \nu_{\Gamma} ({\mathcal T}_{\Gamma, l}) \}
  \end{equation*}
  for any $0 \leq l < g-1$. Notice that $\low_{\Gamma, g}^{s} = {\mathbb N}_{n}$
  for any $s \geq 1$ and indeed ${\mathcal B} = {\mathcal T}_{\Gamma, g-1}$.
  Therefore $\Lambda_{\overline{\Gamma}} = \Lambda_{\Gamma}$ for any
  $\overline{\Gamma}$ in the equisingularity class of $\Gamma$ and
  $\Lambda_{\Gamma}$ is the ${\mathcal C}$-collection generated by $n$ and
  $\beta_{1}$.

  Note that since $n_{j}=2$ for any $j \in {\mathbb N}_{g}$, ${\mathcal B}$ does
  not contain $1$-forms of type \eqref{type:induced_s} (cf. Proposition
  \ref{pro:type_forms}).  Such $1$-forms disappear when we consider an
  ${\mathcal S}$-basis of $1$-forms for $\Gamma$ contained in ${\mathcal B}$.  In
  the present example, we will see that no $1$-form disappears. Indeed, we claim
  that ${\mathcal B} = {\mathcal T}_{\Gamma, g-1}$ is an ${\mathcal S}$-basis of
  $1$-forms for $\Gamma$.

  Let us prove the claim by induction. The induction hypothesis is that the set
  $\nu_{\Gamma} ({\mathcal T}_{\Gamma, l})$ is the ${\mathcal S}$-basis of
  ${\langle \nu_{\Gamma} ({\mathcal T}_{\Gamma, l}) \rangle}_{\mathcal S}$ for any
  $0 \leq l <g$.  The claim is just this result for $l=g-1$.  The result is obvious
  for $l=0$, so it suffices to prove that if it holds for $0 \leq l < g-1$ so
  does it for $l+1$.  Consider values
  $\lambda, \mu \in \nu_{\Gamma} ({\mathcal T}_{\Gamma, l+1})$ with
  $\lambda < \mu$.  It suffices to show that
  $\mu - \lambda \not \in {\mathcal S}_{\Gamma}$.
  Note that
  \begin{equation*}
    \nu_{\Gamma} ({\mathcal T}_{\Gamma, l}) \subset (e_{l+1}) \ \ \mathrm{and}
    \ \ 
    \nu_{\Gamma} ({\mathcal T}_{\Gamma, l+1} \setminus {\mathcal T}_{\Gamma, l})
    \cap (e_{l+1}) = \emptyset.
  \end{equation*}
  Assume, first, that $\lambda, \mu \in (e_{l+1})$. It follows that
  $\lambda, \mu \in \nu_{\Gamma} ({\mathcal T}_{\Gamma, l})$ and hence
  $\mu - \lambda \not \in {\mathcal S}_{\Gamma}$ by the induction hypothesis.  Now,
  suppose that exactly one of elements of the pair $\lambda, \mu$ belongs to
  $(e_{l+1})$. In such a case we also obtain
  $\mu - \lambda \not \in {\mathcal S}_{\Gamma}$.  Otherwise
  $\mu - \lambda \not \in (e_{l+1})$ would imply
  $\mu - \lambda \geq \overline{\beta}_{l+2}$ by Remark \ref{rem:prop_beta} and
  since $\lambda, \mu \in{\mathbb N}_{\overline{\beta}_{l+2}}$ by the properties of
  terminal families, we get a contradiction.
  Finally, assume that neither $\lambda$ nor $\mu$ belongs to $(e_{l+1})$.  In this
  case, we have
  $\lambda = \lambda' + \overline{\beta}_{l+2} - \overline{\beta}_{l+1}$ and
  $\mu = \mu' + \overline{\beta}_{l+2} - \overline{\beta}_{l+1}$ for
  $\lambda', \mu' \in \nu_{\Gamma} ({\mathcal T}_{\Gamma, l})$.
  Again, the induction hypothesis implies that
  $\mu - \lambda = \mu' - \lambda' \not \in {\mathcal S}_{\Gamma}$, concluding the
  proof.
\end{rem}
\begin{rem}
  Let $\Gamma$ be a curve of genus $g=2$ and multiplicity $n=4$.  Since
  $\overline{\beta}_2 = \beta_1 + \beta_2$, the ${\mathbb C}[[x]]$-basis of values
  for $\Gamma$ is equal to
  \begin{equation*}
    {\mathcal V}:=(n, \beta_1, n+ \overline{\beta}_2 -
    \overline{\beta}_1, \beta_1 + \overline{\beta}_2 - \overline{\beta}_1) = (n,
    \beta_1, n+\beta_2, \overline{\beta}_2)
  \end{equation*}
  by Remark \ref{rem:power2}.  Moreover, the same remark also provides that
  ${\mathcal V}$ is the ${\mathcal S}$-basis of $\Lambda_{\Gamma}$.  Since the
  ${\mathcal S}$-basis of ${\mathcal S}_{\Gamma}$ is
  $(n, \beta_{1}, \overline{\beta}_{2})$ by Remark \ref{rem:prop_beta}, it follows
  that $n + \beta_{2} \not \in {\mathcal S}_{\Gamma}$ and
  $\Lambda_{\Gamma} \neq {\mathcal S}_{\Gamma}$.  Since the
  ${\mathcal C}^{w}$-collection generated by $(n, \beta_1)$ is
  ${\mathcal S}_{\Gamma}$, it follows that $(n, \beta_1, n+\beta_{2})$ is the
  ${\mathcal C}^{w}$-basis of values for $\Gamma$.  Finally $(n, \beta_1)$ is the
  ${\mathcal C}$-basis of values for $\Gamma$ by Remark \ref{rem:power2}.
\end{rem}
\section{Singular directions}
\label{sec:singular_directions}
Let $\Gamma$ be a singular branch of plane curve.  Our method to find
a ${\mathbb C}[[x]]$-basis for $\Gamma$ evaluates the contributions of
the coefficients $a_{\beta}$, with $\beta \in {\mathcal E}_{\Gamma}$,
one by one (cf. section \ref{sec:construction_cx_bases}).  So, it
makes sense to identify the singular values of the coefficients
$a_{\beta}$, i.e. the values at which the construction is
qualitatively different than for a generic value of $a_{\beta}$.  This
is the goal of this section.  To find the singular values of the
coefficients is equivalent to finding directions or more precisely,
infinitely near points (cf. Definition \ref{def:seq_inf_pt} and Remark
\ref{rem:coef_dir}) at which the construction has a discontinuous
behavior.  First, we will define singular directions. Then, we will
see that this concept is an analytic invariant.
\begin{defi}
  \label{def:sing_dir0}
   We set
  ${\mathcal D}_{\Gamma} (\beta_1) = \ldots = {\mathcal D}_{\Gamma}
  (\beta_g) = \emptyset$ for the Puiseux exponents.
\end{defi} 
Definition 
\ref{def:sing_dir0} reflects the fact that the generic value of $\Lambda_{\overline{\Gamma}}$ for
\[ \{ \overline{\Gamma} \in  \fga{\Gamma} : 
(a_{\beta_1, \overline{\Gamma}}, \hdots, a_{\beta_g, \overline{\Gamma}}) = (c_1, \hdots, c_g) \} \]   
is independent of $(c_1, \hdots, c_g) \in ({\mathbb C}^{*})^{g}$ and indeed is   the generic value 
of $\Lambda_{\overline{\Gamma}}$ for $\overline{\Gamma} \in  \fga{\Gamma}$
(cf. Remark \ref{rem:puiseux-no-supersede}  and Theorem \ref{teo:ana_inv_sing_dir}).
\begin{defi}
  \label{def:sing_dir}
  Consider $0 \leq l < g$, $s \in {\mathbb Z}_{\geq 1}$ with
  ${\mathfrak n}^{s} (\beta_{l+1}) < \beta_{l+2}$.  We denote by
  ${\mathcal Q}_{\Gamma, l+1}^{s}$ the set consisting of
  $k \in \upper_{\Gamma, l+1}^{s}$ such that there is no
  $j \in \low_{\Gamma, l+1}^{s}$ with
  $\nu_{\Gamma, s} (\Omega_{k}^{s}) - \nu_{\Gamma} (\Omega_{j}^{s})
  \in n {\mathbb Z}_{\geq 0}$.
\end{defi}  
\begin{defi}
  We define \begin{equation*}
    {\mathcal D}_{\Gamma} ({\mathfrak n}^{s}(\beta_{l+1})) = \cup_{k \in {\mathcal Q}_{\Gamma,
        l+1}^{s}} {\mathcal D}_{\Gamma} ({\mathfrak
      n}^{s}(\beta_{l+1}), k),
  \end{equation*}
  cf. Definition \ref{def:sing_dir_j}.  Let $\beta\in {\mathcal E}_{\Gamma}$.  The
  elements of ${\mathcal D}_{\Gamma} (\beta)$ are called the {\it singular
    directions} associated to $\Gamma$ and the exponent $\beta$.   
\end{defi}
\begin{rem}
  Even if ${\mathcal D}_{\Gamma} (\beta)$ is a set of values of the
  coefficient $a_{\beta}$, it makes sense to call its elements
  directions because of the correspondence in Remark
  \ref{rem:coef_dir}
\end{rem}
\begin{rem}
  \label{rem:disc_u}
  An alternative definition of $ {\mathcal Q}_{\Gamma, l+1}^{s}$ is
  provided by
  \[
    \upper_{\Gamma, l+1}^{s} \setminus {\mathcal Q}_{\Gamma, l+1}^{s} =  \{ k \in \upper_{\Gamma, l+1}^{s} :
    k \in \upper_{\overline{\Gamma}, l+1}^{s+1} \ \forall \
    \overline{\Gamma} \in
    \fgba{\Gamma}{{\mathfrak n}^{s}(\beta_{l+1})} \}  .
    \]  
  Indeed, consider $k \in \upper_{\Gamma, l+1}^{s}$. There is a unique
  $j \in \low_{\Gamma, l+1}^{s}$ such that
 \begin{equation*}
   \nu_{\Gamma, s} (\Omega_{k}^{s}) - \nu_{\Gamma} (\Omega_{j}^{s}) = dn
 \end{equation*}
 (cf. Definition \ref{def:gen_ord_s}, page \pageref{def:gen_ord_s})
 for some $d \in \mathbb{Z}$ by Definition \ref{def:good}.  There are two cases,
 depending on whether or not $d \geq 0$.
  
  Suppose $d \geq 0$, i.e.
  $k \not \in {\mathcal Q}_{\Gamma, l+1}^{s}$.  Given
  $\overline{\Gamma} \in \fgba{\Gamma}{{\mathfrak n}^{s}(\beta_{l+1})}$, there exists
  $c_{\overline{\Gamma}} \in {\mathbb C}$ such that
  $\nu_{\overline{\Gamma}} (\Omega) > \nu_{\Gamma, s} (\Omega_{k}^{s})
  $ where
  \begin{equation}
    \label{equ:disc_u}
    \Omega = \Omega_{k}^{s} - c_{\overline{\Gamma}} x^{d} \Omega_{j}^{s}.
  \end{equation}
  We have $k \in \upper_{\overline{\Gamma}, l+1}^{s+1}$. It is clear
  if
  $\Beg_{\overline{\Gamma}}  (\Omega_{k}^{s}) > {\mathfrak
    n}^{s}(\beta_{l+1})$ by definition and we get
  $c_{\overline{\Gamma}} = 0$ (cf. Remark \ref{rem:sd_l}).  It holds
  for
  $\Beg_{\overline{\Gamma}} (\Omega_{k}^{s}) = {\mathfrak
    n}^{s}(\beta_{l+1})$ too since
  $\nu_{\overline{\Gamma}} (\Omega_{k}^{s}) = \nu_{\Gamma, s}
  (\Omega_{k}^{s})$; here we obtain $c_{\overline{\Gamma}} \neq 0$.
 
  Assume $d < 0$, i.e. $k \in {\mathcal Q}_{\Gamma, l+1}^{s}$.  Given
  $\overline{\Gamma} \in \fgba{\Gamma}{{\mathfrak n}^{s}(\beta_{l+1})}$, we have
  \begin{equation*}
    k \in \upper_{\overline{\Gamma}, l+1}^{s+1} \Leftrightarrow
    \Beg_{\overline{\Gamma}} (\Omega_{k}^{s}) >
    {\mathfrak n}^{s}(\beta_{l+1})
    \Leftrightarrow
    \nu_{\overline{\Gamma}} (\Omega_{k}^{s}) >
    \nu_{\Gamma, s} (\Omega_{k}^{s})
  \end{equation*}
  by Remark \ref{rem:sd_l}.
  \end{rem}
  \begin{rem}
    Singular directions in ${\mathcal D}_{\Gamma} ({\mathfrak n}^{s} (\beta_{l+1}))$
    detect changes in the family of $\overline{\Gamma}$-values of K\"{a}hler
    differentials for
    $\overline{\Gamma} \in \fgba{\Gamma}{{\mathfrak n}^{s}(\beta_{l+1})}$, as we explain.  Take
    $k \in \upper_{\Gamma, l+1}^{s}$, and  consider the notations in Remark.
    \ref{rem:disc_u}.

    First, suppose $k \not \in {\mathcal Q}_{\Gamma, l+1}^{s}$.  The contact
    $\nu_{\Gamma, s} (\Omega_{k}^{s})$ can be
    ``superseded" for any
    $\overline{\Gamma} \in \fgba{\Gamma}{{\mathfrak n}^{s}(\beta_{l+1})}$ by replacing $\Omega_{k}^{s}$ with $\Omega$ following
    equation \eqref{equ:disc_u}.  Indeed $\Omega$ belongs to
    ${\mathcal G}_{\overline{\Gamma}, l+1}^{s+1}$.
    So this case does not provide singular directions.
  
    Now, suppose $k \in {\mathcal Q}_{\Gamma, l+1}^{s}$.  Let
    $\overline{\Gamma} \in \fgba{\Gamma}{{\mathfrak n}^{s}(\beta_{l+1})}$.  On the one hand,
    the order
    $\nu_{\overline{\Gamma}} (\Omega_{k}^{s}) = \nu_{\Gamma, s}
    (\Omega_{k}^{s})$ can not be ``superseded" if
    $k \in \low_{\overline{\Gamma}, l+1}^{s+1}$.  On the other hand,
    the $1$-form $\Omega_{k}^{s}$ has a different
    $\overline{\Gamma}$-order if
    $k \in \upper_{\overline{\Gamma}, l+1}^{s+1}$.  Thus, singular
    directions provide a discontinuity in those
    $\overline{\Gamma}$-values of K\"{a}hler differentials that can not be
    ``superseded".
\end{rem}
\begin{rem}\label{rem:puiseux-no-supersede}
  If we apply Definition \ref{def:sing_dir} to the case
  ${\mathfrak n}^{s} (\beta_{l+1}) = \beta_{l+2}$, then,
  $\Omega_{k}^{s} \in \oDD{\beta_{l+2}, <}{0}{ l+2}$ for any
  $k \in \upper_{\Gamma, l+1}^{s}$, and in particular
  $\Omega_{k}^{s} \in \tDd{\beta_{l+2}}{l+2}$ by Proposition
  \ref{pro:unique_companion}.  Thus, Definition \ref{def:sing_dir}
  would provide $\{ 0 \}$ as the set of singular directions associated
  to $\Gamma$ and $\beta_{l+2}$.  Since
  $a_{\beta_{l+2}, \overline{\Gamma}} \neq 0$ for any
  $\overline{\Gamma} \in \fga{\Gamma}$ it is more convenient to
  define ${\mathcal D}_{\Gamma} (\beta_{l+2}) = \emptyset$, as we have
  done in Definition \ref{def:sing_dir0}.
\end{rem}
\begin{rem}
  The set ${\mathcal D}_{\Gamma} ({\mathfrak n}^{s} (\beta_{l+1}))$
  depends just on the coefficients $a_{\beta, \Gamma}$ with
  $\beta_1 \leq \beta < {\mathfrak n}^{s} (\beta_{l+1})$ and
  $\beta \in {\mathcal E}_{\Gamma}$ since
  ${\mathcal G}_{\Gamma, l+1}^{s}$ does so by Proposition
  \ref{pro:term_good_dep}. In other words,
  \begin{equation*}
    {\mathcal D}_{\Gamma} ({\mathfrak n}^{s}
    (\beta_{l+1})) = {\mathcal D}_{\overline{\Gamma}} ({\mathfrak
      n}^{s} (\beta_{l+1}))
  \end{equation*} 
  holds for any
  $\overline{\Gamma} \in \fgba{\Gamma}{{\mathfrak n}^{s} (\beta_{l+1})}$.
\end{rem}
\begin{rem}
  The non-singular directions have a generic behavior. For instance,
  consider a curve $\overline{\Gamma} \in \fga{\Gamma}$ such
  that
  $a_{\beta, \overline{\Gamma}} \not \in {\mathcal
    D}_{\overline{\Gamma}} (\beta)$ for any
  $\beta \in {\mathcal E}_{\Gamma}$.  All such curves share the same
  semi-module $\Lambda = \Lambda_{\overline{\Gamma}}$, that is
  precisely the semi-module of a generic curve in
  $\fga{\Gamma}$.
\end{rem}
\begin{rem}
  Let $\Gamma$ be a curve of multiplicity $n= 2^{g}$ (cf. Remark
  \ref{rem:power2}).  Since ${\mathcal G}_{\Gamma, l+1}^{s}$ is a good
  family, property \ref{good5} of Definition \ref{def:good} implies
  \begin{equation*}
    \low_{\Gamma, l+1}^{s} =
    {\mathbb N}_{2 \nu_{l}} \ \ \mathrm{and} \ \ 
    \upper_{\Gamma, l+1}^{s} = \upper_{l+1}
  \end{equation*}
  for all $0 \leq l < g$ and $s \in {\mathbb Z}_{\geq 1}$ with
  ${\mathfrak n}^{s} (\beta_{l+1}) \leq \beta_{l+2}$.  Thus, the set
  ${\mathcal Q}_{\Gamma, l+1}^{s}$ is empty for all $0 \leq l < g$ and
  $s \in {\mathbb Z}_{\geq 1}$ with
  ${\mathfrak n}^{s} (\beta_{l+1}) < \beta_{l+2}$ by Remark
  \ref{rem:disc_u}. We deduce that
  ${\mathcal D}_{\overline{\Gamma}}(\beta) = \emptyset$ for all
  $\overline{\Gamma} \in \fga{\Gamma}$ and
  $\beta \in {\mathcal E}_{\Gamma}$.  The lack of singular directions
  is natural, since there can be no discontinuities in the values of
  K\"{a}hler differentials in $\fga{\Gamma}$. Indeed, we have
  $\Lambda_{\overline{\Gamma}} = \Lambda_{\Gamma}$ for any
  $\overline{\Gamma} \in \fga{\Gamma}$ by Remark
  \ref{rem:power2}.
\end{rem}
Finally, we see that singular directions are invariant by analytic
conjugacy.
\begin{teo}
  \label{teo:ana_inv_sing_dir}
  Let $\Gamma$ be a branch of plane curve. Let $\phi$ be a germ of
  biholomorphism defined in a neighborhood of $(0,0)$ in
  ${\mathbb C}^{2}$ such that $\overline{\Gamma} := \phi (\Gamma)$
  belongs to $\fga{\Gamma}$.  Given any choice of $0 \leq l < g$
  and $s \in {\mathbb Z}_{\geq 1}$ with
  ${\mathfrak n}^{s} (\beta_{l+1}) < \beta_{l+2}$, the equality
  ${\mathcal Q}_{\Gamma, l+1}^{s}={\mathcal Q}_{\overline{\Gamma},
    l+1}^{s}$ holds. Moreover, given
  $k \in {\mathcal Q}_{\Gamma, l+1}^{s}$, we get
  \begin{equation}
    \label{equ:inv_sing_k}
    a_{{\mathfrak n}^{s}(\beta_{l+1}), {\Gamma}'}
    \in {\mathcal D}_{\Gamma} ({\mathfrak n}^{s}(\beta_{l+1}), k)
    \Leftrightarrow 
    a_{{\mathfrak n}^{s}(\beta_{l+1}), \phi({\Gamma}')} \in
    {\mathcal D}_{\overline{\Gamma}} ({\mathfrak n}^{s}(\beta_{l+1}), k)  
  \end{equation}
  for any
  ${\Gamma}' \in \fgba{\Gamma}{{\mathfrak n}^{s} (\beta_{l+1})}$.  In particular, we obtain
  \begin{equation}
    \label{equ:inv_sing}
    a_{{\mathfrak n}^{s}(\beta_{l+1}), {\Gamma}'} \in
    {\mathcal D}_{\Gamma} ({\mathfrak n}^{s}(\beta_{l+1}))
    \Leftrightarrow 
    a_{{\mathfrak n}^{s}(\beta_{l+1}), \phi({\Gamma}')} \in
    {\mathcal D}_{\overline{\Gamma}} ({\mathfrak n}^{s}(\beta_{l+1}))  
  \end{equation}
  for any
  ${\Gamma}' \in \fgba{\Gamma}{{\mathfrak n}^{s} (\beta_{l+1})}$.
\end{teo}
\begin{proof}
  Consider the terminal families
  \begin{equation*}
    {\mathcal T}_{\Gamma, l} = (\Omega_1, \ldots, \Omega_{2 \nu_{l}}) \ \
    \mathrm{and} \ \ {\mathcal T}_{\overline{\Gamma}, l} =
    (\overline{\Omega}_1, \ldots, \overline{\Omega}_{2 \nu_{l}})
  \end{equation*}
  of level $l$ for $\Gamma$ and $\overline{\Gamma}$
  respectively.

  Assume for now that
    $\nu_{\Gamma} (\Omega_j) = \nu_{\overline{\Gamma}} (\overline{\Omega}_j)$ for any
    $j \in {\mathbb N}_{2 \nu_{l}}$. This assumption will be resolved in the last
  paragraph.

  Since
  $\nu_{\Gamma} (f_{\Gamma, l+1}) = \nu_{\overline{\Gamma}} (f_{\overline{\Gamma},
    l+1}) = \overline{\beta}_{l+1}$, the initial families
  \begin{equation*}
    {\mathcal G}_{\Gamma, l+1}^{0} = (\Omega_1^{0},
    \ldots, \Omega_{2 \nu_{l+1}}^{0}) \ \ \mathrm{and} \ \ {\mathcal
      G}_{\overline{\Gamma}, l+1}^{0} = (\overline{\Omega}_{1}^{0},
    \ldots, \overline{\Omega}_{2 \nu_{l+1}}^{0})
  \end{equation*}
  of level $l+1$ for $\Gamma$ and $\overline{\Gamma}$ satisfy
  $\nu_{\Gamma} (\Omega_{j}^{0}) = \nu_{\overline{\Gamma}}
  (\overline{\Omega}_{j}^{0})$ for any
  $j \in {\mathbb N}_{2 \nu_{l+1}}$ by construction. We claim that the
  following properties hold for $s \geq 1$ with
  ${\mathfrak n}^{s} (\beta_{l+1}) \leq \beta_{l+2}$:
  \begin{enumerate}[(a)]
  \item $\low_{\Gamma, l+1}^{s} = \low_{\overline{\Gamma}, l+1}^{s}$
    and
    $\upper_{\Gamma, l+1}^{s} = \upper_{\overline{\Gamma}, l+1}^{s}$;
  \item
    $\nu_{\Gamma} (\Omega_{j}^{s}) = \nu_{\overline{\Gamma}}
    (\overline{\Omega}_{j}^{s})$ for any
    $j \in \low_{\Gamma, l+1}^{s}$;
  \item
    $\nu_{\Gamma, s} (\Omega_{k}^{s}) = \nu_{\overline{\Gamma},s}
    (\overline{\Omega}_{k}^{s})$ for any
    $k \in \upper_{\Gamma, l+1}^{s}$ (cf. Definition
    \ref{def:gen_ord_s}).
 \end{enumerate}
 In particular, these properties imply
 ${\mathcal Q}_{\Gamma, l+1}^{s}={\mathcal Q}_{\overline{\Gamma},
   l+1}^{s}$ if ${\mathfrak n}^{s}(\beta_{l+1})<\beta_{l+2}$.  We will
 see that they also imply equations \eqref{equ:inv_sing_k} and
 \eqref{equ:inv_sing}.

 First, let us consider the case $s=1$.  Property (a) is a consequence of Definition
 \ref{def:lu}. Property (b) holds because for $s=1$ the set
 $\low_{\Gamma,l+1}^s$ is defined independently of the values of
 $\nu_{\Gamma}(\Omega_j^s)$.  Property (c) follows from
 $\nu_{\Gamma} (\Omega_{k}^{0}) = \nu_{\overline{\Gamma}}
 (\overline{\Omega}_{k}^{0})$ and
 \begin{equation*}
   \nu_{\Gamma, 1} (\Omega_{k}^{1}) - \nu_{\Gamma} (\Omega_{k}^{0}) =  
   \nu_{\overline{\Gamma},1} (\overline{\Omega}_{k}^{1}) -
   \nu_{\overline{\Gamma}} (\overline{\Omega}_{k}^{0}) =
   {\mathfrak n}(\beta_{l+1}) - \beta_{l+1}
 \end{equation*}
 for any $k \in \upper_{\Gamma, l+1}^{s}$.  Suppose the properties
 hold for $s \in {\mathbb Z}_{\geq 1}$ with
 ${\mathfrak n}^{s}(\beta_{l+1}) < \beta_{l+2}$. Let us prove them for
 $s+1$.  Consider $k \in \upper_{\Gamma, l+1}^{s}$ and let
 $j \in \low_{\Gamma, l+1}^{s}$ with
 $\nu_{\Gamma, s} (\Omega_{k}^{s}) - \nu_{\Gamma} (\Omega_{j}^{s}) \in
 (n)$.  Properties (a), (b) and (c), hold if the following ones do:
 \begin{enumerate}
 \item \label{conj1}
   $j \in \low_{\Gamma, l+1}^{s+1} \Leftrightarrow j \in
   \low_{\overline{\Gamma}, l+1}^{s+1}$;
 \item \label{conj2}
   $\nu_{\Gamma} (\Omega_{r}^{s+1}) = \nu_{\overline{\Gamma}}
   (\overline{\Omega}_{r}^{s+1})$ for any
   $r \in \{j,k\} \cap \low_{\Gamma, l+1}^{s+1}$;
 \item \label{conj3}
   $\nu_{\Gamma, s} (\Omega_{r}^{s+1}) = \nu_{\overline{\Gamma},s}
   (\overline{\Omega}_{r}^{s+1})$ for any
   $r \in \{j,k\} \cap \upper_{\Gamma, l+1}^{s+1}$.
 \end{enumerate}
 There are two cases: 
 
 \strut
 
 $\bullet$ {\it Case
   $ \nu_{\Gamma, s} (\Omega_{k}^{s}) \geq \nu_{\Gamma}
   (\Omega_{j}^{s})$}: 
   
   We get
 $j \in \low_{\Gamma, l+1}^{s+1} \cap \low_{\overline{\Gamma},
   l+1}^{s+1} $,
 $k \in \upper_{\Gamma, l+1}^{s+1} \cap \upper_{\overline{\Gamma},
   l+1}^{s+1}$ and
 \begin{multline*}
   \nu_{\Gamma, s+1} (\Omega_{k}^{s+1}) -
   \nu_{\Gamma,s} (\Omega_{k}^{s+1}) =  \\
   \nu_{\overline{\Gamma},s+1} (\overline{\Omega}_{k}^{s+1}) -
   \nu_{\overline{\Gamma},s} (\overline{\Omega}_{k}^{s}) =
   {\mathfrak n}^{s+1} (\beta_{l+1}) -
   {\mathfrak n}^{s} (\beta_{l+1}).
 \end{multline*}
 Thus, properties \eqref{conj1}, \eqref{conj2} and \eqref{conj3}
 hold. As $k\not\in {\mathcal Q}_{\Gamma, l+1}^{s}$, this case is done.
 
 \strut

 $\bullet$ {\it Case
   $ \nu_{\Gamma, s} (\Omega_{k}^{s}) < \nu_{\Gamma}
   (\Omega_{j}^{s})$}:  
   
   Let $ \fg{\Gamma, {\mathfrak g}}^{s}$ be the set 
 consisting of
 $\Gamma ' \in \fgba{\Gamma}{{\mathfrak n}^s(\beta_{l+1})}$ such
 that
 $\nu_{\Gamma'} (\Omega_{b}^{s}) =\nu_{\Gamma, s} (\Omega_{b}^{s})$
 for any $b \in \upper_{\Gamma, l+1}^{s}$.  We are going to use the fact that
 \begin{equation}
   \label{equ:m+f}
   {\mathcal M}_{2 \nu_{l+1}} + (f_{\Gamma, l+2}) dx +  (f_{\Gamma,l+2}) dy
   = \hat{\Omega}^{1} \cn{2}  
 \end{equation}
 to prove Properties \eqref{conj1}, \eqref{conj2} and \eqref{conj3}.  In order to
 profit from this equality of modules, we need to estimate
 $\nu_{\Gamma, s} (\Omega_{k}^{s})$ and $\nu_{\Gamma, s} (f_{\Gamma, l+2})$. We have
 \begin{equation*}
   \nu_{\Gamma, s} (\Omega_{k}^{s}) \leq \nu_{\Gamma} (\Omega_{k}^{0}) +
   {\mathfrak n}^{s}(\beta_{l+1}) - \beta_{l+1}
   \leq
   n_{l+1} \overline{\beta}_{l+1}  + {\mathfrak n}^{s}(\beta_{l+1}) -
   \beta_{l+1}
 \end{equation*}
by Condition \ref{good8} of Definition \ref{def:good} applied to some  
 $\Gamma' \in \fg{\Gamma, {\mathfrak g}}^{s}$.  We have
 \begin{equation*}
   f_{\Gamma, l+2} = \prod_{k=0}^{\nu_{l+1}-1} 
   \left(
     y - \sum_{\beta < \beta_{l+2}}
     a_{\beta, \Gamma} e^{\frac{2 \pi k \beta i}{n}} x^{\frac{\beta}{n}} 
   \right) .
 \end{equation*}
 All terms in the product, except the one for $k=0$, have the same
 $\Gamma'$-order for any
 $\Gamma' \in \fgba{\Gamma}{{\mathfrak n}^s(\beta_{l+1})} \cup \{ \Gamma_{< \beta_{l+1}}\}$. 
 This property, together with
 \begin{equation*}
   \nu_{0} (f_{\Gamma, l+2} \circ \Gamma_{< \beta_{l+1}} (t)) =
   n_{l+1} \nu_{0} (f_{\Gamma, l+1} \circ \Gamma_{< \beta_{l+1}} (t))=
   n_{l+1} \overline{\beta}_{l+1}
 \end{equation*}
 imply that the the minimal order $\nu_{\Gamma, s} (f_{\Gamma, l+2})$
 of $\nu_{\Gamma'} (f_{\Gamma, l+2})$ for the family
 $\Gamma' \in \fgba{\Gamma}{{\mathfrak n}^s(\beta_{l+1})}$
 satisfies
 \begin{equation}
   \label{equ:aux_c1}
   \nu_{\Gamma, s} (f_{\Gamma, l+2}) = n_{l+1} \overline{\beta}_{l+1}  +
   {\mathfrak n}^{s}(\beta_{l+1}) - \beta_{l+1}
   \geq \nu_{\Gamma, s} (\Omega_{k}^{s})  .
 \end{equation}
 Thus any $1$-form in $(f_{\Gamma, l+2}) dx + (f_{\Gamma,l+2}) dy$
 has $\Gamma'$-order greater than
 $\nu_{\Gamma, s} ({\Omega}_{k}^{s})$ for any
 $\Gamma' \in \fgba{\Gamma}{{\mathfrak n}^s(\beta_{l+1})}$.
 
 Now, we want to understand the behavior of
 $\nu_{\Gamma, s} (\sum_{b=1}^{2 \nu_{l+1}} h_{b}(x) \Omega_{b}^{s})$
 for $h_1, \ldots,h_{2 \nu_{l+1}} \in {\mathbb C}[[x]]$. We have
 \begin{equation*}
   \{ [\nu_{\Gamma}(\Omega_{b}^{s})]_{n} :
   b \in \low_{\Gamma, l+1}^{s} \} = 
   \{ [\nu_{\Gamma, s}(\Omega_{d}^{s})]_{n} :
   d \in \upper_{\Gamma, l+1}^{s} \} = [(e_{l+1})]_{n}
 \end{equation*} 
 by Condition \ref{good7} of Definition \ref{def:good} applied to
 $\Gamma' \in \fg{\Gamma, {\mathfrak g}}^{s}$, and Lemma
 \ref{lem:lot}.  So, for any $[r]_{n} \in [(e_{l+1})]_{n}$
 there exist a single $j' \in \low_{\Gamma, l+1}^{s}$ and a
 $k' \in \upper_{\Gamma, l+1}^{s}$ such that
 $[\nu_{\Gamma}(\Omega_{j'}^{s})]_{n} =
 [\nu_{\Gamma, s}(\Omega_{k'}^{s})]_{n}=[r]_{n}$.  Moreover, given
 $h_{j'}, h_{k'} \in {\mathbb C}[[x]]$, we have
 \begin{equation}
   \label{equ:aux_c2}
   \nu_{\Gamma, s} (h_{j'}(x) \Omega_{j'}^{s} + h_{k'}(x) \Omega_{k'}^{s})
   = \min 
   (\nu_{\Gamma, s} (h_{j'}(x) \Omega_{j'}^{s}),
   \nu_{\Gamma, s} (h_{k'}(x) \Omega_{k'}^{s})) 
 \end{equation}
 since a non-vanishing constant and a non-constant linear polynomial
 are linearly independent (cf. Lemma \ref{lem:lot}).
 Denote by 
 $\nu_{[r]_{n}}$ any of the sides of \eqref{equ:aux_c2}. Obviously, $\nu_{[r]_{n}}\in[r]_{n}$ if it is not $\infty$. We obtain
 \begin{equation}
   \label{equ:aux_c3}
  \nu_{\Gamma, s} 
 \left(\sum_{b=1}^{2 \nu_{l+1}} h_{b}(x) \Omega_{b}^{s} \right) 
 =\min_{[r]_{n} \in [(e_{l+1})]_{n}} \nu_{[r]_{n}} = 
 \min_{1 \leq b \leq 2 \nu_{l+1}}
   \nu_{\Gamma, s} (h_{b}(x) \Omega_{b}^{s}) .
\end{equation} 

 Denote
 $\tilde{\Omega}_{k}^{s} = \phi^{*} (\overline{\Omega}_{k}^{s})$.  By
 construction, we have
 \begin{equation*}
   \nu_{\Gamma, s} (\tilde{\Omega}_{k}^{s}) =
   \nu_{\Gamma, s} (\Omega_{k}^{s})    < \nu_{\Gamma} (\Omega_{j}^{s})
 \end{equation*}
 and the coefficient of the lowest order of
 $t ({\Gamma}')^{*} (\tilde{\Omega}_{k}^{s})$ is a non-constant linear
 polynomial on $a_{{\mathfrak n}^{s} (\beta_{l+1}), \Gamma'}$ for
 $\Gamma' \in \fgba{\Gamma}{{\mathfrak n}^s(\beta_{l+1})}$ (cf. Lemma \ref{lem:lot}
 again). We have
 \begin{equation*}
   \tilde{\Omega}_{k}^{s} =
   \sum_{b=1}^{2 \nu_{l+1}} h_{b}(x) \Omega_{b}^{s} +
   f_{\Gamma, l+2} A(x,y) dx    
   + f_{\Gamma, l+2} B(x,y) dy
 \end{equation*}
 for some $h_1, \ldots, h_b \in {\mathbb C}[[x]]$ and
 $A, B \in {\mathbb C}[[x,y]]$ by equation
 \eqref{equ:m+f}.
 Since 
 \[ \nu_{\Gamma, s} (\tilde{\Omega}_{k}^{s}) < 
 \min ( \nu_{\Gamma, s} (f_{\Gamma, l+2} A(x,y) dx),
 \nu_{\Gamma, s} (f_{\Gamma, l+2} B(x,y) dy)) \]
 by equation \eqref{equ:aux_c1}, it follows that 
 \begin{equation*}
   \nu_{\Gamma, s} (\tilde{\Omega}_{k}^{s}) = 
   \min_{1 \leq b \leq 2 \nu_{l+1}}
   \nu_{\Gamma, s} (h_{b}(x) \Omega_{b}^{s})
 \end{equation*}
 by equation \eqref{equ:aux_c3}. As
 $\nu_{\Gamma, s} (\tilde{\Omega}_{k}^{s}) < \nu_{\Gamma}
 (\Omega_{j}^{s})$, we obtain
 \begin{equation*}
   \nu_{\Gamma, s} ( \tilde{\Omega}_{k}^{s} - h_{k} (0) {\Omega}_{k}^{s})
   > \nu_{\Gamma, s} ( {\Omega}_{k}^{s}) .
 \end{equation*}
 In particular, non-vanishing terms of the lowest order of
 $t ({\Gamma}')^{*} (\Omega_{k}^{s})$ and
 $t ({\Gamma}')^{*} (\tilde{\Omega}_{k}^{s})$ are equal up to a
 non-vanishing constant (i.e. $h_k(0)$) for
 $\Gamma' \in \fgba{\Gamma}{{\mathfrak n}^s(\beta_{l+1})}$. Thus, we get
 \begin{equation*}
   a_{{\mathfrak n}^{s}(\beta_{l+1}), {\Gamma}'} \in {\mathcal
     D}_{\Gamma} ({\mathfrak n}^{s}(\beta_{l+1}), k) \Leftrightarrow
   a_{{\mathfrak n}^{s}(\beta_{l+1}), \phi({\Gamma}')} \in {\mathcal
     D}_{\overline{\Gamma}} ({\mathfrak n}^{s}(\beta_{l+1}), k)
 \end{equation*}
 for any
 ${\Gamma}' \in \fgba{\Gamma}{{\mathfrak n}^s(\beta_{l+1})}$.
 Let us remark that
 \begin{equation*}
   k \in \upper_{\Gamma, l+1}^{s+1}
   \Leftrightarrow
   a_{{\mathfrak n}^{s}(\beta_{l+1}), {\Gamma}} \in
   {\mathcal D}_{\Gamma} ({\mathfrak n}^{s}(\beta_{l+1}), k)
 \end{equation*}
 by Remarks \ref{rem:sd_l} and \ref{rem:disc_u} and the analogous
 property holds when we replace $\Gamma$ with
 $\overline{\Gamma}$. Therefore, $k$ belongs to
 $\upper_{\Gamma, l+1}^{s+1}$ if and only if it belongs to
 $\upper_{\overline{\Gamma}, l+1}^{s+1}$ and thus property
 \eqref{conj1} is satisfied. Notice that
 $\Omega_{k}^{s+1} = \Omega_{k}^{s}$ and
 $\overline{\Omega}_{k}^{s+1} = \overline{\Omega}_{k}^{s}$.

  First, if $k \in \upper_{\Gamma, l+1}^{s+1}$, then Property
 \eqref{conj2} clearly holds.  Moreover, we have
 \begin{equation*}
   \nu_{\Gamma,s+1} (\Omega_{k}^{s+1}) = \nu_{\Gamma,s}
   (\Omega_{k}^{s}) + ({\mathfrak n}^{s+1} (\beta_{l+1}) - {\mathfrak
     n}^{s} (\beta_{l+1}))
 \end{equation*}
 and the analogous property for $\overline{\Gamma}$ and
 $\overline{\Omega}_{k}^{s}$ holds.  Hence, we obtain property
 \eqref{conj3}.
 
 If otherwise $k \in \low_{\Gamma, l+1}^{s+1}$, then we have
 \begin{equation*}
   \nu_{\Gamma} (\Omega_{k}^{s+1}) =
   \nu_{\Gamma, s} (\Omega_{k}^{s}), \ \ 
   \nu_{\Gamma, s+1} (\Omega_{j}^{s+1}) =  \nu_{\Gamma} (\Omega_{j}^{s}) +
   ({\mathfrak n}^{s+1} (\beta_{l+1}) - {\mathfrak n}^{s} (\beta_{l+1}))
 \end{equation*}
 and analogue properties for $\overline{\Gamma}$,
 $\overline{\Omega}_{k}^{s}$, $\overline{\Omega}_{k}^{s+1}$
 $\overline{\Omega}_{j}^{s}$, $\overline{\Omega}_{j}^{s+1}$ hold. We
 deduce that properties \eqref{conj1} , \eqref{conj2} and
 \eqref{conj3} are satisfied. This concludes the second case:
 $ \nu_{\Gamma, s} (\Omega_{k}^{s}) < \nu_{\Gamma}
   (\Omega_{j}^{s})$.

   We now finish the proof. Suppose $l+1 < g$ and let
   $s_{l+1} \in {\mathbb Z}_{\geq 1}$ with
   ${\mathfrak n}^{s_{l+1}}(\beta_{l+1}) = \beta_{l+2}$.  We have
   $\nu_{\Gamma} (\Omega_{j}^{s_{l+1}}) = \nu_{\overline{\Gamma}}
   (\overline{\Omega}_{j}^{s_{l+1}})$ for any
   $j \in \low_{\Gamma, l+1}^{s_{l+1}}$.  Since
   $\Beg_{\Gamma}(\Omega_{k}^{s_{l+1}}) = \beta_{l+2}$ (for 
   $\fgb{\Gamma}{\beta_{l+2}}$) and
   $\Beg_{\overline{\Gamma}}(\overline{\Omega}_{k}^{s_{l+1}}) =
   \beta_{l+2}$ (for $\fgb{\overline{\Gamma}}{\beta_{l+2}}$), we
   deduce
   \begin{equation*}
     \nu_{\Gamma} (\Omega_{k}^{s_{l+1}}) =
     \nu_{\Gamma, s_{l+1}} (\Omega_{k}^{s_{l+1}}) = 
     \nu_{\overline{\Gamma}, s_{l+1}} (\overline{\Omega}_{k}^{s_{l+1}}) =
     \nu_{\overline{\Gamma}} (\overline{\Omega}_{k}^{s_{l+1}})
   \end{equation*}
   for any $k \in \upper_{\Gamma, l+1}^{s_{l+1}}$.  We obtain
   $\nu_{\Gamma} (\Omega_{j}^{s_{l+1}}) = \nu_{\overline{\Gamma}}
   (\overline{\Omega}_{j}^{s_{l+1}})$ for any
   $j \in {\mathbb N}_{2 \nu_{l+1}}$.

   We need to resolve the assumption we made on
   ${\mathcal T}_{\Gamma, l}$ and ${\mathcal T}_{\overline{\Gamma}, l}$ at the
   beginning: that
   $\nu_{\Gamma} (\Omega_j) = \nu_{\overline{\Gamma}} (\overline{\Omega}_j)$ for any
   $j \in {\mathbb N}_{2 \nu_{l}}$.  The previous paragraph implies that if the
   assumption holds for $0 \leq l < g-1$ then so does it for $l+1$. Since it clearly
   holds for $l=0$ where $\mathcal{T}_{\Gamma,0}=(dx,dy)$), we are
   done.
\end{proof}

\bibliography{sb.bib}
 \end{document}